\documentclass[11pt]{amsart}

\title[Transport inequalities]{Transport Inequalities. A Survey}
\author{Nathael Gozlan, Christian L\'eonard}
    \address{Laboratoire d'analyse et math\'ematiques appliqu\'ees, UMR CNRS 8050. Universit\'e Paris Est. 5 bd Descartes, 77454 Marne la Vall\'ee, France}
    \email{nathael.gozlan@univ-mlv.fr}
     \address{Modal-X. Universit\'e Paris Ouest. B\^at.\! G, 200 av. de la R\'epublique. 92001 Nanterre, France}
 \email{christian.leonard@u-paris10.fr}

\usepackage{amssymb,amsmath}
\usepackage{amsthm, enumerate}
\usepackage{qsymbols}
\usepackage[usenames,dvipsnames]{pstricks}
\usepackage{epsfig}

\newtheorem{thm}{Theorem}[section]
\newtheorem{lem}[thm]{Lemma}
\newtheorem{prop}[thm]{Proposition}
\newtheorem{cor}[thm]{Corollary}
\newtheorem{defi}[thm]{Definition}

\theoremstyle{remark}
\newtheorem{rem}[thm]{Remark}

\newtheorem{expl}[thm]{Example}

\newcommand{\D}{\displaystyle}
\newcommand{\X}{\mathcal{X}}
\newcommand{\Y}{\mathcal{Y}}

\newcommand{\PX}{\mathrm{P}(\X)}

\newcommand{\PpX}{\mathrm{P}_p(\X)}
\newcommand{\PXI}{\mathrm{P}(\X^\II)}

\newcommand{\R}{\mathbb{R}}

\newcommand{\Rk}{{\mathbb{R}^k}}

\def\AArm{\fam0 \rm}%
\newdimen\AAdi%
\newbox\AAbo%
\def\AAk#1#2{\setbox\AAbo=\hbox{#2}\AAdi=\wd\AAbo\kern#1\AAdi{}}%
\newcommand{\1}{{\ensuremath{{\AArm 1\AAk{-.8}{I}I}}}}

\renewcommand{\P}{\mathbb{P}}
\newcommand{\E}{\mathbb{E}}

\newcommand{\Var}{\operatorname{Var}}

\renewcommand{\epsilon}{\varepsilon}

\newcommand{\T}{\mathbf{T}}
\newcommand{\LSI}{\mathbf{LS}}
\newcommand{\rLSI}{\mathbf{rLS}}
\newcommand{\PI}{\mathbf{P}}
\newcommand{\WI}{\mathbf{W}_1\mathbf{I}}
\newcommand{\WWI}{\mathbf{W}_2\mathbf{I}}

\newcommand{\SG}{\mathbf{P}}

\newcommand{\prob}{probability measure }

\newcommand{\PXX}{\mathrm{P}(\X^2)}
\newcommand{\CX}{\mathcal{C}_b(\X)}
\newcommand{\BX}{\mathcal{B}_b(\X)}
\newcommand\scal{\!\cdot\!}
\newcommand{\mc}{\circledast}
\newcommand{\lsc}{lower semicontinuous}
\newcommand{\PU}{\mathrm{P}_\mathcal{U}}

\newcommand{\Pu}{\mathrm{P}_1}
\newcommand{\Pd}{\mathrm{P}_2}
\newcommand{\Pp}{\mathrm{P}_p}
\newcommand{\Pac}{\mathrm{P}^{\mathrm{ac}}}
\newcommand{\IX}{\int_\X}
\newcommand{\IXX}{\int_{\X^2}}
\newcommand{\IF}{I_F}
\newcommand{\IDV}{I}
\newcommand{\II}{\mathrm{I}}

\newcommand{\ii}[2]{#1_i^{\widetilde{#2}_i}}
\newcommand{\dd}{\mathbb{D}}
\def\LL{\mathcal L}

\newcommand\logsup[1]{\limsup_{#1\rightarrow\infty}\frac{1}{#1}\log}

\def\<{\langle}
\def\>{\rangle}
\newcommand{\boulette}[1]{$\bullet$\ Proof of #1.}
\newcommand{\Boulette}[1]{\par\medskip\noindent $\bullet$\ Proof of #1.}

\begin{document}


\begin{abstract}
This is a survey of recent developments in the area of transport
inequalities. We investigate their consequences in terms of
concentration and deviation inequalities and sketch their links
with other functional inequalities and also large deviation
theory.
\end{abstract}

\date{\today}
\maketitle

\section*{Introduction}

In the whole paper,  $\X$ is a polish (complete metric and
separable) space equipped with its Borel $\sigma$-field  and we
denote $\PX$ the set of all Borel probability measures on $\X.$

Transport inequalities  relate a cost $\mathcal{T}(\nu,\mu)$ of
transporting a generic probability measure $\nu\in\PX$ onto a
reference probability measure $\mu\in\PX$ with another functional
$J(\nu|\mu).$ A typical transport inequality is written:
$$\alpha(\mathcal{T}(\nu,\mu))\le J(\nu|\mu),\quad\textrm{ for all
}\nu\in\PX,$$ where $\alpha:[0,\infty)\to[0,\infty)$ is an
increasing function with $\alpha(0)=0$. In this case, it is said
that the reference probability measure $\mu$ satisfies
$\alpha(\mathcal{T})\le J.$

Typical transport inequalities are built with $\mathcal{T}=W^p_p$
where $W_p$ is the Wasserstein metric of order $p,$ and
$J(\cdot|\mu)=H(\cdot|\mu)$ is the relative entropy with respect
to $\mu.$ The left-hand side of $$\alpha(W_p^p)\le H$$ contains
$W$ which is built with some metric $d$ on $\X,$ while its
right-hand side is the relative entropy $H$ which, as Sanov's
theorem indicates, is a measurement of the difficulty for a large
sample of independent particles with common law $\mu$ to deviate
from the prediction of the law of large numbers. On the left-hand
side: a cost  for displacing  mass in terms of the ambient metric
$d$; on the right-hand side: a cost for displacing mass in terms
of fluctuations. Therefore, it is not  a surprise that this
interplay between displacement and fluctuations gives rise to a
quantification of how fast  $\mu(A^r)$ tends to 1 as  $r\ge 0$
increases, where $A^r:=\{x\in\X; d(x,y)\le r\textrm{ for some }
y\in A\}$ is the enlargement of size $r$  with respect to the
metric $d$ of the subset $A\subset\X.$ Indeed, we shall see that
such transport-entropy inequalities are intimately related to the
concentration of measure phenomenon and to deviation inequalities
for average observables of  samples.

Other transport inequalities are built with the Fisher information
$I(\cdot|\mu)$ instead of the relative entropy on the right-hand
side. It is known since Donsker and Varadhan, see \cite{DV75a,DS},
that $I$ is a measurement of the fluctuations of the occupation
measure of a very long trajectory of a time-continuous Markov
process with invariant ergodic law $\mu.$ Again, the
transport-information inequality $\alpha(W_p^p)\le I$ allows to
quantify concentration and deviation properties of $\mu.$

Finally, there exist also free transport inequalities. They
compare a transport cost with a free relative entropy which is the
large deviation rate function of the spectral empirical  measures
of large random matrices, as was proved by Ben Arous and Guionnet
\cite{BAG97}.

This is a survey paper about transport inequalities: a research
topic which flied off in 1996 with the publications of several
papers on the subject by Dembo,  Marton, Talagrand and Zeitouni
\cite{D96, DZ96,M96a,M96b,T96}. It was known from the end of the
sixties that the total variation norm of the difference of two
probability measures is controlled by their relative entropy. This
is expressed by the Csisz\'ar-Kullback-Pinsker inequality
\cite{Pin64, Csi67,Kul67} which is a transport inequality from
which deviation inequalities have been derived. But the keystone
of the edifice was the discovery in 1986 by Marton  \cite{M86} of
the link between transport inequalities and the concentration of
measure. This result was motivated by information theoretic
problems; it remained unknown to the analysts and probabilists
during ten years. Meanwhile, during the second part of the
nineties, important progresses about the understanding of optimal
transport have been achieved, opening the way to new unified
proofs of several related functional inequalities, including a
certain class of transport inequalities.

Concentration of measure inequalities can be obtained by means of
other functional inequalities such as isoperimetric and
logarithmic Sobolev inequalities, see the textbook by Ledoux
\cite{Led} for an excellent account on the subject. Consequently,
one expects that there are deep connections between these various
inequalities. Indeed, during the recent years, these links have
been explored and some of them have been clarified.

These recent developments will be sketched in the following pages.

No doubt that our treatment of this vast subject fails to be
exhaustive. We apologize in advance for all kind of omissions. All
comments, suggestions and reports of omissions  are welcome.

 \tableofcontents

\thanks{The authors are grateful to
the organizers of the conference on Inhomogeneous Random Systems
for this opportunity to present transport inequalities to a broad
community of physicists and mathematicians.}

\section{An overview}\label{sec-overview}

In order to present as soon as possible a couple of important
transport inequalities and their consequences in terms of
concentration of measure and deviation inequalities, let us recall
precise definitions of the optimal transport cost and the relative
entropy.

\subsection*{Optimal transport cost} Let   $c$ be a $[0,\infty)$-valued \lsc\ function on
the polish product space $\X^2$ and fix $\mu,\nu\in\PX.$ The
Monge-Kantorovich optimal transport problem is
\begin{equation}\label{MK}\tag{MK}
    \textsl{Minimize }\pi\in\PXX\mapsto\IXX c(x,y)\,d\pi(x,y)\in [0,\infty]\quad  \textsl{ subject to
    } \pi_0=\nu,\pi_1=\mu
\end{equation}
where $\pi_0,\pi_1\in P(\X)$ are the first and second marginals of
$\pi\in \PXX.$ Any $\pi\in\PXX$ such that $ \pi_0=\nu$ and
$\pi_1=\mu$ is called a \emph{coupling} of $\nu$ and $\mu.$ The
value of this convex minimization problem is
\begin{equation}\label{eq-14}
\mathcal{T}_c(\nu,\mu):=\inf\left\{ \IXX
c(x,y)\,d\pi(x,y);\pi\in\PXX; \pi_0=\nu,\pi_1=\mu\right\}\in
[0,\infty].
\end{equation}
It is called the \emph{optimal cost} for transporting $\nu$ onto
$\mu.$ Under the natural assumption that $c(x,x)=0,$ for all
$x\in\X,$ we have: $\mathcal{T}_c(\mu,\mu)=0,$ and
$\mathcal{T}_c(\nu,\mu)$ can be interpreted as a cost for coupling
$\nu$ and $\mu.$
\\
A popular cost function is $c=d^p$ with $d$ a  metric on $\X$ and
$p\ge1.$ One can prove that under some conditions
$$
W_p(\nu,\mu):=\mathcal{T}_{d^p}(\nu,\mu)^{1/p}
$$
defines a metric on a subset of $\PX$. This is the
\emph{Wasserstein metric} of order $p$ (see e.g \cite[Chp 6]{Vill2}).
A deeper investigation of optimal transport is presented at
Section \ref{section optimal transport}. It will be necessary for
a better understanding of transport inequalities.

\subsection*{Relative entropy}
The relative entropy with respect to $\mu\in\PX$ is defined by
\begin{equation*}
    H(\nu|\mu)=\left\{%
\begin{array}{ll}
    \int_\X\log\left(\frac{d\nu}{d\mu}\right)\,d\nu & \hbox{if }\nu\ll\mu \\
    +\infty & \hbox{otherwise} \\
\end{array}%
\right.,\quad \nu\in\PX.
\end{equation*}
For any probability measures $\nu\ll\mu,$ one can rewrite
$H(\nu|\mu)=\int h(d\nu/d\mu)\,d\mu$ with $h(t)=t\log t-t+1$ which
is a strictly convex nonnegative function such that
$h(t)=0\Leftrightarrow t=1.$

\begin{center}
\scalebox{1} 
{
\begin{pspicture}(0,-1.7429688)(6.955,1.7129687)
\psline[linewidth=0.02cm,arrowsize=0.05291667cm
2.0,arrowlength=1.4,arrowinset=0.4]{->}(2.5221875,-1.6770313)(2.5221875,1.7029687)
\psline[linewidth=0.02cm,arrowsize=0.05291667cm
2.0,arrowlength=1.4,arrowinset=0.4]{->}(0.9421875,-1.2970313)(6.9021873,-1.3170313)
\psbezier[linewidth=0.04,dotsize=0.07055555cm
2.0]{*-}(2.5421875,-0.49703124)(2.5221875,-1.0770313)(2.8614914,-1.281493)(3.1621876,-1.2770313)(3.4628835,-1.2725695)(5.3821874,-0.59703124)(6.7221875,0.60296875)
\psline[linewidth=0.04cm,tbarsize=0.07055555cm
5.0,rbracketlength=0.15]{)-}(2.5221875,1.3829688)(1.0821875,1.3829688)
\usefont{T1}{ptm}{m}{n}
\rput(2.2696874,-1.5770313){\small $0$}
\usefont{T1}{ptm}{m}{n}
\rput(3.1696875,-1.5570313){\small $1$}
\usefont{T1}{ptm}{m}{n}
\rput(6.6896877,-1.5570313){\small $t$}
\usefont{T1}{ptm}{m}{n}
\rput(0.5096875,1.3829688){\small $\infty$}
\usefont{T1}{ptm}{m}{n}
\rput(2.1296875,-0.49703124){\small $1$}
\psdots[dotsize=0.12,dotstyle=|](3.1221876,-1.2770313)
\end{pspicture}
}
\par\medskip
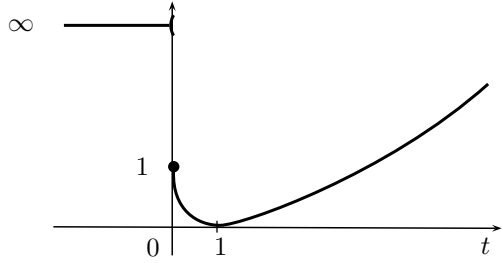
\textit{Graphic representation of $h(t)=t\log t-t+1.$}
\end{center}
    \par\medskip
Hence, $\nu\mapsto H(\nu|\mu)\in[0,\infty]$ is a convex function
and $ H(\nu|\mu)=0$ if and only if $\nu=\mu.$

\subsection*{Transport inequalities}
We can now define a general class of inequalities involving transport costs.
\begin{defi}[Transport inequalities]
Besides the cost function $c,$ consider also two functions
$J(\,\cdot\,|\mu):\PX\to [0,\infty]$ and
$\alpha:[0,\infty)\to[0,\infty)$ an increasing function such that
$\alpha(0)=0.$ One says that $\mu\in\PX$ satisfies the
\emph{transport inequality $\alpha(\mathcal{T}_c)\leq J$} if
\begin{equation}\label{eqL-b}\tag{$\alpha(\mathcal{T}_c)\leq J$}
    \alpha(\mathcal{T}_c(\nu,\mu))\le
    J(\nu|\mu),\quad \textrm{for all }\nu\in \PX.
    \end{equation}
When $J(\,\cdot\,)=H(\,\cdot\,|\mu)$, one talks about \emph{transport-entropy inequalities}.
\end{defi}
For the moment, we focus on transport-entropy inequalities, but in Section \ref{section transport-information}, we shall encounter the class of transport-information inequalities, where the
functional $J$ is the Fisher information.

Note that, because of  $ H(\mu|\mu)=0,$ for the transport-entropy
inequality to hold true, it is necessary that
$\alpha(\mathcal{T}_c(\mu,\mu))=0.$ A sufficient condition for the
latter equality is
\begin{itemize}
    \item $c(x,x)=0,$ for all $x\in\X$ and
    \item $\alpha(0)=0$.
\end{itemize}
This will always be assumed in the remainder of this article.

Among this general family of inequalities, let us isolate the
classical $\T_1$ and $\T_2$ inequalities. For $p=1$ or $p=2$, one
says that $\mu\in\Pp:=\{\nu\in\PX;\int
d(x_o,\cdot)^p\,d\nu<\infty\}$ satisfies the inequality $\T_p(C)$,
with $C>0$ if
\begin{equation*}
\tag{$\T_p(C)$}
    W_p^2(\nu,\mu)\le C H(\nu|\mu),
\end{equation*}
for all $\nu\in\PX.$
\begin{rem}\label{rem-a}
Note that this inequality implies that $\mu$ is such that
$H(\nu|\mu)=\infty$ whenever $\nu\not\in\Pp.$
\end{rem}
With the previous notation, $\T_1(C)$ stands for the inequality
$\displaystyle C^{-1}\mathcal{T}_d^2\leq H$ and $\T_2(C)$ for the
inequality $\displaystyle C^{-1}\mathcal{T}_{d^2}\leq H$. Applying
Jensen inequality, we get immediately that
\begin{equation}\label{eq-18}
    \mathcal{T}_d^2(\nu,\mu)\leq \mathcal{T}_{d^2}(\nu,\mu).
\end{equation}
As a consequence, for a given metric $d$ on $\X$, the inequality
$\T_1$ is always weaker than the inequality $\T_2$.

We now present two important examples of transport-entropy
inequalities: the Csisz\'ar-Kullback-Pinsker inequality, which is
a $\T_1$ inequality and Talagrand's $\T_2$ inequality for the
Gaussian measure.

\subsection*{Csisz\'ar-Kullback-Pinsker inequality}
The total variation distance between two probability measures $\nu$ and $\mu$ on $\X$ is defined by
$$\|\nu-\mu\|_{TV}=\sup|\nu(A)-\mu(A)|,$$
where the supremum runs over all measurable $A\subset \X.$ It
appears that the total variation distance is an optimal
transport-cost. Namely, consider the so-called Hamming metric
$$d_H(x,y)=\mathbf{1}_{x\not=y},\quad x,y\in \X,$$ which assigns the
value $1$ if $x$ is different from $y$ and the value $0$
otherwise.
Then we have the following result whose proof can be found in e.g
\cite[Lemma 2.20]{Mas07}.
\begin{prop}
For all $\nu,\mu \in \PX$,
$\mathcal{T}_{d_H}(\nu,\mu)=\|\nu-\mu\|_{TV}.$
\end{prop}

The following theorem gives the celebrated Csisz\'ar-Kullback-Pinsker inequality (see \cite{Pin64,
Csi67,Kul67}).
\begin{thm}
The inequality
$$\|\nu-\mu\|_{TV}^2\leq \frac{1}{2}H(\nu|\mu),$$
holds for all probability measures $\mu,\nu$ on $\X.$
\end{thm}
In other words, any probability $\mu$ on $\X$ enjoy the inequality $\T_1(1/2)$ with respect to the Hamming distance $d_H$ on $\X.$
\proof The following proof is taken from \cite[Remark 22.12]{Vill2} and is attributed to Talagrand.
Suppose that $H(\nu | \mu)<+\infty$ (otherwise there is nothing to prove) and let $f=\frac{d\nu}{d\mu}$ and $u=f-1$. By definition and since $\int u\,d\mu=0$, $$H(\nu|\mu)=\int_\X f\log f\,d\mu =\int_\X (1+u)\log(1+u)-u\,d\mu.$$
The function $\varphi(t)=(1+t)\log(1+t)-t$, verifies $\varphi'(t)=\log(1+t)$ and $\varphi''(t)=\frac{1}{1+t}$, $t>-1$. So, using a Taylor expansion,
$$\varphi(t)=\int_0^t (t-x)\varphi''(x)\,dx=t^2\int_0^1\frac{1-s}{1+st}\,ds,\quad t>-1.$$
So,
$$H(\nu | \mu)=\int_{\X\times[0,1]} \frac{u^2(x)(1-s)}{1+su(x)}\,ds\,d\mu(x).$$
According to Cauchy-Schwarz inequality,
\begin{align*}
    \Big(\int_{\X\times[0,1]} |u|(x)& (1-s)\,d\mu(x)ds\Big)^2\\
    &\leq \int_{\X\times[0,1]} \frac{u(x)^2(1-s)}{1+su(x)}\,d\mu(x)ds\cdot \int_{\X\times[0,1]} (1-s)(1+su(x))\,d\mu(x)ds\\
    &= \frac{H(\nu|\mu)}{2}.
\end{align*}
Since $\|\nu-\mu\|_{TV}=\frac{1}{2}\int |1-f|\,d\mu$, the
left-hand side equals $\|\nu-\mu\|_{TV}^2$ and this completes the
proof.
\endproof

\subsection*{Talagrand's transport inequality for the Gaussian measure}

In \cite{T96}, Talagrand proved the following transport inequality $\T_2$
for the standard Gaussian measure $\gamma$ on $\R$ equipped with the standard distance $d(x,y)=|x-y|$.
\begin{thm}
The standard Gaussian measure $\gamma$ on $\R$ verifies
\begin{equation}\label{eq-11}
    W_2^2(\nu,\gamma)\le 2 H(\nu|\gamma),
\end{equation}
for all $\nu\in \mathrm{P}(\R).$
\end{thm}
This inequality is sharp. Indeed, taking $\nu$ to be a translation
of $\gamma,$ that is a normal law with unit variance, we easily
check that equality holds true.

The following notation will appear frequently in the sequel: if $T:\X\to \X$ is a measurable map, and $\mu$ is a probability measure on $\X$, the \emph{image of $\mu$ under $T$} is the probability measure denoted by $T_\#\mu$ and defined by
\begin{equation}
T_\#\mu(A)=\mu\left( T^{-1}(A)\right), 
\end{equation}
for all Borel set $A\subset \X$.
\proof
In the following lines, we present the short and elegant proof of
\eqref{eq-11}, as it appeared in \cite{T96}.
Let us  consider a reference measure
$$
d\mu(x)=e^{-V(x)}\,dx.
$$
We shall specify later to the Gaussian case, where the potential
$V$ is given by $V(x)=x^2/2+\log(2\pi)/2,$ $x\in\R.$ Let $\nu$ be
another probability measure on $\R.$ It is known since Fr\'echet
that any measurable map $y=T(x)$ which verifies the equation
\begin{equation}\label{eq-09}
    \nu((-\infty, T(x)])=\mu((-\infty,x]),\quad x\in\R
\end{equation}
is a coupling of $\nu$ and $\mu,$ i.e.\ such that $\nu=T_\#\mu,$
which minimizes the average squared distance (or equivalently:
which maximizes the correlation), see \eqref{eq-08} below for a
proof of this statement. Such a transport map is called a monotone
rearrangement. Clearly $T$ is increasing, and assuming from now on
that $\nu=f\mu$ is absolutely continuous with respect to $\mu,$
one sees that $T$ is Lebesgue almost everywhere differentiable
with $T'>0.$ Equation \eqref{eq-09} becomes for all real $x,$
$\int_{-\infty}^{T(x)}f(z)e^{-V(z)}\,dz=\int_{-\infty}^{x}e^{-V(z)}\,dz.$
Differentiating, one obtains
\begin{equation}\label{eq-10}
    T'(x)f(T(x))e^{-V(T(x))}=e^{-V(x)},\quad x\in\R.
\end{equation}
The relative entropy writes: $H(\nu|\mu)=\int
\log(f)\,d\nu=\int\log(f(T(x))\,d\mu$ since $\nu=T_\#\mu.$
Extracting $f(T(x))$ from \eqref{eq-10} and plugging it into this
identity, we obtain
$$
H(\nu|\mu)=\int [V(T(x))-V(x)-\log T'(x)]\,e^{-V(x)}\,dx.
$$
On the other hand, we have $\int (T(x)-x)V'(x)e^{-V(x)}\,dx=\int
(T'(x)-1)e^{-V(x)}\,dx$ as a result of an integration by parts.
Therefore,
\begin{equation}\label{eq-12}
\begin{split}
    H(\nu|\mu)&=\int \Big(V(T(x))-V(x)-V'(x)[T(x)-x]\Big)\,d\mu(x)\\
    &\hskip 3cm +\int (T'(x)-1-\log T'(x))\, d\mu(x).\\
    &\ge \int \Big(V(T(x))-V(x)-V'(x)[T(x)-x]\Big)\,d\mu(x)
\end{split}
\end{equation}
where we took advantage of $b-1-\log b\ge0$ for all $b>0,$ at the
last inequality. Of course,  the last integral  is nonnegative if
$V$ is assumed to be convex.
\\
Considering the Gaussian potential $V(x)=x^2/2+\log(2\pi)/2,$
$x\in\R,$ we have shown that
$$
H(\nu|\gamma)\ge\int_\R (T(x)-x)^2/2\ d\gamma(x)\ge
W_2^2(\nu,\gamma)/2
$$
for all $\nu\in\mathrm{P}(\R),$ which is \eqref{eq-11}.
\endproof

\subsection*{Concentration of measure}

If $d$ is a metric on $\X,$ for any $r\ge0,$ one defines the
$r$-neighborhood of the set  $A\subset\X$  by
\[A^r:=\{x\in \X; d(x,A)\leq r\},\quad r\geq 0,\]
where $d(x,A):=\inf_{y\in A}d(x,y)$ is the distance of $x$ from
$A.$

Let $\beta:[0,\infty)\to \R^+$ such that $\beta(r)\to 0$ when $r\to \infty$; it is said that the probability measure $\mu$ verifies the
\emph{concentration inequality} with profile $\beta$ if
$$\mu(A^r)\geq 1-\beta(r),\quad r\geq 0,$$
for all measurable $A\subset \X$, with $\mu(A)\geq 1/2.$

According to the following classical proposition, the
concentration of measure (with respect to metric enlargement) can
be alternatively described in terms of deviations of Lipschitz
functions from their median.

\begin{prop}\label{resL-14}
Let $(\X,d)$ be a metric space, $\mu\in \PX$ and
$\beta:[0,\infty)\to [0,1]$; the following propositions are
equivalent
\begin{enumerate}
\item The probability $\mu$ verifies the concentration inequality
$$\mu(A^r)\geq 1-\beta(r),\quad r\geq 0,$$
for all $A\subset \X$, with $\mu(A)\geq 1/2.$ \item For all
$1$-Lipschitz function $f:\X\to \R$,
$$\mu(f> m_{f}+r)\leq \beta(r),\quad r\geq 0,$$
where $m_{f}$ denotes a median of $f$.
\end{enumerate}
\end{prop}
\proof $(1)\Rightarrow (2)$. Let $f$ be a $1$-Lipschitz function
and define $A=\{f\leq m_{f}\}$. Then it is easy to check that
$A^r\subset\{ f\leq m_{f}+r\}$. Since $\mu(A)\geq 1/2$, one has
$\mu(f\leq m_{f}+r)\geq \mu(A^r)\geq 1-\beta(r)$, for all $r\geq
0$.

$(2)\Rightarrow (1)$. For all $A\subset \X$, the function $f_{A}:
x\mapsto d(x,A)$ is $1$-Lipschitz. If $\mu(A)\geq 1/2$, then $0$
is  a median of $f_{A}$. Since $A^r=\{f_{A}\leq r\}$, one has
$\mu(A^r)\geq 1-\mu\{f_{A}>r\}\geq 1-\beta(r)$, $r\geq 0.$
\endproof
Applying the deviation inequality to $\pm f$, we arrive at
$$\mu(|f-m_{f}|< r)\leq 2\beta(r),\quad r\geq 0.$$
In other words, Lipschitz functions are, with a high probability,
concentrated around their median, when the concentration profile
$\beta$ decreases rapidly to zero. In the above proposition, the
median can be replaced by the mean $\mu(f)$ of $f$ (see e.g.\!
\cite{Led}):
\begin{equation}\label{eq-16}
    \mu(f>\mu(f)+ r)\leq \beta(r),\quad r\geq 0.
\end{equation}

The following theorem explains how to derive concentration
inequalities (with profiles decreasing exponentially fast) from transport-entropy inequalities of the form
$\alpha\left(\mathcal{T}_{d}\right)\leq H,$ where the cost
function $c$ is the metric $d.$ The argument used in the proof is
due to Marton \cite{M86} and is referred  as ``Marton's argument''
in the literature.
\begin{thm}\label{Marton-intro}
Let $\alpha:\R^+\to\R^+$ be a bijection and suppose that $\mu\in \PX$ verifies the transport-entropy inequality
$\alpha\left(\mathcal{T}_{d}\right)\leq H$.  Then, for all measurable $A\subset \X$ with $\mu(A)\geq
1/2$, the following concentration inequality holds
$$\mu(A^r)\geq 1-e^{-\alpha(r-r_{o})},\quad  r \geq r_{o}:=\alpha^{-1}(\log 2),$$
where $A^r$ is the enlargement of $A$ for the metric $d$ which is
defined above.

Equivalently, for all $1$-Lipschitz $f:\X\to\R$, the following inequality holds
$$\mu(f>m_f+r+r_o)\leq e^{-\alpha(r)},\quad r\geq 0.$$
\end{thm}
\proof Take $A\subset \X$, with $\mu(A)\geq 1/2$ and set
$B=\X\setminus A^r$. Consider the probability measures
$d\mu_{A}(x)=\frac{1}{\mu(A)}\mathbf{1}_{A}(x)\,d\mu(x)$ and
$d\mu_{B}(x)=\frac{1}{\mu(B)}\mathbf{1}_{B}(x)\,d\mu(x)$. Obviously, if
$x\in A$ and $y\in B$, then $d(x,y)\geq r$. Consequently, if $\pi$
is a coupling between $\mu_{A}$ and $\mu_{B}$, then $\int d(x,y)\,
d\pi(x,y)\geq r$ and so $\mathcal{T}_{d}(\mu_{A},\mu_{B})\geq r$.
Now, using the triangle inequality and the transport-entropy
inequality we get
$$r\leq \mathcal{T}_{d}(\mu_{A},\mu_{B})\leq \mathcal{T}_{d}(\mu_{A},\mu)+\mathcal{T}_{d}(\mu_{B},\mu)\leq \alpha^{-1}\left(H(\mu_{A}|\mu)\right)+\alpha^{-1}\left(H(\mu_{B}|\mu)\right).$$
It is easy to check that $H(\mu_{A}|\mu)=-\log \mu(A)\leq \log 2$
and $H(\mu_{B}|\mu)=-\log(1-\mu(A^r)).$ It follows immediately
that $\mu(A^r)\geq 1-e^{-\alpha(r-r_{o})},$ for all  $r \geq
r_{o}:=\alpha^{-1}(\log 2).$
\endproof

If $\mu$ verifies $\T_2(C),$ by \eqref{eq-18} it also verifies
$\T_1(C)$ and one can apply Theorem \ref{Marton-intro}. Therefore,
it appears that if $\mu$ verifies $\T_1(C)$ or $\T_2(C)$,  then it
concentrates like a Gaussian measure:
$$\mu(A^r)\geq 1-e^{-(r-r_o)^2/C},\quad r\geq r_o=\sqrt{C\log(2)}.$$
At this stage, the difference between $\T_1$ and $\T_2$ is
invisible. It will appear clearly in the next paragraph devoted to
tensorization of transport-entropy inequalities.

\subsection*{Tensorization}

A central question in the field of concentration of measure is to
obtain concentration estimates not only for $\mu$ but for the
entire family $\{\mu^n; n\geq 1\}$ where $\mu^n$ denotes the
product probability measure $\mu\otimes \cdots \otimes \mu$ on
$\X^n$. To exploit transport-entropy inequalities, one has to know
how they tensorize. This will be investigated in details in
Section \ref{section concentration}. Let us give in this
introductory section, an insight on this important question.

It is enough to understand what happens with the product
$\X_1\times\X_2$ of two spaces. Indeed, it will be clear in a
moment that the extension to the product of $n$ spaces will follow
by induction.
\\

Let $\mu_1,$ $\mu_2$ be two probability measures on two polish
spaces $\X_1,$ $\X_2,$ respectively. Consider two cost functions
$c_1(x_1,y_1)$ and
 $c_2(x_2,y_2)$ defined on $\X_1\times\X_1$ and $\X_2\times\X_2$; they
give rise to the optimal transport cost functions
$\mathcal{T}_{c_1}(\nu_1,\mu_1),$ $\nu_1\in\mathrm{P}(\X_1)$ and
$\mathcal{T}_{c_2}(\nu_2,\mu_2),$ $\nu_2\in\mathrm{P}(\X_2).$
\\
On the product space $\X_1\times\X_2,$ we now consider the product
measure $\mu_1\otimes\mu_2$ and the cost function
\[
c_1\oplus
c_2\big((x_1,y_1),(x_2,y_2)\big):=c_1(x_1,y_1)+c_2(x_2,y_2),\quad
x_1,y_1\in\X_1, x_2,y_2\in\X_2
\]
which give rise to the tensorized optimal transport cost function
\[
\mathcal{T}_{c_1\oplus c_2}(\nu,\mu_1\otimes\mu_2), \quad\nu\in
\mathrm{P}(\X_1\times\X_2).
\]
A fundamental example is $\X_1=\X_2=\R^k$ with
$c_1(x,y)=c_2(x,y)=|y-x|_2^2:$ the Euclidean metric on $\R^k$
tensorizes as the squared Euclidean metric on $\R^{2k}.$

For any probability measure  $\nu$ on the product space
 $\X_1\times\X_2,$ let us write the disintegration of $\nu$
(conditional expectation) with respect to the first coordinate as follows:
\begin{equation}\label{eq-13}
    d\nu(x_1,x_2)=d\nu_1(x_1)d\nu_2^{x_1}(x_2).
\end{equation}
As was suggested by Marton \cite{M96a} and Talagrand \cite{T96},
it is possible to prove  the intuitively clear following
assertion:
\begin{equation}\label{eq-41}
     \mathcal{T}_{c_1\oplus c_2}(\nu,\mu_1\otimes\mu_2)
 \leq
 \mathcal{T}_{c_1}(\nu_1,\mu_1)+\int_{\X_1}\mathcal{T}_{c_2}(\nu_2^{x_1},\mu_2)\,d\nu_1(x_1).
\end{equation}
We give a detailed proof of this claim at the Appendix,
Proposition \ref{res-tenstrans}.
\\
On the other hand, it is well-known that the fundamental property
of the logarithm together with the product form of the
disintegration formula \eqref{eq-13} yield the analogous
tensorization property of the relative entropy:
\begin{equation}\label{eq-42}
    H(\nu | \mu_1\otimes\mu_2)=H(\nu_1 | \mu_1)+\int_{\X_1}
H(\nu_2^{x_1} | \mu_2)\,d\nu_1(x_1).
\end{equation}
Recall that the inf-convolution of two functions $\alpha_1$ and
$\alpha_2$ on $[0,\infty)$ is defined by
\[
\alpha_1 \square \alpha_2(t):=\inf\{\alpha_1(t_1)+ \alpha_2(t_2);
t_1, t_2\geq 0: t_1+t_2=t\},\quad t\geq 0.
\]
\begin{prop}\label{res-tensorization}
 Suppose that the transport-entropy inequalities
\begin{align*}
  \alpha_1(\mathcal{T}_{c_1}(\nu_1,\mu_1))
  &\leq H(\nu_1 | \mu_1),\quad \nu_1\in\mathrm{P}(\X_1) \\
 \alpha_2(\mathcal{T}_{c_2}(\nu_2,\mu_2))
  &\leq H(\nu_2 | \mu_2),\quad \nu_2\in\mathrm{P}(\X_2)
\end{align*}
hold with $\alpha_1, \alpha_2:[0,\infty)\to[0,\infty)$ convex
increasing functions. Then, on the product space $\X_1\times\X_2,$
we have
\begin{equation*}
    \alpha_1 \square \alpha_2\big(\mathcal{T}_{c_1\oplus c_2}(\nu,\mu_1\otimes\mu_2)\big)
    \leq H(\nu | \mu_1\otimes\mu_2),
\end{equation*}
for all $\nu\in \mathrm{P}(\X_1\times\X_2).$
\end{prop}

\begin{proof}
For all
$\nu\in\mathrm{P}(\X_1\times\X_2),$
\begin{align*}
 \alpha_1\square\alpha_2 (\mathcal{T}_{c_1\oplus c_2}(\nu,\mu_1\otimes\mu_2))
  &\stackrel{(a)}{\leq}\alpha_1\square\alpha_2\left( \mathcal{T}_{c_1}(\nu_1,\mu_1)+\int_{\X_1}\mathcal{T}_{c_2}(\nu_2^{x_1}\mu_2)\,d\nu_1(x_1)\right) \\
   &\stackrel{(b)}\leq \alpha_1(\mathcal{T}_{c_1}(\nu_1,\mu_1))+\alpha_2\left(\int_{\X_1}\mathcal{T}_{c_2}(\nu_2^{x_1},\mu_2)\,d\nu_1(x_1)\right) \\
    &\stackrel{(c)}\leq \alpha_1(\mathcal{T}_{c_1}(\nu_1,\mu_1))+\int_{\X_1}\alpha_2\big(\mathcal{T}_{c_2}(\nu_2^{x_1},\mu_2)\big)\,d\nu_1(x_1) \\
   &\stackrel{(d)}\leq H(\nu_1 | \mu_1)+\int_{\X_1}H(\nu_2^{x_1} | \mu_2)\,d\nu_1(x_1) \\
   &=  H(\nu | \mu_1\otimes\mu_2).
\end{align*}
Inequality (a) is verified thanks to \eqref{eq-41} since
$\alpha_1\square\alpha_2$ is increasing, (b) follows from the very
definition of the inf-convolution, (c) follows from Jensen
inequality since $\alpha_2$ is convex, (d) follows from the
assumed transport-entropy inequalities and
the last equality is (\ref{eq-42}).
\end{proof}
Obviously, it follows by an induction argument on the dimension
$n$ that, if $\mu$ verifies $\alpha(\mathcal{T}_c)\le H,$ then
$\mu^n$ verifies $\alpha^{\square n}(\mathcal{T}_{c^{\oplus
n}})\le H$ where as a definition
$$
c^{\oplus n}\Big((x_1,y_1),\dots,(x_n,y_n)\Big):=\sum_{i=1}^n
c(x_i,y_i).
$$
Since $\alpha^{\square n}(t)=n\alpha(t/n)$ for all $t\ge0,$ we
have proved the next proposition.

\begin{prop}\label{res-n-tensorization}
Suppose that $\mu\in\PX$ verifies the transport-entropy inequality
$\alpha(\mathcal{T}_c)\le H$ with $\alpha:[0,\infty)\to[0,\infty)$
a convex increasing function. Then, $\mu^n\in\mathrm{P}(\X^n)$
verifies the transport-entropy inequality
\begin{equation*}
    n\alpha\left(\frac{\mathcal{T}_{c^{\oplus n}}(\nu,\mu^n)}{n}\right)\le H(\nu|\mu^n),
\end{equation*}
for all  $\nu\in\mathrm{P}(\X^n).$
\end{prop}
We also give at the end of Section \ref{section transport inequalities} an alternative
proof of this result which is based on a duality argument. The general statements of Propositions \ref{res-tensorization} and
\ref{res-n-tensorization} appeared in the authors' paper
\cite{GL07}.

In particular, when $\alpha$ is linear, one observes that the inequality $\alpha(\mathcal{T}_c)\leq H$ tensorizes independently of the dimension. This is for example the case for the inequality $\T_2$. So, using the one dimensional $\T_2$ verified by the standard Gaussian measure $\gamma$ together with the above tensorization property, we conclude that for all positive integer $n$, the standard Gaussian measure $\gamma^n$ on $\R^n$ verifies the inequality $\T_2(2)$.

Now let us compare the concentration properties of product
measures derived from $\T_1$ or $\T_2$. Let $d$ be a metric on
$\X$, and let us consider the $\ell_1$ and $\ell_2$ product
metrics associated to the metric $d$: $$d_1(x,y)=\sum_{i=1}^n d(x_i,y_i)\quad
\text{and}\quad d_2(x,y)=\left(\sum_{i=1}^n
d^2(x_i,y_i)\right)^{1/2},\quad x,y\in \X^n.$$ The distance $d_1$ and
$d_2$ are related by the following obvious inequality
$$\frac{1}{\sqrt{n}}d_1(x,y)\leq d_2(x,y)\leq d_1(x,y),\quad x,y\in \X^n. $$

If $\mu$ verifies $\T_1$ on $\X$, then according to Proposition
\ref{res-n-tensorization}, $\mu^n$ verifies the inequality
$\T_1(nC)$ on the space $\X^n$ equipped with the metric $d_1$. It
follows from Marton's concentration Theorem \ref{Marton-intro},
that
\begin{equation}\label{comparaison T1-T2-1}
\mu^n(f>m_f+r+r_o)\leq e^{-\frac{r^2}{nC}},\quad r\geq r_o=\sqrt{nC\log(2)},
\end{equation}
for all function $f$ which is $1$-Lipschitz with respect to $d_1$. So the constants appearing in the concentration inequality are getting worse and worse when the dimension increases.

On the other hand, if $\mu$ verifies $\T_2(C)$, then according to
Proposition \ref{res-n-tensorization}, $\mu^n$ verifies the
inequality $\T_2(C)$ on the space $\X^n$ equipped with $d_2$.
Thanks to Jensen inequality $\mu^n$ also verifies the inequality
$\T_1(C)$ on $(\X^n,d_2), $ and so
\begin{equation}\label{comparaison T1-T2-2}
\mu^n(g>m_g+r+r_o)\leq e^{-\frac{r^2}{C}},\quad r\geq r_o=\sqrt{C\log(2)},
\end{equation}
for all function $g$ which is $1$-Lipschitz with respect to $d_2$.
This time, one observes that the concentration profile does not
depend on the dimension $n$. This phenomenon is called (Gaussian)
\emph{dimension-free concentration of measure}. For instance, if
$\mu=\gamma$ is the standard Gaussian measure, we thus obtain
\begin{equation}\label{eqL-03}
    \gamma^n(f>m_f+r+r_o)\ge 1-e^{-r^2/2},\quad r\ge r_o:=\sqrt{2\log 2}
\end{equation}
for all function $f$ which is $1$-Lipschitz for the Euclidean
distance on $\R^n.$ This result is very near the optimal
concentration profile obtained by an isoperimetric method, see
\cite{Led}. In fact the Gaussian dimension-free property
\eqref{comparaison T1-T2-2} is intrinsically related to the
inequality $\T_2$. Indeed, a recent result of Gozlan \cite{G09}
presented in Section \ref{section large deviations} shows that
Gaussian dimension concentration holds if and only if the
reference measure $\mu$ verifies $\T_2$ (see Theorem
\ref{Concentration->Transport} and Corollary
\ref{T2=concentration}).

Since a $1$-Lipschitz function $f$ for $d_1$ is
$\sqrt{n}$-Lipschitz for $d_2$, it is clear that
\eqref{comparaison T1-T2-2} gives back \eqref{comparaison
T1-T2-1}, when applied to $g=f/\sqrt{n}.$ On the other hand, a
$1$-Lipschitz function $g$ for $d_2$ is also $1$-Lipschitz for
$d_1$, and its is clear that for such a function $g$ the
inequality \eqref{comparaison T1-T2-2} is much better than
\eqref{comparaison T1-T2-1} applied to $f=g$. So, we see from this
considerations that $\T_2$ is a much stronger property than
$\T_1$. We refer to \cite{Led} or \cite{T95, T96c}, for examples
of applications where the independence on $n$ in concentration
inequalities plays a decisive role.

Nevertheless, dependence on $n$ in concentration is not always
something to fight against, as shown in the following example of
deviation inequalities. Indeed, suppose that $\mu$ verifies the
inequality $\alpha\left(\mathcal{T}_d\right) \leq H$, then for
all positive integer $n$,
$$\mu^n\left(f\geq \int f\,d\mu^n +t\right)\leq e^{-n\alpha(t/n)},\quad t\geq 0,$$
for all $f$ $1$-Lipschitz for $d_1$ (see Corollary \ref{resL-11}).
In particular, choose $f(x)=u(x_1)+\cdots +u(x_n)$, with $u$ a
$1$-Lipschitz function for $d$; then $f$ is $1$-Lipschitz for
$d_1$, and so if $X_i$ is an i.i.d sequence of law $\mu$, we
easily arrive at the following deviation inequality
$$\P\left(\frac{1}{n}\sum_{i=1}^n u(X_i)\geq \E[u(X_1)]+t\right)\leq e^{-n\alpha(t)},\quad t\geq 0.$$
This inequality presents the right dependence on $n$. Namely,
according to Cram\'er theorem (see \cite{DZ}) this probability
behaves like $e^{-n\Lambda^*_u(t)}$ when $n$ is large, where
$\Lambda_u^*$ is the Cram\'er transform of $u(X_1)$. The reader
can look at \cite{GL07} for more information on this subject. Let
us mention that this family of deviation inequalities characterize
the inequality $\alpha(\mathcal{T}_d)\leq H$ (see Theorem
\ref{resL-10} and Corollary \ref{resL-11}).

\section{Optimal transport}\label{section optimal transport}

Optimal transport is an active field of research. The recent
textbooks by Villani \cite{Vill, Vill2} make a very  good
account on the subject. Here, we recall basic results which will
be necessary to understand transport inequalities. But the
interplay between optimal transport and functional inequalities in
general is wider than what will be exposed below, see \cite{Vill,
Vill2} for instance.

Let us make our underlying assumptions precise. The cost function
$c$ is assumed to be a \lsc\ $[0,\infty)$-valued function on the
product $\X^2$ of the polish space $\X.$ The Monge-Kantorovich
problem with  cost function $c$ and marginals $\nu,\mu$ in $\PX,$
as well as its optimal value $\mathcal{T}_c(\nu,\mu)$ were stated
at \eqref{MK} and \eqref{eq-14} in Section \ref{sec-overview}.

\begin{prop}\label{resL-a}
The Monge-Kantorovich problem \eqref{MK} admits a solution if and
only if $\mathcal{T}_c(\nu,\mu)<\infty.$
\end{prop}

\begin{proof}[Outline of the proof]
The main ingredients of the proof of this proposition are
\begin{itemize}
    \item the compactness with respect to the narrow topology of
$\{\pi\in\PXX; \pi_0=\nu,\pi_1=\mu\}$ which is inherited from the
tightness of $\nu$ and $\mu$ and
    \item the lower semicontinuity of $\pi\mapsto\IXX c\,d\pi$ which is inherited from
    the lower semicontinuity of $c.$
\end{itemize}
The polish assumption on $\X$ is invoked at the first item.
\end{proof}

The minimizers of \eqref{MK} are called optimal transport plans,
they are not unique in general since \eqref{MK} is not a strictly
convex problem: it is an infinite dimensional linear programming
problem.

If $d$ is a \lsc\ metric on $\X$ (possibly different from the
metric which turns $\X$ into a polish space), one can consider the
cost function $c=d^p$ with $p\ge1.$ One can prove that
$
W_p(\nu,\mu):=\mathcal{T}_{d^p}(\nu,\mu)^{1/p}
$
defines a metric on the set $P_{d^p}(\X)$ (or $\Pp$ for short) of
all probability measures which integrate $d^p(x_o,\,\cdot\,)$: it is
the so-called Wasserstein metric of order $p.$ Since
$\mathcal{T}_{d^p}(\nu,\mu)<\infty$ for all $\nu,\mu$ in $\Pp,$
Proposition \ref{resL-a} tells us that the corresponding problem
\eqref{MK} is attained in $\Pp.$

\subsection*{Kantorovich dual equality}
In the perspective of transport inequalities, the keystone is the
following result. Let $\CX$ be the space of all continuous bounded
functions on $\X$ and denote $u\oplus v(x,y)=u(x)+v(y),$
$x,y\in\X.$

\begin{thm}[Kantorovich dual equality]\label{res-01}
For all $\mu$ and $\nu$ in $\PX,$ we have
\begin{align}
    \mathcal{T}_c(\nu,\mu)
    &=\sup\left\{\IX u(x)\,d\nu(x)+\IX v(y)\,d\mu(y); u,v\in\CX,u\oplus v\le c\right\}\label{eq-01a}\\
    &= \sup\left\{\IX u(x)\,d\nu(x)+\IX v(y)\,d\mu(y); u\in L^1(\nu),v\in L^1(\mu),u\oplus v\le c\right\}.\label{eq-01b}
\end{align}
\end{thm}

Note that for all $\pi$ such that $\ \pi_0=\nu,\pi_1=\mu$ and
$(u,v)$ such that $u\oplus v\le c,$ we have $\IX u\,d\nu+\IX
v\,d\mu=\IXX u\oplus v\,d\pi\le \IXX c\,d\pi.$ Optimizing both
sides of this inequality leads us to
\begin{equation}\label{eq-03}
\begin{split}
 \sup\bigg\{\IX u\,d\nu+\IX v\,d\mu;\ & u\in L^1(\nu),v\in L^1(\mu),u\oplus v\le
 c\bigg\}\\
    \le &\inf\left\{ \IXX c\,d\pi;\pi\in\PXX; \pi_0=\nu,\pi_1=\mu\right\}
\end{split}
\end{equation}
and Theorem \ref{res-01} appears to be a \emph{no dual gap}
result.

The following is a sketch of proof which is borrowed from
L\'eonard's paper \cite{Leo10}.

\begin{proof}[Outline of the proof of Theorem \ref{res-01}] For a
detailed proof, see \cite[Thm 2.1]{Leo10}. Denote
$\mathrm{M}(\X^2)$ the space of all signed measures on $\X^2$ and
$
\iota_{\{x\in A\}}=\left\{
\begin{array}{cl}
    0 & \hbox{if }x\in A \\
    +\infty & \hbox{otherwise} \\
\end{array}
\right.. $ Consider the $(-\infty,+\infty]$-valued function
\begin{equation*}
    K(\pi,(u,v))=\IX u\,d\nu+\IX v\,d\mu-\IXX u\oplus v\,d\pi
    +\IXX c\,d\pi+\iota_{\{\pi\ge0\}},
    \quad \pi\in \mathrm{M}(\X^2), u,v\in\CX.
\end{equation*}
For each fixed $(u,v),$ it is a convex function of $\pi$ and for
each fixed $\pi,$ it is a concave function of $(u,v).$ In other
words, $K$ is a convex-concave function and one can expect that it
admits a saddle value, i.e.
\begin{equation}\label{eq-02}
    \inf_{\pi\in \mathrm{M}(\X^2)}\sup_{u,v\in\CX}K(\pi,(u,v))
    =\sup_{u,v\in\CX}\inf_{\pi\in \mathrm{M}(\X^2)} K(\pi,(u,v)).
\end{equation}
The detailed proof amounts to check that standard assumptions for
this min-max result hold true for $(u,v)$ as in \eqref{eq-01a}. We
are going to show that \eqref{eq-02} is the desired equality
\eqref{eq-01a}. Indeed, for fixed $\pi,$
\begin{align*}
  \sup_{(u,v)} K(\pi,(u,v))
  &= \IXX c\,d\pi+\iota_{\{\pi\ge0\}} + \sup_{(u,v)}\left\{\IX u\,d\nu+\IX v\,d\mu-\IXX u\oplus v\,d\pi\right\}\\
  &= \IXX c\,d\pi+\iota_{\{\pi\ge0\}} + \sup_{(u,v)}\left\{\IX u\,d(\nu-\pi_0)+\IX v\,d(\mu-\pi_1)\right\}\\
  &= \IXX c\,d\pi+\iota_{\{\pi_0=\nu,\pi_1=\mu\}}
\end{align*}
and for fixed $(u,v),$
\begin{equation*}
  \inf_{\pi} K(\pi,(u,v))
  = \IX u\,d\nu+\IX v\,d\mu+\inf_{\pi\ge0}\IXX(c-u\oplus v)\,d\pi
  = \IX u\,d\nu+\IX v\,d\mu-\iota_{\{u\oplus v\le c\}}.
\end{equation*}
Once \eqref{eq-01a} is obtained, \eqref{eq-01b} follows
immediately from \eqref{eq-03} and the following obvious
inequality:
    $ \displaystyle 
    \sup_{u,v\in\CX,u\oplus v\le c}\le \sup_{u\in L^1(\nu),v\in L^1(\mu),u\oplus v\le c}.
    $ 
\end{proof}
Let $ u$ and $ v$ be measurable functions on $\X$ such that $
u\oplus v\leq c.$  The family of inequalities $ v(y)\leq c(x,y)-
u(x),$ for all $x,y$ is equivalent to $ v(y)\leq \inf_x\{c(x,y)-
u(x)\}$ for all $y.$ Therefore, the function
\[
 u^c(y):= \inf_{x\in\X}\{c(x,y)- u(x)\},\quad y\in\X
\]
satisfies $ u^c\geq  v$ and $ u\oplus u^c\leq c.$ As $J(u,v):=\IX
u\,d\nu +\IX v\,d\mu$ is an increasing function of its arguments $
u$ and $ v,$ in view of maximizing $J$ on the set
    $ 
\{(u,v)\in L_1(\nu)\times L_1(\mu):  u\oplus v\leq c\},
    $
the couple $( u, u^c)$ is better than $(u,v).$ Performing this
trick once again, we see that with
    $
 v^{c}(x):= \inf_{y\in\X}\{c(x,y)- v(y)\},\ x\in\X,
   $
the couple $( u^{cc}, u^c)$ is  better than $( u, u^c)$ and
$(u,v).$ We have obtained the following result.

\begin{lem}\label{res-02}
Let $ u$ and $ v$ be functions on $\X$ such that $ u(x)+ v(y)\leq
c(x,y)$ for all $x,y.$ Then, $ u^c$ and $ u^{cc}$  also satisfy $
u^{cc}\geq  u,
 u^c\geq  v$ and $ u^{cc}(x)+ u^c(y)\leq c(x,y)$
for all $x,y.$
\end{lem}
Iterating the trick of Lemma \ref{res-02}  doesn't improve
anything.

\begin{rem} (Measurability of $u^c$).\ This  issue is often
neglected in the literature. The aim of this remark is to indicate
a general result which solves this difficult problem. If $c$ is
continuous, $u^c$ and $u^{cc}$ are upper semicontinuous, and
therefore they are Borel measurable. In the general case where $c$
is \lsc, it can be shown that some measurable version of $u^c$
exists. More precisely, Beiglb\"ock and Schachermayer have proved
recently in \cite[Lemmas 3.7, 3.8]{BS09} that, even if $c$ is only
supposed to be Borel measurable, for each probability measure
$\mu\in\PX,$ there exists a $[-\infty,\infty)$-valued Borel
measurable function $\tilde{u}^c$ such that $\tilde{u}^c\le u^c$
everywhere and $\tilde{u}^c= u^c,$ $\mu$-almost everywhere. This
is precisely what is needed for the purpose of defining the
integral $\int u^c\,d\mu.$
\end{rem}

Recall that whenever $A$ and $B$ are two vector spaces linked by
the duality bracket $\langle a,b\rangle,$ the convex conjugate of
the function $f:A\to (-\infty,\infty]$ is defined by
\begin{equation*}
    f^*(b):=\sup_{a\in A}\{\langle a,b\rangle-f(a)\}\in(-\infty,\infty],\quad b\in B.
\end{equation*}
Clearly, the definition of $u^c$ is reminiscent of that of $f^*.$
Indeed, with the quadratic cost function
$c_2(x,y)=|y-x|^2/2=\frac{|x|^2}2+\frac{|y|^2}2-x\scal y$ on
$\Rk,$ one obtains
\begin{equation}\label{eqL-a}
    \frac{|\,\cdot\,|^2}2-u^{c_2}=\left(\frac{|\,\cdot\,|^2}2-u\right)^*.
\end{equation}
It is worth recalling basic facts about convex conjugates for we
shall use them several times later. Being the supremum of a family
of affine continuous functions, $f^*$  is convex and
$\sigma(B,A)$-\lsc. Defining $f^{**}(a)=\sup_{b\in B}\{\langle
a,b\rangle-f^*(b)\}\in(-\infty,\infty],$ $a\in A,$ one knows that
$f^{**}=f$ if and only if $f$ is a \lsc\ convex function. It is a
trivial remark that
\begin{equation}\label{ineq-Fenchel}
    \langle a,b\rangle\le f(a)+f^*(b),\quad  (a,b)\in A\times B.
\end{equation}
 The case of equality (Fenchel's identity) is of special
interest, we have
\begin{equation}\label{eq-Fenchel}
    \langle a,b\rangle = f(a)+f^*(b)\Leftrightarrow b\in\partial f(a)\Leftrightarrow a\in\partial f^*(b)
\end{equation}
whenever $f$ is convex and $\sigma(A,B)$-\lsc. Here, $\partial
f(a):=\{b\in B; f(a+h)\ge f(a)+ \langle h,b\rangle,\forall h\in
A\}$ stands for the subdifferential of $f$ at $a.$

\subsection*{Metric cost}
The cost function to be considered is $c(x,y)=d(x,y)$: a \lsc\
metric on $\X$ which might be \emph{different} from the original
polish metric on $\X.$

\begin{rem}\label{rem-00}
In the sequel, the Lipschitz functions are to be considered with
respect to the metric cost $d$ and not with respect to the
underlying metric on the polish space $\X$ which is here to
generate the Borel $\sigma$-field, specify the continuous, \lsc\
or Borel functions. Indeed, we have in mind to
 work sometimes with trivial metric costs (weighted Hamming's metrics) which
are \lsc\ with respect to any reasonable non-trivial metric but
generate a too rich Borel $\sigma$-field. As a consequence a
$d$-Lipschitz function might not be Borel measurable.
\end{rem}

One writes that $u$ is $d$-Lipschitz(1) to specify that
$|u(x)-u(y)|\leq d(x,y)$ for all $x,y\in\X.$ Denote $\Pu:=\{\nu\in
\PX; \IX d(x_o,x)\,d\nu(x)\}$ where $x_o$ is any fixed element in
$\X.$ With the triangle inequality, one sees that $\Pu$ doesn't
depend on the choice of $x_o.$
\\
Let us denote the Lipschitz seminorm
    $
\|u\|_\mathrm{Lip}:=\sup_{x\not =y}\frac{|u(y)-u(x)|}{d(x,y)}.
    $
Its dual norm is for all $\mu,\nu$ in $\Pu,$
    $
\|\nu-\mu\|^*_\mathrm{Lip}=\sup\left\{\IX u(x)\,[\nu-\mu](dx); u
\mathrm{\ measurable}, \|u\|_\mathrm{Lip}\leq 1\right\}.
    $
As it is assumed that $\mu,\nu\in \Pu,$ note that any measurable
$d$-Lipschitz function is integrable with respect to $\mu$ and
$\nu.$
\begin{thm}[Kantorovich-Rubinstein]\label{res-KR}
For all $\mu,\nu\in \Pu, $
    $
    W_1(\nu,\mu) = \|\nu-\mu\|^*_\mathrm{Lip}.
    $
\end{thm}
\proof For all measurable $d$-Lipschitz(1) function  $u$ and all
$\pi$ such that $\pi_0=\nu$ and $\pi_1=\mu,$
 $\IX u(x)\,[\nu-\mu](dx)=\IXX (u(x)-u(y))\,d\pi(x,y) \leq \IXX d(x,y)\,d\pi(x,y).$
   Optimizing in $u$ and $\pi$ one obtains
 $\|\nu-\mu\|_\mathrm{Lip}^*\leq W_1(\nu,\mu).$
\\
Let us look at the reverse inequality.
\\
\textit{Claim.} For any function $u$ on $\X,$ (i)\
    $u^d$ is $d$-Lipschitz(1)
    and (ii)\ $u^{dd}=-u^d.$
\\
Let us prove (i). Since $y\mapsto d(x,y)$ is $d$-Lipschitz(1),
$y\mapsto u^d(y)=\inf_x\{d(x,y)-u(x)\}$ is also $d$-Lipschitz(1)
as an infinum of  $d$-Lipschitz(1) functions.
\\
Let us prove (ii). Hence for all $x,y,$ $u^d(y)-u^d(x)\leq
d(x,y).$ But this implies that for all $y,$ $-u^d(x)\leq
d(x,y)-u^d(y).$ Optimizing in $y$ leads to $-u^d(x)\leq
u^{dd}(x).$ On the other hand,
$u^{dd}(x)=\inf_y\{d(x,y)-u^d(y)\}\leq -u^d(x)$ where the last
inequality is obtained by taking $y=x.$

With Theorem \ref{res-01}, Lemma \ref{res-02} and the above claim,
we obtain that
\begin{align*}
  W_1(\nu,\mu)&=\sup_{(u,v)}\left\{\IX u\,d\nu +\IX v\,d\mu\right\}=\sup_u\left\{\IX
  u^{dd}\,d\nu
+\IX u^d\,d\mu\right\} \\
  &\le\sup\left\{\IX u \,d[\nu-\mu]; u :  \|u\|_\mathrm{Lip}\leq 1
   \right\} = \|\nu-\mu\|^*_\mathrm{Lip}.\\
\end{align*}
which completes the proof of the theorem.
\endproof
For interesting consequences in probability theory, one can look
at Dudley's textbook \cite[Chp.\ 11]{Dud02}.

\subsection*{Optimal plans}
What about the optimal plans? If one applies \emph{formally} the
Karush-Kuhn-Tucker characterization of the saddle point of the
Lagrangian function $K$ in the proof of Theorem \ref{res-01}, one
obtains that $\hat\pi$ is an optimal transport plan if and only if
$0\in\partial_\pi K(\hat \pi,(\hat u,\hat v))$ for some couple of
functions $(\hat u,\hat v)$  such that
$0\in\widehat{\partial}_{(u,v)}K(\hat \pi,(\hat u,\hat v))$ where
$\partial_\pi K$ stands for the subdifferential of the convex
function $\pi \mapsto K(\pi,(\hat u,\hat v))$ and
$\widehat{\partial}_{(u,v)}K$ for the superdifferential of the
concave function $(u,v)\mapsto K (\hat\pi,(u,v)).$ This gives us
the system of equations
\begin{equation*}
    \left\{
    \begin{array}{rcl}
      \hat{u}\oplus \hat{v} -c &\in&\partial (\iota_{M_+})(\hat{\pi})\\
      (\hat{\pi}_0,\hat{\pi}_1)&=&(\nu,\mu) \\
    \end{array}
    \right.
\end{equation*}
where $M_+$ is the cone of all positive measures on $\X^2.$ Such a
couple $(\hat u,\hat v)$ is called a dual optimizer. The second
equation expresses the marginal constraints of \eqref{MK} while by
\eqref{eq-Fenchel} one can recast the first one as the Fenchel
identity
    $\langle \hat{u}\oplus \hat{v}-c,\hat{\pi}\rangle=\iota_{M_+}^*(\hat{u}\oplus
    \hat{v}-c)+\iota_{M_+}(\hat{\pi}).$
Since for any function $h,$ $\iota_{M_+}^*(h)=\sup_{\pi\in
M_+}\langle h,\pi\rangle=\iota_{\{h\le 0\}},$ one sees that
$\langle \hat{u}\oplus \hat{v}-c,\hat{\pi}\rangle=0$ with
$\hat{u}\oplus \hat{v}-c\le0$ and $\hat{\pi}\ge0$ which is
equivalent to $\hat{\pi}\ge0,$ $\hat{u}\oplus \hat{v}\le c$
everywhere  and $\hat{u}\oplus \hat{v}= c,$ $\hat{\pi}$-almost
everywhere. As $\hat{\pi}_0=\nu$ has a unit mass, so has the
positive measure $\hat{\pi}:$ it is a probability measure. These
formal considerations should prepare the reader to trust the
subsequent rigorous statement.

\begin{thm}\label{resL-c} Assume that $\mathcal{T}_c(\nu,\mu)<\infty.$
Any $\pi\in \PXX$ with the prescribed marginals $\pi_0=\nu$ and
$\pi_1=\mu$ is an optimal plan if and only if there exist two
measurable functions $u,v:\X\to[-\infty,\infty)$ such that
\begin{equation*}
    \left\{
    \begin{array}{rlll}
      u\oplus v &\le&c, & \textrm{everywhere} \\
       u\oplus v &=&c, & \pi\textrm{-almost everywhere.} \\
    \end{array}
    \right.
\end{equation*}
\end{thm}
This theorem can be found in \cite[Thm.\ 5.10]{Vill2} with a proof
which has almost nothing in common with the saddle-point strategy
that has been described above.

An important instance of this result is the special case of the
quadratic cost.

\begin{cor}\label{resL-b} Let us consider the
quadratic cost $c_2(x,y)=|y-x|^2/2$ on $\X=\Rk$ and take two
probability measures $\nu$ and $\mu$ in $\Pd(\X).$
\begin{enumerate}[(a)]
    \item There exists an optimal transport plan.
    \item Any $\pi\in \PXX$ is optimal if and only if there exists
    a  convex \lsc\ function
    $\phi:\X\to(-\infty,\infty]$ such that the Fenchel identity $\phi(x)+\phi^*(y)=x\scal
    y$ holds true $\pi$-almost everywhere.
\end{enumerate}
\end{cor}

\begin{proof}
Proof of (a). Since $|y-x|^2/2\le |x|^2+|y|^2$ and $\nu,\mu\in
\Pd,$ any $\pi\in \PXX$ such that $\pi_0=\nu$ and $\pi_1=\mu$
satisfies $\IXX c\,d\pi\le \IX |x|^2\,d\nu(x)+\IX
|y|^2\,d\mu(y)<\infty.$ Therefore
$\mathcal{T}(\nu,\mu)=W^2_2(\nu,\mu)<\infty,$ and one concludes
with Proposition \ref{resL-a}.

\noindent Proof of (b). In view of Lemma \ref{res-02}, one sees
that an optimal dual optimizer is necessarily of the form
$(u^{cc},u^c).$  Theorem \ref{resL-c} tells us that $\pi$ is
optimal if and only there exists some function $u$ such that
$u^{cc}\oplus u^c=c,$ $\pi$-almost everywhere. With \eqref{eqL-a},
by considering the functions $\phi(x)=|x|^2/2-u^{c_2c_2}(x)$ and
$\psi(y)=|y|^2/2-u^{c_2}(y),$ one also obtains $\phi=\psi^*$ and
$\psi=\phi^*,$ which means that $\phi$ and $\psi$ are convex
conjugate to each other and in particular that $\phi$ is convex
and \lsc.
\end{proof}

By \eqref{eq-Fenchel}, another equivalent statement for the
Fenchel identity $\phi(x)+\phi^*(y)=x\scal y$ is
\begin{equation}\label{eq-07}
y\in\partial \phi(x).
\end{equation}
In the special case of the real line $\X=\R,$ a popular coupling
of $\nu$ and $\mu=T_\#\nu$ is given by the so-called
\emph{monotone rearrangement.} It is defined by
\begin{equation}\label{eq-08}
    y=T(x):=F_\mu^{-1}\circ F_\nu(x), \quad x\in\R,
\end{equation}
where $F_\nu(x)=\nu((-\infty,x]),$ $F_\mu(y)=\mu((-\infty,y])$ are
the distribution functions of $\nu$ and $\mu,$  and
$F_\mu^{-1}(u)=\inf\{y\in\R; F_\mu(y)> u\},$ $u\in[0,1]$ is the
generalized inverse of $F_\mu.$ When $\X=\R,$ the identity
\eqref{eq-07} simply states that $(x,y)$ belongs to the graph of
an increasing function. Of course, this is the case of
\eqref{eq-08}. Hence Corollary \ref{resL-b} tells us that the
monotone rearrangement is an optimal transport map for the
quadratic cost.

Let us go back to $\X=\Rk.$ If $\phi$ is G\^ateaux differentiable
at $x,$ then $\partial \phi(x)$ is restricted to a single element:
the gradient $\nabla\phi(x),$ and \eqref{eq-07}  simply becomes
$y=\nabla\phi (x).$ Hence, if $\phi$ were differentiable
everywhere, condition (b) of Corollary \ref{resL-b} would be
$y=\nabla\phi (x),$ $\pi$-almost everywhere. But this is too much
demanding. Nevertheless, Rademacher's theorem states that a convex
function on $\Rk$ is differentiable Lebesgue almost everywhere on
its effective domain. This allows to derive the following
improvement.

\begin{thm}[Quadratic cost on $\X=\Rk$]\label{res-06}
Let us consider the quadratic cost $c_2(x,y)=|y-x|^2/2$ on
$\X=\Rk$ and take two probability measures $\nu$ and $\mu$ in
$\Pd(\X)$ which are absolutely continuous. Then, there exists a
unique optimal plan. Moreover, $\pi\in\PXX$ is optimal if and only
if $\pi_0=\nu, \pi_1=\mu$ and there exists a convex function
$\phi$ such that
\begin{equation*}
\left\{
    \begin{array}{rcl}
      y&=&\nabla\phi(x) \\
      x&=&\nabla\phi^*(y) \\
    \end{array}
\right. ,\quad \pi \textrm{-almost
    everywhere}.
\end{equation*}
\end{thm}

\begin{proof}
It follows from Rademacher's theorem that, if $\nu$ is an
absolutely continuous measure, $\phi$ is differentiable
$\nu$-almost everywhere. This and Corollary \ref{resL-b} prove the
statement about the characterization of the optimal plans. Note
that the quadratic transport is symmetric with respect to $x$ an
$y,$ so that one obtains the same conclusion if $\mu$ is
absolutely continuous; namely $x=\nabla\phi^*(y),$ $\pi$-almost
everywhere, see \eqref{eq-Fenchel}.

We have just proved that, under our assumptions, an optimal plan
is concentrated on a functional graph.  The uniqueness of the
optimal plan follows directly from this. Indeed, if one has two
optimal plans $\pi^0$ and $\pi^1$, by convexity their half sum
$\pi^{1/2}$ is still optimal. But for $\pi^{1/2}$ to be
concentrated on a \emph{functional} graph, it is necessary that
$\pi^0$ and $\pi^1$ share the same graph.
\end{proof}

For more details, one can have a look at \cite[Thm.\ 2.12]{Vill}.
The existence result has been obtained by Knott and Smith
\cite{KS84}, while the uniqueness is due to Brenier \cite{Bre91}
and McCann \cite{McC95}. The application $y=\nabla\phi(x)$ is
often called the \emph{Brenier map} pushing forward $\nu$ to
$\mu.$

\section{Dual equalities and inequalities}\label{section transport inequalities}
In this section, we present Bobkov and G\"otze dual approach to transport-entropy inequalities \cite{BG99}. More precisely, we are going to take advantage of variational formulas for the optimal transport cost and for the relative entropy to give another formulation of transport-entropy inequalities. The relevant variational formula for the transport cost is given by
the Kantorovich dual equality at Theorem \ref{res-01}

\begin{equation}\label{eqL-e}
    \begin{split}
    \mathcal{T}_c(\nu,\mu)
    &=\ \sup\left\{\int u\,d\nu+\int v\,d\mu; u,v\in\CX,u\oplus v\le c\right\}.
\end{split}
\end{equation}

On the other hand, the relative entropy admits the following
variational representations. For all $\nu\in\PX,$
\begin{equation}\label{eqL-c}
\begin{split}
  H(\nu|\mu)&= \sup\left\{\int u\,d\nu-\log\int e^u\,d\mu; u\in\CX\right\}.\\
        &= \sup\left\{\int u\,d\nu-\log\int e^u\,d\mu; u\in \BX\right\}
\end{split}
\end{equation}
and for all $\nu\in\PX$ \textrm{ such that } $\nu\ll\mu,$
\begin{equation}\label{eqL-d}
    H(\nu|\mu) =\sup\left\{\int u\,d\nu-\log\int e^u\,d\mu; u: \textrm{measurable,} \int
    e^{u}\,d\mu<\infty, \int u_-\,d\nu<\infty
    \right\}
\end{equation}
where $u_-=(-u)\vee 0$ and $\int u\,d\nu\in(-\infty,\infty]$ is
well-defined for all $u$ such that $\int u_-\,d\nu<\infty.$

The identities \eqref{eqL-c} are well-known, but the proof of
\eqref{eqL-d} is more confidential. This is the reason why we give
their detailed proofs at the Appendix, Proposition \ref{resL-13}.

As regards Remark \ref{rem-a}, a sufficient condition for $\mu$ to
satisfy $H(\nu|\mu)=\infty$ whenever $\nu\not\in\Pp$ is
\begin{equation}\label{eqL-f}
    \int e^{s_od^p(x_o,x)}\,d\mu(x)<\infty
\end{equation}
for some $x_o\in\X$ and  $s_o>0.$ Indeed, by \eqref{eqL-d}, for
all $\nu\in\PX,$ $s_o\int d^p(x_o,x)\,d\nu(x)\le
H(\nu|\mu)+\log\int e^{s_od^p(x_o,x)}\,d\mu(x).$ On the other
hand,  Proposition \ref{res-08} below tells us that \eqref{eqL-f}
is also a necessary condition.

Since $u\mapsto \Lambda(u):=\log\int e^u\,d\mu$ is convex (use
H\"older inequality to show it) and \lsc\ on $\CX$ (resp. $\BX$)
(use Fatou's lemma), one observes that $H(\,\cdot\,|\mu)$ (more
precisely its extension to the vector space of signed bounded
measures which achieves the value $+\infty$ outside $\PX$) and
$\Lambda$ are convex conjugate to each other:
$$
\left\{
    \begin{array}{l}
      H(\,\cdot\,|\mu)=\Lambda^*, \\
      \Lambda=H(\,\cdot\,|\mu)^*. \\
    \end{array}
\right.
$$
It appears that $\mathcal{T}_c(\,\cdot\,,\mu)$ and
$H(\,\cdot\,|\mu)$ both can be written as convex conjugates of
functions on a class of functions on $\X.$ This structure will be
exploited in a moment to give a dual formulation of inequalities
$\alpha(\mathcal{T}_c)\leq H$, for $\alpha$ belonging to the
following class.

\begin{defi}[of $\mathcal{A}$]\label{class A}
The class $\mathcal{A}$ consists of all the functions $\alpha$ on
$[0,\infty)$ which are convex, increasing with $\alpha(0)=0.$
\end{defi}

The convex conjugate of a function $\alpha\in\mathcal{A}$ is
replaced by the monotone conjugate $\alpha^{\mc}$ defined by
\[
\alpha^{\mc}(s)=\sup_{r\geq 0}\{sr-\alpha(r)\},\quad s\geq 0
\]
where the supremum is taken on $r\geq 0$ instead of $r\in\R.$

\begin{thm}\label{BG1}
Let $c$ be a lower semicontinuous cost function, $\alpha\in \mathcal{A}$ and $\mu\in \mathrm{P}(\X)$; the following propositions are equivalent.
\begin{enumerate}
\item The probability measure $\mu$ verifies the inequality $\alpha(\mathcal{T}_c)\leq H$.
\item For all $u,v\in \CX$, such that $u\oplus v\leq c$,
$$\int e^{su}\,d\mu\leq e^{-s\int v\,d\mu+\alpha^\mc(s)},\quad s\geq 0.$$
\end{enumerate}
Moreover, the same result holds with $\BX$ instead of $\CX$.
\end{thm}

A variant of this result can be found in the authors' paper
\cite{GL07} and in Villani's textbook \cite[Thm.\ 5.26]{Vill2}.
It extends the dual characterization of transport inequalities $\T_1$ and $\T_2$ obtained by  Bobkov and G{\"o}tze in \cite{BG99}.
\proof
First we extend $\alpha$ to the whole real line by defining $\alpha(r)=0$, for all $r\leq 0$.
Using Kantorovich dual equality and the fact that $\alpha$ is continuous and increasing on $\R$, we see that the inequality $\alpha(\mathcal{T}_c)\leq H$ holds if and only if for all $u,v\in \CX,$ such that $u\oplus v\leq c$, one has
$$\alpha\left(\int u\,d\nu+\int v\,d\mu \right)\leq H(\nu|\mu),\quad \nu\in \PX.$$
Since $\alpha$ is convex and continuous on $\R$, it satisfies $\alpha(r)=\sup_s\{sr-\alpha^*(s)\}$. So the preceding condition is equivalent to the following one
$$s\int u\,d\nu-H(\nu|\mu)\leq -s\int v\,d\mu+\alpha^*(s) ,\quad \nu\in \PX, s\in \R, u\oplus v\leq c.$$
Since $H(\,\cdot\,|\mu)^*=\Lambda$, optimizing over $\nu\in \PX$, we arrive at
$$\log \int e^{su}\,d\mu\leq -s\int v\,d\mu+\alpha^*(s) ,\quad  s\in \R, u\oplus v\leq c.$$
Since $\alpha^*(s)=+\infty$ when $s<0$ and $\alpha^*(s)=\alpha^\mc(s)$ when $s\geq 0$, this completes the proof.
\endproof

Let us define, for all $f,g\in \BX$,
$$P_cf(y)=\sup_{x\in \X}\{f(x)-c(x,y)\},\quad y\in \X,$$
and
$$Q_cg(x)=\inf_{y\in \X}\{g(y)+c(x,y)\},\quad x\in \X.$$
For a given function $f:\X\to\R$, $P_cf$ is the best function $g:\X\to\R$ (the smallest) such that
$f(x)-g(y)\leq c(x,y)$, for all $x,y\in \X.$ And for a given function $g:\X\to \R$, $Q_cg$ is the best function $f:\X\to \R$ (the biggest) such that $f(x)-g(y)\leq c(x,y)$, for all $x,y\in\X$.

The following immediate corollary gives optimized forms of the dual condition (2) stated in Theorem \ref{BG1}.
\begin{cor}\label{BG2}
Let $c$ be a lower semicontinuous cost function, $\alpha\in \mathcal{A}$ and $\mu\in \PX$. The following propositions are equivalent.
\begin{enumerate}
\item The probability measure $\mu$ verifies the inequality $\alpha (\mathcal{T}_c)\leq H$.
\item For all $f\in \CX$,
$$\int e^{sf}\,d\mu\leq e^{s\int P_cf\,d\mu +\alpha^\mc(s)},\quad s\geq 0.$$
\item For all $g\in \CX$,
$$\int e^{sQ_cg}\,d\mu\leq e^{s\int g\,d\mu + \alpha^\mc(s)},\quad s\geq 0.$$
\end{enumerate}
Moreover, the same result holds true with $\BX$ instead of $\CX$.
\end{cor}
When the cost function is a lower semicontinuous distance, we have the following.
\begin{cor}\label{BG3}
Let $d$ be a lower semicontinuous distance, $\alpha\in \mathcal{A}$ and $\mu\in \PX$. The following propositions are equivalent.
\begin{enumerate}
\item The probability measure $\mu$ verifies the inequality $\alpha (\mathcal{T}_d)\leq H$.
\item For all $1$-Lipschitz function $f$,
$$\int e^{sf}\,d\mu\leq e^{s\int f\,d\mu +\alpha^\mc(s)},\quad s\geq 0.$$
\end{enumerate}
\end{cor}

Corollary \ref{BG2} enables us to give an alternative proof of the tensorization property given at Proposition \ref{res-n-tensorization}.

\proof[Proof of Proposition \ref{res-n-tensorization}] For the
sake of simplicity, let us explain the proof for $n=2$. The
general case is done by induction (see for instance \cite[Theorem
5]{GL07}). Let us consider, for all $f\in\mathcal{B}_b(\X)$
$$Q_cf(x)=\inf_{y\in \X}\{f(y)+c(x,y)\},\quad x\in \X$$
and for all $f\in\mathcal{B}_b(\X\times\X)$,
$$Q^{(2)}_cf(x)=\inf_{y\in \X\times \X}\{f(y_1,y_2)+c(x_1,y_1)+c(x_2,y_2)\},\quad x\in \X\times\X.$$
According to the dual formulation of transport-entropy
inequalities (Corollary \ref{BG2} (2)), $\mu$ verifies the
inequality $\alpha(\mathcal{T}_c)\leq H$ if and only if
\begin{equation}\label{tensorization1}
\int e^{sQ_cf}\,d\mu\leq e^{s\int f\,d\mu+\alpha^*(s)},\quad s\geq 0
\end{equation}
for all $f\in \mathcal{B}_b(\X).$ On the other hand, $\mu^2$
verifies the inequality $2\alpha\left(\frac{\mathcal{T}_{c^{\oplus
2}}}{2}\right)\leq H$ if and only if
$$\int e^{sQ^{(2)}_cf}\,d\mu^2\leq e^{s\int f\,d\mu^2 + 2\alpha^*(s)},\quad  s\geq 0,$$
holds for all $f\in \mathcal{B}_b(\X\times\X).$ Let $f\in
\mathcal{B}_b(\X\times\X)$,
\begin{align*}
  Q^{(2)}_cf(x_1,x_2) &= \inf_{y_1, y_2\in \X} \{f(y_1,y_2)+c(x_1,y_1)+c(x_2,y_2)\} \\
   &= \inf_{y_1\in\X}\left\{ \inf_{y_2\in \X}\{f(y_1,y_2)+c(x_2,y_2)\}+c(x_1,y_1)\right\} \\
   &= \inf_{y_1\in\X}\{ Q_c(f_{y_1})(x_2)+c(x_1,y_1)\}
   \end{align*}
where for all $y_1\in \X$, $f_{y_1}(y_2)=f(y_1,y_2)$,  $y_2\in\X$.

So, applying \eqref{tensorization1} gives
\begin{align*}
    \int_{\X\times\X} e^{sQ^{(2)}_cf}\,d\mu^2
    &= \int\left(\int e^{s\inf_{y_1\in\X}\{ Q_c(f_{y_1})(x_2)+c(y_1,x_1)\}}\,d\mu(x_1)\right)\,d\mu(x_2) \\
   &\leq e^{\alpha^{\mc}(s)}\int e^{s\int Q_c(f_{x_1})(x_2)\,d\mu(x_1)}\,d\mu(x_2).
\end{align*}
But, $$\int Q_c(f_{x_1})(x_2)\,d\mu(x_1)=\int \inf_{y_1\in
\X}\{f(x_1,y_1)+c(x_2,y_1)\}\,d\mu(x_1)\leq Q_c(\bar{f})(x_2),$$
with $\bar{f}(y_1)=\int f(x_1,y_1)\,d\mu(x_1).$

Applying \eqref{tensorization1} again yields
$$ \int e^{s\int Q_c(f_{x_1})(x_2)\,d\mu(x_1)}\,d\mu(x_2)\leq \int e^{sQ_c(\bar{f})(x_2)}\,d\mu(x_2)\leq e^{\alpha^{\mc}(s) +s\int \bar{f}(x_2)\,d\mu(x_2)}.$$
Since $\int \bar{f}(x_2)\,d\mu(x_2)=\int f\,d\mu^2$, this
completes the proof.
\endproof

To conclude this section, let us put the preceding results in an abstract general setting. Our motivation to do that is to consider transport inequalities involving other functionals $J$ than the entropy.

Consider two convex functions on some
vector space $\mathcal{U}$ of measurable functions on $\X,$
$\Theta: \mathcal{U}\to (-\infty,\infty]$ and $\Upsilon:
\mathcal{U}\to (-\infty,\infty].$  Their convex conjugates are
defined for all  $\nu$ in the space $\mathrm{M}_\mathcal{U}$ of
all measures on $\X$ such that $\int |u|\,d\nu<\infty,$ for all
$u\in\mathcal{U}$ by
\begin{equation*}
\left\{\begin{array}{rcl}
   T(\nu) &=& \sup_{u\in\mathcal{U}}\left\{\int u\,d\nu-\Theta(u)\right\}\\
   J(\nu) &=& \sup_{u\in\mathcal{U}}\left\{\int u\,d\nu-\Upsilon(u)\right\}
\end{array}\right.
\end{equation*}
Without loss of generality, one assumes that $\Upsilon$ is a
convex and $\sigma(\mathcal{U}, \mathrm{M}_\mathcal{U})$-\lsc\
function, so that $J$ and $\Upsilon$ are convex conjugate to each
other. It is assumed that $\mathcal{U}$ contains the constant
functions, $\Theta(0)=\Upsilon(0)=0,$
$\Theta(u+a\mathbf{1})=\Theta(u)+a$ and
$\Upsilon(u+a\mathbf{1})=\Upsilon(u)+a$ for all real $a$ and all
$u\in\mathcal{U}$ and that $\Theta$ and $\Upsilon$ are increasing.
This implies that $T$ and $J$ are $[0,\infty]$-valued with their
effective domain in $\PU:=\{\nu\in \PX; \int
|u|\,d\nu<\infty,\forall u\in\mathcal{U}\}.$
In this setting, we have the following theorem whose proof is a straightforward adaptation of the proof of Theorem \ref{BG1}.
\begin{thm}\label{res-09}
Let $\alpha\in \mathcal{A}$  and $\mathcal{U},T,J$ as above. For
all $u\in\mathcal{U}$ and $s\ge0$ define
$\Upsilon_u(s):=\Upsilon(su)-s\Theta(u).$ The following statements
are equivalent.
\begin{enumerate}[(a)]
    \item For all $\nu\in\PU,$ $\alpha(T(\nu))\le J(\nu).$
    \item For all $u\in\mathcal{U}$ and $s\ge0,$ $ \Upsilon_u(s)\le\alpha^{\mc}(s).$
\end{enumerate}
\end{thm}
This general result will be used in Section \ref{section
transport-information} devoted to transport-information
inequalities, where the functional $J$ is the Fisher information.

\section{Concentration for product probability measures}\label{section concentration}

Transport-entropy inequalities are intrinsically linked to the
concentration of measure phenomenon for product probability
measures. This relation was first discovered by K. Marton in
\cite{M86}. Informally, a concentration of measure inequality
quantifies how fast the probability goes to $1$ when a set $A$ is
enlarged.
\begin{defi}
Let $\X$ be a Hausdorff topological space and let $\mathcal{G}$ be
its Borel $\sigma$-field. An enlargement function is a function
$\mathrm{enl} : \mathcal{G}\times [0,\infty)\to \mathcal{G}$ such
that
\begin{itemize}
\item For all $A\in \mathcal{G}$, $r\mapsto \mathrm{enl} (A,r)$ is
increasing on $[0,\infty)$ (for the set inclusion). \item For
all $r\geq0$,  $A\mapsto \mathrm{enl}(A,r)$ is increasing (for
the set inclusion). \item For all $A\in \mathcal{G}$, $A\subset
\mathrm{enl} (A,0)$. \item For all $A\in \mathcal{G}$,
$\cup_{r\geq 0} \mathrm{enl}(A,r)=\X.$
\end{itemize}
If $\mu$ is a probability measure on $\X$, one says that it
verifies a  concentration of measure inequality if there is a
function $\beta:[0,\infty)\to [0,\infty)$ such that $\beta(r)\to
0$ when $r\to+\infty$ and such that for all $A \in \mathcal{G}$
with $\mu(A)\geq 1/2$ the following inequality holds
\[\mu (\mathrm{enl} (A,r))\geq 1-\beta(r),\quad r\geq 0.\]
\end{defi}

There are many ways of enlarging sets. If $(\X,d)$ is a metric space, a classical way is to
consider the $r$-neighborhood of $A$ defined by
\[A^r=\{x\in \X; d(x,A)\leq r\},\quad r\geq 0,\]
where the distance of $x$ from $A$ is defined by $d(x,A)=\inf_{y\in A}d(x,y).$

Let us recall the statement of Marton's concentration theorem
whose proof was given at Theorem \ref{Marton-intro}.

\begin{thm}[Marton's concentration theorem]\label{Marton}
Suppose that $\mu$ verifies the inequality
$\alpha\left(\mathcal{T}_{d}(\nu,\mu)\right)\leq H(\nu|\mu),$
for all $\nu\in \mathrm{P}(\X)$. Then for all $A\subset \X$, with
$\mu(A)\geq 1/2$, the following holds
$$\mu(A^r)\geq 1-e^{-\alpha(r-r_{o})},\quad  r \geq r_{o}:=\alpha^{-1}(\log 2).$$
\end{thm}

We already stated at Proposition \ref{res-n-tensorization} an
important tensorization result. Its statement is recalled below at
Proposition \ref{tensorization}.

\begin{prop}\label{tensorization}
Let $c$ be a lower semicontinuous cost function on $\X$ and
$\alpha \in \mathcal{A}$ (see Definition \ref{class A}). Suppose
that a probability measure $\mu$ verifies the transport-entropy
inequality $\alpha(\mathcal{T}_c)\leq H$ on $\X$, then $\mu^n$,
$n\geq 1$ verifies the inequality
$$n\alpha\left(\frac{\mathcal{T}_{c^{\oplus n}}(\nu,\mu^n)}{n}\right)\leq H(\nu|\mu^n),\quad \nu\in \mathrm{P}(\X^n),$$
where $c^{\oplus n}(x,y)=\sum_{i=1}^n c(x_{i},y_{i})$.
\end{prop}
Other forms of non-product tensorizations have been studied (see
\cite{M96a,M96b, M98, M04}, \cite{S00} or \cite{DGW04,Wu06}) in
the context of Markov chains or Gibbs measures (see Section
\ref{sec-Gibbs}).

Let us recall a first easy consequence of this tensorization
property.
\begin{cor}\label{Marton T2}
Suppose that a probability measure $\mu$ on $\X$ verifies the
inequality $\T_2(C)$, then $\mu^n$ verifies the inequality
$\T_2(C)$ on $\X^n$, for all positive integer $n$. In particular,
the following dimension-free Gaussian concentration property
holds: for all positive integer $n$ and for all $A\subset \X^n$
with $\mu^n (A)\geq 1/2$,
\[\mu^n (A^r)\geq 1-\exp (-\frac{1}{C}(r-r_o)^2),\quad r\geq r_o:=\sqrt{\log(2)},\]
where $A^r=\{ x\in \X^n; d_2(x,A)\leq r\}$ and
$d_2(x,y)=\left[\sum_{i=1}^n d(x_i,y_i)^2\right]^{1/2}.$

Equivalently, when $\mu$ verifies $\T_{2}(C)$,
$$\mu^n(f>m_{f}+r+r_{o})\leq e^{-r^2/C},\quad r\geq 0,$$
for all positive integer $n$ and all $1$-Lipschitz function $f:\X^n\to\R$ with median $m_{f}$.
\end{cor}
\proof According to the tensorization property, $\mu^n$ verifies
the inequality $\T_{2}(C)$ on $\X^n$ equipped with the metric
$d_{2}$ defined above. It follows from Jensen inequality that
$\mu^n$ also verifies the inequality
$\left(\mathcal{T}_{d_{2}}\right)^2\leq CH$. Theorem \ref{Marton}
and Proposition \ref{resL-14} then give the conclusion.
\endproof
\begin{rem}
So, as was already emphasized at Section \ref{sec-overview}, when
$\mu$ verifies $\T_{2}$, it verifies a dimension-free Gaussian
concentration inequality. Dimension-free means that the
concentration inequality does not depend explicitly on $n$. This
independence on $n$ corresponds to an optimal behavior. Indeed,
the constants in concentration inequalities cannot improve when
$n$ grows.
\end{rem}

More generally,  the following proposition explains what kind of
concentration inequalities can be derived from a transport-entropy inequality.

\begin{prop}\label{concentration produit}
Let $\mu$ be a probability measure on $\X$ satisfying the inequality $\alpha(\mathcal{T}_{\theta(d)})\leq H$, where the function $\theta$ is convex and such that $\sup_{t>0} \theta(2t)/\theta(t)<+\infty.$\\

Then for all $\lambda\in (0,1)$, there is some constant
$a_{\lambda}>0$ such that
$$\inf_{\pi} \int_{\X\times\X} \theta\left(\frac{d(x,y)}{\lambda}\right)\,d\pi(x,y) \leq a_{\lambda}\alpha^{-1}\left( H(\nu | \mu)\right),$$
where the infimum is over the set of couplings of $\nu$ and $\mu$.\\

Furthermore, the product probability measure $\mu^n$ on $\X^n$, $n\geq 1$ satisfies the following concentration property.\\
For all $A\subset \X^n$ such that $\mu^n(A)\geq 1/2$,
$$\mu^n(\mathrm{enl}_{\theta}(A,r))\geq 1- \exp\left(-n\alpha\left(\frac{r-r_{o}^n(\lambda)}{n\lambda a_{\lambda}}\right)\right),\quad r\geq r_{o}^n(\lambda),\quad \lambda\in (0,1),$$
where
\begin{equation}\label{enlargement}
\mathrm{enl}_{\theta}(A,r)=\left\{x \in \X^n; \inf_{y\in A}\sum_{i=1}^n \theta (d(x_{i},y_{i}))\leq r\right\},
\end{equation}
and
$$r_{o}^n(\lambda)=(1-\lambda)a_{1-\lambda}n\alpha^{-1}\left(\frac{\log 2}{n}\right).$$
\end{prop}
The proof can be easily adapted from \cite[Proposition 3.4]{G09}.

\begin{rem}\
\begin{enumerate}
\item It is not difficult to check that $a_{\lambda}\to+\infty$
when $\lambda\to 0$.
  \item  If $\alpha$ is linear, the right-hand side does not depend explicitly on $n$. In this case, the concentration inequality is dimension-free.
  \item For example, if $\mu$ verifies $\mathbb{T}_{2}(C)$ (which corresponds to $\alpha(t)=t/C$ and $\theta(t)=t^2$), then one can take $a_{\lambda}=\frac{1}{\lambda^2}$. Defining as before $A^r=\{x\in \X^n; \inf_{y\in A} d_{2}(x,y)\leq r\}$ where $d_{2}(x,y)=\left(\sum_{i=1}^nd(x_{i},y_{i})^2\right)^{1/2}$ and optimizing over $\lambda\in(0,1)$, yields
$$\mu^n(A^r)\geq 1-e^{-\frac{1}{C}(r-r_{o})^2},\quad r\geq r_{o}=\sqrt{C\log 2}.$$
So we recover the dimension-free Gaussian inequality of Corollary
\ref{Marton T2}.
\end{enumerate}
\end{rem}

As we said above there are many ways of enlarging sets, and
consequently there many ways to describe the concentration of
measure phenomenon. In a series of papers \cite{T95, T96b, T96c}
Talagrand has deeply investigated the concentration properties of
product of probability measures. In particular, he has proposed
different families of enlargements which do not enter into the
framework of \eqref{enlargement}. In particular he has obtained
various concentration inequalities based on convex hull
approximation or $q$-points control, which have found numerous
applications (see \cite{Led} or \cite{T95}): deviation bounds for
empirical processes, combinatoric, percolation, probability on
graphs, etc.

The general framework is the following: One considers a product
space $\X^n$. For all $A\subset \X^n$, a function
$\varphi_{A}:\X^n\to[0,\infty)$ measures how far is the point
$x\in \X^n$ from the set $A$. The enlargement of $A$ is then
defined by $$\mathrm{enl}(A,r)=\{x\in \X^n; \varphi_{A}(x)\leq
r\},\quad r\geq0.$$

\subsection*{Convex hull approximation.} Define on $\X^n$ the
following weighted Hamming metrics:
$$d_{a}(x,y)=\sum_{i=1}^n a_{i}\mathbf{1}_{{x_{i}\neq y_{i}}},\quad x,y\in \X^n,$$
where $a\in ([0,\infty))^n$ is such that
$|a|=\sqrt{a_{1}^2+\cdots+ a_{n}^2}=1$. The function $\varphi_{A}$
is defined as follows: $$\varphi_{A}(x)=\sup_{|a|=1}d(x,A),\quad x\in
\X^n.$$ An alternative definition for $\varphi_{A}$ is the
following. For all $x\in \X^n$, consider the set
$$U_{A}(x)=\{(\mathbf{1}_{x_{1}\neq y_{1}},\ldots,\mathbf{1}_{x_{n}\neq y_{n}});
y\in A\},$$ and let $V_{A}(x)$ be the convex hull of $U_{A}(x)$.
Then it can be shown that
$$\varphi_{A}(x)=d(0,V_{A}(x)),$$
where $d$ is the Euclidean distance in $\R^n$.

A basic result related to convex hull approximation is the following theorem by Talagrand (\cite[Theorem 4.1.1]{T95}).
\begin{thm}
For every product probability measure $P$ on $\X^n$, and every $A\subset \X^n$,
$$\int e^{\varphi_{A}^2(x)/4}\,dP(x)\leq \frac{1}{P(A)}.$$
In particular,
$$P(\mathrm{enl}(A,r))\geq 1-\frac{1}{P(A)}e^{-r^2/4},\quad r\geq 0.$$
\end{thm}
This result admits many refinements (see \cite{T95}).

In \cite{M96b}, Marton developed transport-entropy inequalities to recover some of Talagrand's results on convex hull approximation.
To catch the Gaussian type concentration inequality stated in the above theorem, a natural idea would be to consider a $\T_2$ inequality with respect to the Hamming metric. In fact, it can be shown easily that such an inequality cannot hold. Let us introduce a weaker form of the transport-entropy inequality $\T_2$. Let $\X$ be some polish space, and $d$ a metric on $\X$; define
$$\widetilde{\mathcal{T}}_2(Q,R)=\inf _\pi\int_\X \left(\int_\X d(x,y)\, d\pi^y(x)\right)^2\,dR(y),\quad  Q,R\in \PX,$$
where the infimum runs over all the coupling $\pi$ of $Q$ and $R$ and where $\X\to\PX : y\mapsto \pi^y$ is a regular disintegration of $\pi$ given $y$: $$\int_{\X\times\X} f(x,y)\,d\pi(x,y)=\int_{\X}\left(\int_{\X} f(x,y)d\pi^y(x)\right)\,dR(y),$$
for all bounded measurable $f:\X\times\X\to\R.$

According to Jensen inequality,
$$\mathcal{T}_1(Q,R)^2\leq \widetilde{\mathcal{T}}_2(Q,R)\leq \mathcal{T}_2(Q,R).$$

One will says that $\mu\in \PX$ verifies the inequality $\widetilde{\mathbf{T}}_2(C)$ if
$$\widetilde{\mathcal{T}}_{2}(Q,R)\leq CH(Q|P)+CH(R|P),$$
for all probability measures $Q,R$ on $\X$.

The following theorem is due to Marton.
\begin{thm}
Every probability measure $P$ on $\X$, verifies the inequality $\widetilde{\mathbf{T}}_2(4)$ with respect to the Hamming metric. In other words,
$$\widetilde{\mathcal{T}}_{2}(Q,R)=\inf_\pi \int  \pi^y\{x; x_{i}\neq y_{i}\}^2\, dR(y)\leq 4H(Q|P)+4H(R|P),$$
for all probability measures $Q,R$ on $\X$.
\end{thm}
A proof of this result can be found in \cite{M96b} or in \cite{Led}.

Like $\T_2$, the inequality $\widetilde{\mathbf{T}}_2$ admits a
dimension-free tensorization property. A variant of Marton's
argument can be used to derive dimension-free concentration and
recover Talagrand's concentration results for the convex hull
approximation distance. We refer to \cite{M96b} and \cite[Chp
6]{Led} for more explanations and proofs.

\subsection*{Control by $q$-points.} Here the point of view is quite
different: $q\geq 2$ is a fixed integer and a point $x\in \X^n$
will be close from $A$ if it has many coordinates in common with
$q$ vectors of $A$. More generally, consider
$A_{1},\ldots,A_{q}\subset \X^n$; the function
$\varphi_{A_{1},\ldots,A_{q}}$ is defined as follows:
$$\varphi_{A_{1},\ldots,A_{q}}(x)=\inf_{y^{1}\in A_{1},\ldots,y^q \in A_{q}}\mathrm{Card}\left\{ i; x_{i}\not\in \{y_{i}^1,\ldots,y^q_{i}\}\right\}\}.$$
Talagrand's has obtained the following result (see \cite[Theorem 3.1.1]{T95} for a proof and further refinements).
\begin{thm}
For every product probability measure $P$ on $\X^n$, and every family $A_{1},\ldots,A_{q}\subset \X^n$, $q\geq 2$, the following inequality holds
$$\int q^{\varphi_{A_{1},\ldots,A_{q}}(x)}\,dP(x)\leq \frac{1}{P(A_{1})\cdots P(A_{q})}.$$
In particular, defining $\mathrm{enl}(A,r)=\{x\in \X^n; \varphi_{A,\ldots,A}(x)\leq r\}$, one gets
$$P(\mathrm{enl}(A,r))\geq 1-\frac{1}{q^rP(A)^r},\quad r\geq 0.$$
\end{thm}
In \cite{D96}, Dembo has obtained transport-entropy inequalities giving back Talagrand's results for $q$-points control. See also \cite{DZ96}, for related inequalities.

\section{Transport-entropy inequalities and large deviations}\label{section large deviations}

In \cite{GL07}, Gozlan and L\'eonard have proposed an interpretation of
transport-entropy inequalities in terms of large deviations
theory. To expose this point of view, let us introduce some
notation. Suppose that $(X_{n})_{n\geq 1}$ is a sequence of
independent and identically distributed $\X$ valued random
variables with common law $\mu$. Define their empirical measure
$$L_{n}=\frac{1}{n}\sum_{i=1}^n \delta_{X_{i}},$$ where $\delta_{a}$
stands for the Dirac mass at point $a\in \X$. Let
$\mathcal{C}_{b}(\X)$ be the set of all bounded continuous
functions on $\X.$ The set of all Borel probability measures on
$\X$, denoted by $\PX$, will be endowed with the weak topology,
that is the smallest topology with respect to which all
functionals $\nu\mapsto \int_{\X} \varphi\,d\nu$ with $\varphi\in
\mathcal{C}_{b}(\X)$ are continuous. If $B\subset \X$, let us
denote $H(B|\mu)=\inf\{H(\nu|\mu); \nu\in B \}.$ According to a
famous theorem of large deviations theory (Sanov's theorem), the
relative entropy functional governs the asymptotic behavior of
$\P(L_{n}\in A)$, $A\subset \X$ when $n$ goes to $\infty$.

\begin{thm}[Sanov's theorem]\label{Sanov}
For all $A\subset \PX$ measurable with respect to the Borel
$\sigma$-field,
$$
-H(\mathrm{int}(A)| \mu)\leq \liminf_{n
\to+\infty}\frac{1}{n}\log \P\left(L_{n}\in A\right)\leq
\limsup_{n \to+\infty}\frac{1}{n}\log \P\left(L_{n}\in
A\right)\leq -H(\mathrm{cl}(A)|\mu),
$$
where $\mathrm{int}(A)$ denotes the interior of $A$ and
$\mathrm{cl}(A)$ its closure (for the weak topology).
\end{thm}
For a proof of Sanov's theorem, see \cite[Thm 6.2.10]{DZ}.

Roughly speaking, $\P(L_{n}\in A)$ behaves like $e^{-nH(A|
\mu)}$ when $n$ is large. We write the statement of this theorem:
$\P(L_{n}\in A)\underset{n\rightarrow\infty}{\asymp} e^{-nH(A|
\mu)}$ for short.

Let us explain the heuristics upon which rely \cite{GL07} and also
the articles \cite{G09,GLWY09}. To interpret the transport-entropy
inequality $\alpha\left(\mathcal{T}_{c}\right)\leq H$, let us
define $A_{t}=\{\nu \in \PX; \mathcal{T}_{{c}}(\nu,\mu)\geq t\}$,
for all $t\geq 0$. Note that the transport-entropy inequality can
be rewritten as $\alpha(t)\leq H(A_{t}|\mu)$, $t\geq 0$. But,
according to Sanov's theorem,
$$\P(\mathcal{T}_{c}(L_{n},\mu)\geq t)=\P(L_{n}\in A_{t})\underset{n\rightarrow\infty}{\asymp} e^{-nH(A_{t}|\mu)}.$$
Consequently, the transport-entropy inequality
$\alpha\left(\mathcal{T}_{c}\right)\leq H$ is intimately linked to
the large deviation estimate
$$\limsup_{n\to+\infty}\frac{1}{n}\log \P(\mathcal{T}_{c}(L_{n},\mu)\geq t)\leq -\alpha(t),\quad t\geq 0.$$

Based on this large deviation heuristics, Gozlan and L\'eonard have
obtained in \cite{GL07} the following estimates for the deviation
of the empirical mean.
\begin{thm}\label{resL-10}
Let $\alpha$ be any function in $\mathcal{A}$ and assume that
$c(x,x)=0,$ for all $x.$ 
Define $\mathcal{U}^\forall_\mathrm{exp}(\mu):=\left\{u:\X\to\R,\textrm{
measurable}, \forall s>0, \int e^{s|u|}\,d\mu<\infty\right\}.$
It is supposed that $c$ is such that $u^{cc}$ and $u^c$ are
measurable functions for all $u\in
\mathcal{U}^\forall_\mathrm{exp}(\mu).$ This is the case in
particular if either $c=d$ is a \lsc\ metric cost or $c$ is
continuous. Then, the following statements are
equivalent.
\begin{enumerate}
    \item[(a)] The transport-entropy inequality
    $$\alpha(\mathcal{T}_c(\nu,\mu))\leq H(\nu|\mu),$$
    holds for all $\nu\in\PX$.
    \item[(b)] For all function $u\in \mathcal{U}^\forall_\mathrm{exp}(\mu)$, the inequality
    $$\logsup n\P\left(\IX u^{cc}\,dL_n+\IX u^c\,d\mu  \geq  r\right)
    \leq -\alpha(r),$$
    holds for all $r\geq 0$.
    \item[(c)] For all $u \in \mathcal{U^\forall}_\mathrm{exp}(\mu)$, the inequality 
    $$\frac{1}{n}\log\P\left(\IX u^{cc}\,dL_n+\IX u^c\,d\mu\geq  r\right)
    \leq -\alpha(r),$$
    holds for all positive integer $n$ and $r\geq 0.$
\end{enumerate}
 \end{thm}

Specializing to the situation where $c=d,$ since $u^{dd}=-u^d\in$
Lip(1), this means:
\begin{cor}[Deviation of the empirical mean]\label{resL-11}
Suppose that $\IX e^{s d(x_o,\,\cdot\,)}\,d\mu<\infty$ for some
$x_o\in\X$ and all $s>0.$ Then, the following statements are
equivalent.
\begin{enumerate}
    \item[(a)] The transport-entropy inequality 
    $$\alpha(W_1(\nu,\mu))\leq H(\nu|\mu),$$holds for all
    $\nu\in\PX.$
    \item[(b)] For all $u\in \mathrm{Lip(1)},$ the inequality
    $$\logsup n\P\left(\frac 1n\sum_{i=1}^nu(X_i)\geq  \IX u\,d\mu+r\right)
    \leq -\alpha(r),$$
    holds for all $r\geq 0.$ 
    \item[(c)] For all $u\in \mathrm{Lip(1)},$ the inequality
    $$\frac{1}{n}\log\P\left(\frac 1n\sum_{i=1}^nu(X_i)\geq  \IX u\,d\mu+r\right)
    \leq -\alpha(r),$$
holds for all positive integer $n$ and $r\geq 0.$
\end{enumerate}
\end{cor}

Sanov's Theorem and concentration inequalities match also well
together, since both give asymptotic results for probabilities of
events related to an i.i.d sequence.

In \cite{G09}, Gozlan has established the following
converse to Proposition \ref{concentration produit}:

\begin{thm}\label{Concentration->Transport}
Let $\mu$ be a probability measure on $\X$ and $(r_{o}^n)_{n}$ a
sequence of nonnegative numbers such that $r_{o}^n/n\to 0$ when
$n\to+\infty$. Suppose that for all integer $n$ the product
measure $\mu^n$ verifies the following concentration inequality:
\begin{equation}\label{eq Concentration->Transport}
\mu^n(\mathrm{enl}_{\theta}(A,r))\geq
1-\exp\left(-n\alpha\left(\frac{r-r_{o}^n}{n}\right)\right),\quad  r\geq r_{o}^n,
\end{equation}
for all $A\subset \X^n$ with $\mu^n(A)\geq 1/2$, where
$\mathrm{enl}_{\theta}(A,r)$ is defined in Proposition
\ref{concentration produit}. Then $\mu$ satisfies the
transport-entropy inequality
$\alpha(\mathcal{T}_{{\theta(d)}})\leq H$.
\end{thm}

Together with Proposition \ref{concentration produit}, this result
shows that the transport-entropy inequality $\alpha
\left(\mathcal{T}_{\theta(d)}\right)\leq H$ is an equivalent
formulation of the family of concentration inequalities \eqref{eq
Concentration->Transport}.

Let us emphasize a nice particular case.

\begin{cor}\label{T2=concentration}
Let $\mu$ be a probability measure on $\X$; $\mu$ enjoys the
Gaussian dimension-free concentration property if and only if
$\mu$ verifies Talagrand inequality $\T_{2}$. More precisely,
$\mu$ satisfies $\T_{2}(C)$ if and only if there is some $K>0$
such that for all integer $n$ the inequality
$$\mu^n(A^r)\geq 1-Ke^{-r^2/C},\quad r\geq 0,$$
holds for all $A\subset \X^n$ with $\mu^n(A)\geq 1/2$ and where
$A^r=\{x\in \X^n; \inf_{y\in A}d_{2}(x,A)\leq r\}$ and
$d_{2}(x,y)=\left( \sum_{i=1}^nd(x_{i},y_{i})^2\right)^{1/2}.$
\end{cor}
To put these results in perspective, let us recall that in recent
years numerous functional inequalities and tools were introduced
to describe the concentration of measure phenomenon. Besides
transport-entropy inequalities, let us mention other recent
approaches based on Poincar\'e inequalities \cite{GM83, BL97},
logarithmic Sobolev inequalities \cite{L96,BG99}, modified
logarithmic Sobolev inequalities \cite{BL97,BZ05,GGM05,BR08},
inf-convolution inequalities \cite{Mau91,LW08}, Beckner-Lata\l
a-Oleszkiewicz inequalities \cite{Beck89,LO00,BR03,BCR06}\ldots So
the interest of Theorem \ref{Concentration->Transport} is that it
tells that transport-entropy inequalities are the right point of
view, because they are equivalent to concentration estimates for
product measures.

\proof[Proof of Corollary \ref{T2=concentration}.] Let us show
that dimension-free Gaussian concentration implies Talagrand
inequality (the other implication is Corollary \ref{Marton T2}).
For every integer $n$, and $x\in \X^n$, define
$L_n^x=n^{-1}\sum_{i=1}^{n}\delta_{x_{i}}.$ The map $x\mapsto
W_2(L_n^x,\mu)$ is $1/\sqrt{n}$-Lipschitz with respect to the
metric $d_{2}$. Indeed, if $x=(x_{1},\ldots,x_{n})$ and
$y=(y_{1},\ldots,y_{n})$ are in $\X^n$, then the triangle
inequality implies that
\[ \left|W_2(L_n^x,\mu)-W_2(L_n^y,\mu)\right|\leq W_2(L_n^x,L_n^y).\]
According to the convexity property of
$\mathcal{T}_2(\,\cdot\,,\,\cdot\,)$ (see e.g \cite[Theorem
4.8]{Vill2}), one has
\[ \mathcal{T}_2(L_n^x,L_n^y)\leq \frac{1}{n}\sum_{i=1}^n \mathcal{T}_2(\delta_{x_{i}},\delta_{y_{i}})=\frac{1}{n}\sum_{i=1}^n d(x_{i},y_{i})^2=\frac{1}{n}d_2(x,y)^2,\]
which proves the claim.

Now, let $(X_i)_i$ be an i.i.d sequence of law $\mu$ and let $L_n$
be its empirical measure. Let $m_n$ be the median of
$W_2(L_n,\mu)$ and define $A=\left\{ x\in \X; W_2(L_n^x,\mu)\leq m_n\right\}$.
Then $\mu^n(A)\geq 1/2$ and it is easy to show that $A^r\subset
\left\{ x\in\X; W_2(L_n^x,\mu)\leq m_n+r/\sqrt{n}\right\}$. Applying the
Gaussian concentration inequality to $A$ gives
\[\P\left( W_2(L_n,\mu)> m_n+r/\sqrt{n}\right)\leq K\exp\left(-r^2/C\right),\quad r\geq 0.\]
Equivalently, as soon as $u\geq m_{n}$, one has
\[ \P\left( W_2(L_n,\mu)> u\right)\leq K\exp\left(-n(u-m_n)^2/C\right).\]
Now, it is not difficult to show that $m_{n}\to 0$ when
$n\to\infty$ (see the proof of \cite[Theorem 3.4]{G09}).
Consequently,
\[\limsup_{n\to+\infty}\frac{1}{n}\log \P\left( W_2(L_n,\mu)> u\right) \leq -u^2/C.\]
for all $u\geq 0.$

On the other hand, according to Sanov's Theorem \ref{Sanov},
\[ \liminf_{n\to+\infty}\dfrac{1}{n}\log \P\left( W_2(L_n,\mu)> u\right)\geq-\inf\left\{ H(\nu|\mu); \nu\in \mathrm{P}(\X) \text{ s.t. } W_2(\nu,\mu)>u \right\}.\]
This together with the preceding inequality yields
\[ \inf\left\{ H(\nu|\mu); \nu\in \mathrm{P}(\X) \text{ s.t. } W_2(\nu,\mu)>u\right\}\geq u^2/C\]
or in other words,
\[ W_2(\nu,\mu)^2\leq CH(\nu|\mu),\]
and this completes the proof.
\endproof

\section{Integral criteria}\label{section integral criteria}
Let us begin with a basic observation concerning the
integrability.
\begin{prop}\label{res-08}
Suppose that a probability measure $\mu$ on $\X$ verifies the
inequality $\alpha(\mathcal{T}_{\theta(d)})\leq H$, and let
$x_{o}\in \X$ ; then $\int _{\X}
\exp\left(\alpha\circ\theta(\varepsilon
d(x,x_{o}))\right)\,d\mu(x)$ is finite for all $\varepsilon>0$
small enough.
\end{prop}
\proof If $\mu$ verifies the inequality
$\alpha(\mathcal{T}_{\theta(d)})\leq H$, then according to Jensen
inequality, it verifies the inequality
$\alpha\circ\theta\left(\mathcal{T}_{d}\right)\leq H$ and
according to Theorem \ref{Marton}, the inequality
$$\mu(A^r)\geq 1-\exp(-\alpha\circ\theta(r-r_{o})),\
r\geq r_{o}=\theta^{-1}\circ\alpha^{-1}(\log 2),$$ holds for all
$A$ with $\mu(A)\geq 1/2$. Let $m$ be a median of the function
$x\mapsto d(x,x_{o})$ ; applying the previous inequality  to
$A=\{x\in \X; d(x,x_{o})\leq m\}$ yields
$$\mu(d(x,x_{o})> m+r)=\mu(\X\setminus A^r)\leq\exp(-\alpha\circ\theta(r-r_{o})),\quad r\geq r_{o}.$$
It follows easily that $\int
\exp\left(\alpha\circ\theta(\varepsilon
d(x,x_{o}))\right)\,d\mu(x)<+\infty$, if $\varepsilon$ is
sufficiently small.
\endproof

The theorem below shows that this integrability condition is also
sufficient when the function $\alpha$ is supposed to be
subquadratic near $0$.

\begin{thm}\label{Integral criteria}
Let $\mu$ be a probability measure on $\X$ and define
$\alpha^\circledast(s)=\sup_{t\geq 0}\{st-\alpha(t)\}$, for all
$s\geq 0$. If the function $\alpha$ is such that
$\limsup_{t\to0}\alpha(t)/t^2<+\infty$ and
$\sup\{\alpha^\circledast(t); t :
\alpha^\circledast(t)<+\infty\}=+\infty$, then the following
statements are equivalent:
\begin{enumerate}
\item There is some $a>0$ such that
$\alpha\left(a\mathcal{T}_{\theta(d)}(\nu,\mu)\right)\leq
H(\nu|\mu).$ \item There is some $b>0$ such that $\int_{\X\times\X}
e^{\alpha\circ\theta(bd(x,y))}\,d\mu(x)d\mu(y)<+\infty.$
\end{enumerate}
\end{thm}

Djellout, Guillin and Wu \cite{DGW04} were the first ones to
notice that the inequality $\T_{1}$ is equivalent to the
integrability condition $\int_{\X\times\X}
e^{bd(x,y)^2}\,d\mu(x)d\mu(y)<+\infty.$ After them, this
characterization was extended to other functions $\alpha$ and
$\theta$ by Bolley and Villani \cite{BV05}. Theorem \ref{Integral
criteria} is due to Gozlan \cite{G06}. Let us mention that the
constants $a$ and $b$ are related to each other in \cite[Theorem
1.15]{G06}.

Again to avoid technical difficulties, we are going to establish a
particular case of Theorem \ref{Integral criteria}:
\begin{prop}
If $M=\int_{\X\times\X} e^{b^2 \frac{d(x,y)^2}{2}}\,d\mu(x)d\mu(y)$ is finite
for some $b>0$, then $\mu$ verifies the following $\T_{1}$
inequality:
\[\mathcal{T}_{d}(\nu,\mu)\leq \frac{1}{b}\sqrt{1+2\log M}\sqrt{2H(\nu|\mu)},\]
for all $\nu\in \mathrm{P}(\X).$
\end{prop}
\proof First one can suppose that $b=1$ (if this is not the case,
just replace the distance $d$ by the distance $bd$). Let
$C=2(1+2\log(M))$; according to Corollary \ref{BG3}, it is
enough to prove that
\[\int e^{s f}\,d\mu\leq e^{\frac{s^2C}{4}},\quad s\geq 0\]
for all $1$-Lipschitz function with $\int f\,d\mu=0.$ Let $X, Y$
be two independent variables of law $\mu$; using Jensen
inequality, the symmetry of $f(X)-f(Y)$, the inequality $(2i)!\geq
2^i\cdot i!$ and the fact that $f$ is $1$-Lipschitz, one gets
\begin{align*}
\E\left[e^{sf(X)}\right]&\leq \E\left[e^{s(f(X)-f(Y))}\right]=\sum_{i=0}^{+\infty} \frac{s^{2i}\E \left[(f(X)-f(Y))^{2i}\right]}{(2i)!}\\
&\leq \sum_{i=0}^{+\infty} \frac{s^{2i}\E
\left[d(X,Y)^{2i}\right]}{2^i\cdot i!} =\E\left[\exp
\left(\frac{s^2d(X,Y)^2}{2}\right)\right].
\end{align*}
So, for $s\leq 1$, Jensen inequality gives $\E\left[e^{sf(X)}\right]\leq
M^{s^2}.$ If $s\geq 1$, then Young inequality implies
$\E\left[e^{sf}\right]\leq \E\left[e^{s(f(X)-f(Y))}\right] \leq
e^{\frac{s^2}{2}}M.$ So in all cases, $\E\left[e^{sf}\right]\leq
e^{\frac{s^2}{2}}M^{s^2}=e^{s^2C/4}$ which completes the proof.
\endproof

\section{Transport inequalities with uniformly convex potentials}\label{section uniformly convex}

This section is devoted to some results which have been proved by
Cordero-Erausquin in \cite{C02} and  Cordero-Erausquin,
Gangbo and Houdr\'e in \cite{CGH04}.

Let us begin with a short overview of \cite{C02}. The state space
is $\X=\Rk.$ Let $V:\Rk\to\R$ be a  function of class $\mathcal{C}^2$
which is semiconvex, i.e.\ $\mathrm{Hess}_x\,V\ge\kappa\mathrm{Id}$
for all $x,$ for some real $\kappa.$ If $\kappa>0,$ the potential
$V$ is said to be uniformly convex. Define
$$d\mu(x):=e^{-V(x)}\,dx$$ and assume that $\mu$ is a probability
measure. The main result of \cite{C02} is the following

\begin{thm}\label{res-07}
Let $f,g$ be nonnegative compactly supported functions with $f$ of
class $\mathcal{C}^1$ and $\int f\,d\mu=\int g\,d\mu=1.$ If
$T(x)=x+\nabla\theta(x)$ is the Brenier map pushing forward $f\mu$
to $g\mu$ (see Theorem \ref{res-06}), then
\begin{equation}\label{eq-04}
    H(g\mu|\mu)\ge H(f\mu|\mu)+\int_\Rk\nabla
    f\cdot\nabla\theta\,d\mu+\frac \kappa 2\int_\Rk|\nabla\theta|^2\,fd\mu.
\end{equation}
\end{thm}

Before presenting a sketch of the proof of this result, let us
make a couple of comments.
\begin{itemize}
    \item[-] This result is an extension of Talagrand inequality
    \eqref{eq-12}.
    \item[-] About the regularity of $\theta.$ As a
convex function, $\theta$ is differentiable almost everywhere and
it admits a Hessian in the sense of Alexandrov  almost everywhere
(this is the statement of Alexandrov's theorem). A function
$\theta$ admits a Hessian in the sense of Alexandrov at $x\in\Rk$
if it is differentiable at $x$ and there exists a symmetric linear
map $H$ such that
$$
\theta(x+u)=\theta(x)+\nabla\theta(x)\scal u+\frac 12 Hu\scal
u+o(|u|^2).
$$
As a definition, this linear map $H$ is the Hessian in the sense
of Alexandrov of $\theta$ at $x$ and it is denoted
$\mathrm{Hess}_x\,\theta.$ Its trace is called the Laplacian in
the sense of Alexandrov and is denoted $\Delta_A\theta(x).$
\end{itemize}

\begin{proof}[Outline of the proof]
The change of variables formula leads us to the Monge-Amp\`ere
equation
\begin{equation*}
    f(x)e^{-V(x)}=g(T(x))e^{-V(T(x))}
    \det(\mathrm{Id}+\mathrm{Hess}_x\,\theta)
\end{equation*}
Taking the logarithm, we obtain
\begin{equation*}
    \log g(T(x))=\log
    f(x)+V(x+\nabla\theta(x))-V(x)-\log\det(\mathrm{Id}+\mathrm{Hess}_x\,\theta).
\end{equation*}
Our assumption on $V$ gives us $V(x+\nabla\theta(x))-V(x)\ge
\nabla V(x)\cdot\nabla\theta(x)+\kappa|\nabla\theta|^2/2.$ Since
$\log(1+t)\le t,$ we  have also $\log\det
(\mathrm{Id}+\mathrm{Hess}_x\,\theta)\le \Delta_A\theta(x)$ where
$\Delta_A\theta$ stands for the Alexandrov Laplacian. This
implies that $f\mu$-almost everywhere
\begin{equation*}
    \log g(T(x))\ge\log
    f(x)+\nabla V(x)\scal\nabla\theta(x)-\Delta_A\theta(x)+\kappa|\nabla\theta|^2/2
\end{equation*}
and integrating
\begin{equation*}
    \int_\Rk\log g(T)\,fd\mu\ge\int_\Rk f\log
    f\,d\mu+\int_\Rk [\nabla V\scal\nabla\theta-\Delta_A\theta]\,fd\mu+\frac\kappa 2\int_\Rk|\nabla\theta|^2\,fd\mu
\end{equation*}
Integrating by parts (at this point, a rigorous proof necessitates
to take account of the \textit{almost everywhere} in the
definition of $\Delta_A$), we obtain
\begin{equation*}
    H(g\mu|\mu)\ge H(f\mu|\mu)+\int_\Rk\nabla\theta\scal\nabla f \,d\mu+\frac\kappa 2\int_\Rk|\nabla\theta|^2\,fd\mu
\end{equation*}
which is the desired result.
\end{proof}

Next results are almost immediate corollaries of this theorem.

\begin{cor}[Transport inequality]
If $V$ is of class $\mathcal{C}^2$ with $\mathrm{Hess}\,V\ge\kappa
\mathrm{Id}$ and $\kappa>0$, then the probability measure
$d\mu(x)=e^{-V(x)}\,dx$ satisfies the transport inequality
$\T_2(2/\kappa)$:
\begin{equation*}
   \frac\kappa 2 W_2^2(\nu,\mu)\le H(\nu|\mu),
 \end{equation*}
for all $\nu\in\mathrm{P}(\Rk).$
\end{cor}

\begin{proof}[Outline of the proof]
Plug $f=1$ into \eqref{eq-04}.
\end{proof}

This transport inequality extends Talagrand's $\T_2$-inequality
\cite{T96}.

In \cite{FeyUstu04a}, Feyel and \"Ust\"unel have derived another
type of extension of $\T_2$ from the finite dimension setting to
an abstract Wiener space. Their proof is based on  Girsanov
theorem.

Next result  is the well-known Bakry-Emery criterion for the
logarithmic Sobolev inequality \cite{BE85}.

\begin{cor}[Logarithmic Sobolev inequality]\label{resL-d}
If $V$ is of class $\mathcal{C}^2$ with $\mathrm{Hess}\,V\ge\kappa
\mathrm{Id}$ and $\kappa>0$, then the probability measure
$d\mu(x)=e^{-V(x)}\,dx$ satisfies the logarithmic Sobolev
inequality $\LSI(2/\kappa)$ (see Definition \ref{def-PLS} below):
\begin{equation*}
    H(f\mu|\mu)\le \frac{2}{\kappa} \IX |\nabla \sqrt f|^2\,d\mu
\end{equation*}
for all sufficiently regular $f$ such that
$f\mu\in\mathrm{P}(\Rk).$
\end{cor}

\begin{proof}[Outline of the proof]
Plugging $g=1$ into \eqref{eq-04} yields
\begin{equation}\label{eq-05}
    H(f\mu|\mu)\le -\int_\Rk\nabla f\cdot\nabla\theta\,d\mu-\frac
\kappa 2\int_\Rk|\nabla\theta|^2\,fd\mu
\end{equation}
where $T(x)=x+\nabla(x)$ is the Brenier map pushing forward $f\mu$
to $\mu.$ Since $\nabla\theta$ is unknown to us, we are forced to
optimize as follows
\begin{equation*}
    H(f\mu|\mu)\le \sup_{\nabla\theta}\left\{ -\int_\Rk\nabla
    f\cdot\nabla\theta\,d\mu-\frac \kappa
    2\int_\Rk|\nabla\theta|^2\,fd\mu\right\}=\frac{2}{\kappa}\IDV(f\mu|\mu),
\end{equation*}
which is the desired inequality.
\end{proof}

The next inequality has been discovered by Otto and Villani
\cite{OV00}. It will be used in Section \ref{section functional
inequalities} for comparing transport and logarithmic Sobolev
inequalities. More precisely, Otto-Villani's Theorem \ref{OV}
states that if $\mu$ satisfies the logarithmic Sobolev inequality,
then it satisfies $\T_2.$

Let us define the (usual) Fisher information with respect to $\mu$
by
\begin{equation*}
    \IF(f|\mu)=\int|\nabla\log f|^2\,fd\mu
\end{equation*}
for all positive and sufficiently smooth function $f.$

\begin{cor}[HWI inequality]\label{HWI}
If $V$ is of class $\mathcal{C}^2$ with $\mathrm{Hess}\,V\ge\kappa \mathrm{Id}$ for
some \emph{real} $\kappa,$ the probability measure
$d\mu(x)=e^{-V(x)}\,dx$ satisfies the HWI inequality
\begin{equation*}
    H(f\mu|\mu)\le
    W_2(f\mu,\mu)\sqrt{\IF(f|\mu)}-\frac{\kappa}{2}W_2^2(f\mu,\mu)
\end{equation*}
for all nonnegative smooth compactly supported function $f$ with
$\int_\Rk f\,d\mu=1.$
\end{cor}
Note that the HWI inequality gives back the celebrated Bakry-Emery
criterion.

\begin{proof}[Outline of the proof]
Start from \eqref{eq-05}, use $W_2^2(f\mu,\mu)=
\int_\Rk|\nabla\theta|^2\,fd\mu$ and
\begin{align*}
   -\int \nabla\theta\cdot\nabla
f\,d\mu&=-\int \nabla\theta\cdot\nabla\log f\,fd\mu \\
  &\le
\left(\int|\nabla\theta|^2\,fd\mu\int|\nabla\log f|^2\,fd\mu
\right)^{1/2}=W_2(f\mu|\mu)\sqrt{\IF(f|\mu)},
\end{align*}
and here you are.
\end{proof}

Now, let us have a look at the results of \cite{CGH04}. They
extend Theorem \ref{res-07} and its corollaries. Again, the state
space is $\X=\Rk$ and the main ingredients are
\begin{itemize}
     \item An entropy profile: $r\in[0,\infty)\mapsto \mathrm{s}(r)\in\R;$
    \item A cost function: $v\in\Rk\mapsto c(v)\in[0,\infty)$;
    \item A potential: $x\in\Rk\mapsto V(x)\in\R$.
\end{itemize}
The framework of our previous Theorem \ref{res-07} corresponds to
the entropy profile $\mathrm{s}(r)=r\log r-r$ and the quadratic
transport cost $c(y-x)=|y-x|^2/2.$

We are only interested in probability measures
$d\rho(x)=\rho(x)\,dx$ which are absolutely continuous and we
identify $\rho$ and its density. The \emph{free energy functional}
is
\begin{equation*}
    F(\rho):=\int_\Rk [\mathrm{s}(\rho)+\rho V](x)\,dx
\end{equation*}
and our reference measure $\mu$ is the \emph{steady state}:  the
unique minimizer of $F.$ Since $\mathrm{s}$ will be assumed to be
strictly convex, $\mu$ is the unique solution of
\begin{equation}\label{eqL-02}
    \mathrm{s}'(\mu)=-V,
\end{equation}
which, by \eqref{eq-Fenchel} is
$$
\mu=\mathrm{s}^{*\prime}(-V).
$$
As $\mathrm{s}(\rho)+\mathrm{s}^*(-V)\ge -V\rho,$ see \eqref{ineq-Fenchel}, in
order that $F$ is a well-defined $(-\infty,\infty]$-valued
function, it is enough to assume that $\int_\Rk
\mathrm{s}^*(-V)(x)\,dx<\infty.$ One also requires that $\int_\Rk
\mathrm{s}^{*\prime}(-V)(x)\,dx=1$ so that $\mu$ is a probability
density.
\\
The free energy is the sum of the \emph{entropy} $S(\rho)$ and the
\emph{internal energy} $U(\rho)$ which are defined by
\begin{equation*}
  S(\rho) := \int_\Rk \mathrm{s}(\rho)(x)\,dx, \qquad
  U(\rho) := \int_\Rk V(x)\rho(x)\,dx.
\end{equation*}
It is assumed that
\begin{enumerate}
    \item[(A$_\mathrm{s}$)] \begin{enumerate}
        \item $\mathrm{s}\in \mathcal{C}^2(0,\infty)\cap \mathcal{C}([0,\infty))$ is strictly convex, $\mathrm{s}(0)=0,$ $\mathrm{s}'(0)=-\infty$  and
        \item $r\in(0,\infty)\mapsto r^d\mathrm{s}(r^{-d})$ is convex increasing;
    \end{enumerate}
         \item[(A$_c$)] $c$ is convex, of class $\mathcal{C}^1,$ even,  $c(0)=0$ and $\lim_{|v|\rightarrow\infty}c(v)/|v|=\infty;$
    \item[(A$_V$)] For some  real number $\kappa,$ $V(y)-V(x)\ge\nabla V(x)\scal (y-x)+\kappa c(y-x),$ for all $x,y.$
\end{enumerate}
If $\kappa>0,$ the potential $V$ is said to be uniformly
$c$-convex.

We see with Assumption (A$_V$) that the cost function $c$ is a
tool for quantifying the curvature of the potential $V.$ Also note
that if  $\kappa>0$ and $c(y-x)=\|y-x\|^p$ for some $p$ in
(A$_V$), letting $y$ tend to $x,$  one sees that it is necessary
that $p\ge2.$

The  transport cost associated with $c$ is
$\mathcal{T}_c(\rho_0,\rho_1).$ Theorem \ref{res-06} admits an
extension to the case of strictly convex transport cost $c(y-x)$
(instead of the quadratic cost). Under the assumption (A$_c$) on
$c,$ if the transport cost $\mathcal{T}_c(\rho_0,\rho_1)$ between
two absolutely continuous probability measures $\rho_0$ and
$\rho_1$ is finite, there exists a unique (generalized) Brenier
map $T$ which pushes forward $\rho_0$ to $\rho_1$ and it is
represented by
$$
T(x)=x+\nabla c^*(\nabla\theta(x))
$$
for some function $\theta$  such that
$\theta(x)=-\inf_{y\in\Rk}\{c(y-x)+\eta(y)\}$ for some function
$\eta.$ This has been proved by Gangbo and McCann in
\cite{GMC96} and $T$ will be named later the Gangbo-McCann map.

The fundamental result of  \cite{CGH04} is the following extension
of Theorem \ref{res-07}.

\begin{thm}
For any $\rho_0,\rho_1$ which are compactly supported and such
that $\mathcal{T}_c(\rho_0,\rho_1)<\infty,$ we have
\begin{equation}\label{eqL-h}
    F(\rho_1)-F(\rho_0)\ge\kappa\mathcal{T}_c(\rho_0,\rho_1)+
    \int_\Rk
    (T(x)-x)\cdot\nabla[\mathrm{s}'(\rho_0)-\mathrm{s}'(\mu)](x)\,\rho_0(x)\,dx
\end{equation}
where $T$ is the Gangbo-McCann map which pushes forward $\rho_0$
to $\rho_1.$
\end{thm}

\begin{proof}[Outline of the proof]
Let us first have a look at $S.$ Thanks to the assumption  (A$_s$)
one can prove that $S$ is displacement convex. This formally means
that if $\mathcal{T}_c(\rho_0,\rho_1)<\infty,$
\begin{equation*}
    S(\rho_1)-S(\rho_0)\ge \frac{d}{dt}S(\rho_t)_{|t=0}
\end{equation*}
where $(\rho_t)_{0\le t\le1}$ is the displacement interpolation of
$\rho_0$ to $\rho_1$ which is defined by
$$\rho_t:=[(1-t)\mathrm{Id}+tT]_\# \rho_0,\quad0\le t\le1$$ where $T$
pushes forward $\rho_0$ to $\rho_1.$ Since $
\frac{\partial\rho}{\partial
    t}(t,x)+\nabla\cdot\left[\rho(t,x)(T(x)-x)\right]=0,$ another way of writing
this convex inequality is
\begin{equation}\label{eqL-01}
    S(\rho_1)-S(\rho_0)\ge \int_\Rk
    (T(x)-x)\cdot\nabla[\mathrm{s}'(\rho_0)](x)\,\rho_0(x)\,dx.
\end{equation}
Note that when the cost is quadratic, by Theorem \ref{res-06} the
Brenier-Gangbo-McCann map between the uniform measures $\rho_0$
and $\rho_1$ on the balls $B(0,r_0)$ and $B(0,r_1)$ is given by
$T= (r_1/r_0) \mathrm{Id}$ so that the image $\rho_t={T_t}_\#\rho_0$
of $\rho_0$ by the displacement $T_t=(1-t)\mathrm{Id}+tT$ at time
$0\le t\le1$ is the uniform measure on the ball $B(0,r_t)$ with
$r_t=(1-t)r_0+tr_1.$ Therefore, $t\in[0,1]\mapsto
S(\rho_t)=r_t^d\mathrm{s}(r_t^{-d})$ is convex for all $0<r_0\le r_1$ if
and only if $r^d\mathrm{s}(r^{-d})$ is convex: i.e.\ assumption
(A$_s$-b).
\\
It is immediate that under the assumption (A$_V$) we have
\begin{equation*}
    U(\rho_1)-U(\rho_0)\ge \int_\Rk \nabla V(x)\scal
    [T(x)-x]\,\rho_0(x)dx+\kappa\mathcal{T}_c(\rho_0,\rho_1).
\end{equation*}
One can also prove that this is a necessary condition for
assumption (A$_V$) to hold true. Summing  \eqref{eqL-01} with this
inequality, and taking \eqref{eqL-02} into account, leads us to
\eqref{eqL-h}.
\end{proof}

Let us define the \emph{generalized relative entropy}
$$
S(\rho|\mu):=F(\rho)-F(\mu)=\int_\Rk
[\mathrm{s}(\rho)-\mathrm{s}(\mu)-\mathrm{s}'(\mu)(\rho-\mu)](x)\,dx
$$
on the set $\Pac(\Rk)$ of all absolutely continuous probability
measures on $\Rk.$ It is a $[0,\infty]$-valued convex function
which admits $\mu$ as its unique minimum.

\begin{thm}[Transport-entropy inequality]
Assume that the constant $\kappa$ in assumption \emph{(A$_V$)} is
positive: $\kappa>0.$ Then, $\mu$ satisfies the following
transport-entropy inequality
\begin{equation*}
    \kappa\mathcal{T}_c(\rho,\mu)\le
    S(\rho|\mu),
    \end{equation*}
for all $\rho\in\Pac(\Rk).$
\end{thm}

\begin{proof}[Outline of the proof]
If $\rho$ and $\mu$ are compactly supported, plug $\rho_0=\mu$ and
$\rho_1=\rho$ into \eqref{eqL-h} to obtain the desired result.
Otherwise, approximate $\rho$ and $\mu$ by compactly supported
probability measures.
\end{proof}

Let us define the \emph{generalized relative Fisher information}
for all $\rho\in\Pac(\Rk)$ by
\begin{equation*}
    \mathcal{I}(\rho|\mu):=\int_\Rk \kappa\, c^*\Big(-\kappa^{-1}\nabla[\mathrm{s}'(\rho)-\mathrm{s}'(\mu)](x)\Big)\,d\rho(x)\in [0,\infty]
\end{equation*}
with $\mathcal{I}(\rho|\mu)=\infty$ if $ \nabla\rho$ is undefined
on a set with positive Lebesgue measure.

\begin{thm}[Entropy-information inequality]
Assume that $\kappa>0.$ Then, $\mu$ satisfies the following
entropy-information inequality
\begin{equation*}
    S(\rho|\mu)\le \mathcal{I}(\rho|\mu),
   \end{equation*}
for all $\rho\in\Pac(\Rk).$
\end{thm}

\begin{proof}[Outline of the proof]
Change $c$ into $\kappa c$ so that the constant $\kappa$ becomes
$\kappa=1.$ If $\rho$ and $\mu$ are compactly supported, plug
$\rho_0=\rho$ and $\rho_1=\mu$ into \eqref{eqL-h} to obtain
\begin{align*}
  F(\rho)-F(\mu)+\mathcal{T}_c(\rho,\mu)
  &\le \int_\Rk (T(x)-x)\cdot\nabla[\mathrm{s}'(\mu)-\mathrm{s}'(\rho)](x)\,d\rho(x) \\
  &\le \int_\Rk c(T(x)-x)\,d\rho(x) + \int_\Rk c^*\Big(\nabla[\mathrm{s}'(\mu)-\mathrm{s}'(\rho)](x)\Big)\,d\rho(x)
\end{align*}
where the last inequality is a consequence of Fenchel inequality
\eqref{ineq-Fenchel}. Since $T$ is the Gangbo-McCann map between
$\rho$ and $\mu,$ we have $\int_\Rk
c(T(x)-x)\,d\rho(x)=\mathcal{T}_c(\rho,\mu).$  Therefore,
\begin{equation*}
    F(\rho)-F(\mu)
    \le \int_\Rk c^*\Big(-\nabla[\mathrm{s}'(\rho)-\mathrm{s}'(\mu)](x)\Big)\,d\rho(x)
\end{equation*}
which is the desired result for compactly supported measures. For
the general case, approximate $\rho$ and $\mu$ by compactly
supported probability measures.
\end{proof}

A direct application of these theorems allow us to recover
inequalities which were first proved by Bobkov and Ledoux in
\cite{BL00}.
\begin{cor}
Let $\|\cdot\|$ be a norm on $\Rk$ and $V$ a convex potential.
Suppose that $V$ is uniformly $p$-convex with respect to
$\|\cdot\|$ for some $p\ge2;$ this means that there exists a
constant $\kappa>0$ such that for all $x,y\in\Rk$
\begin{equation}\label{eq-06}
    V(x)+V(y)-2V\left(\frac{x+y}{2}\right)\ge
    \frac{\kappa}{p}\|y-x\|^p.
\end{equation}
Denote $d\mu(x):=e^{-V(x)}\,dx$ (where it is understood that $\mu$
is a probability measure) and $H(\rho|\mu)$ the usual relative
entropy.
\begin{enumerate}
    \item $\mu$ verifies the transport-entropy inequality
\begin{equation*}
    \kappa\mathcal{T}_{c_p}(\rho,\mu)\le
    H(\rho|\mu),
    \end{equation*}
for all $\rho\in\Pac(\Rk)$, where $c_p(y-x)=\|y-x\|^p/p.$
    \item  $\mu$ verifies the entropy-information inequality
\begin{equation*}
    H(f\mu|\mu)\le \frac{1}{q\kappa^{q-1}}\int_\Rk\left\|\nabla\log
    f\right\|_*^q\,fd\mu
\end{equation*}
for all smooth nonnegative function $f$ such that $\int_\Rk
f\,d\mu=1,$ where $\|\cdot\|_*$ is the dual norm of $\|\cdot\|$
and $1/p+1/q=1.$
\end{enumerate}
\end{cor}

\section{Links with other functional inequalities}\label{section functional inequalities}

In this section, we investigate the position of transport-entropy
inequalities among the relatively large class of functional
inequalities (mainly of Sobolev type) appearing in the literature.
We will be concerned with transport-entropy inequalities of the
form $\mathcal{T}_{\theta(d)}\leq H$ since inequalities of the
form $\alpha\left(\mathcal{T}_{\theta(d)}\right)\leq H$ with a
nonlinear function $\alpha$ are mainly described in terms of
integrability conditions according to Theorem \ref{Integral
criteria}.

\subsection{Links with the Property $(\tau)$}

In \cite{Mau91}, Maurey introduced the Property $(\tau)$ which
we describe now. Let $c$ be a cost function on $\X$. Recall that for all $f\in \BX$, the function $Q_c f$ is defined by $Q_cf(x):=\inf_{y}\{f(y)+c(x,y)\}.$ If a probability measure $\mu\in
\PX$  satisfies
\begin{equation*}\tag{$\tau$}
    \left(\int e^{Q_c f}\,d\mu\right)\left( \int e^{-f}\,d\mu\right)\le
    1,
\end{equation*}
for all $f\in\BX$, one says that the couple $(\mu,c)$ satisfies the Property
$(\tau).$

The basic properties of this class of functional inequalities are summarized in the following result.
\begin{prop}
Suppose that $(\mu,c)$ verifies the Property $(\tau)$, then the following holds.
\begin{enumerate}
\item The probability measure $\mu$ verifies the transport-entropy
inequality $\mathcal{T}_c \leq H$. \item For all positive integer
$n$, the couple $(\mu^n,c^{\oplus n})$, with $c^{\oplus
n}(x,y)=\sum_{i=1}^nc(x_i,y_i)$, $x,y\in \X^n$, verifies the
Property $(\tau)$. \item For all positive integer $n$, and all
Borel set $A\subset \X^n$ with $\mu^n(A)>0$,
$$\mu^n(\mathrm{enl}_{c^{\oplus n}}(A,r))\geq 1 - \frac{1}{\mu^n(A)} e^{-r},\quad r\geq 0$$
where
$$\mathrm{enl}_{c^{\oplus n}}(A,r)=\left\{x \in \X^n ; \inf_{y\in A}\sum_{i=1}^n c(x_i,y_i)\leq r\right\}$$
\end{enumerate}
\end{prop}
The third point of the proposition above was the main motivation
for the introduction of this class of inequalities. In
\cite{Mau91}, Maurey established that the symmetric exponential
probability measure $d\nu(x)=\frac{1}{2}e^{-|x|}\,dx$ on $\R$
satisfies the Property $(\tau)$ with the cost function
$c(x,y)=a\min(|x|^2,|x|)$ for some constant $a>0$. It enables him
to recover Talagrand's concentration results for the
multidimensional exponential distribution with sharp constants.
Moreover, using the Prekopa-Leindler inequality (see Theorem
\ref{PL} below), he showed that the standard Gaussian measure
$\gamma$ verifes the Property $(\tau)$ with the cost function
$c(x,y)=\frac{1}{4}|x-y|^2$. The constant $1/4$ is sharp.

\proof (1) According to the dual formulation of transport-entropy
inequalities, see Corollary \ref{BG2}, $\mu$ verifies the
inequality $\mathcal{T}_c\leq H$ if and only if $\int e^{Q_c
f}\,d\mu\cdot e^{-\int f\,d\mu}\leq 1$ for all $f\in
\mathcal{B}_b(\X).$ Jensen inequality readily implies that this
condition is weaker than the Property $(\tau)$.
\\
(2) The proof of this tensorization property follows the lines of
the proof of Theorem \ref{tensorization}. We refer to \cite{Mau91}
or \cite{Led} for a complete proof.
\\
(3) Applying the Property $(\tau)$ satisfied by $(\mu^n,c^{\oplus
n})$ to the function $u(x)=0$ if $x\in A$ and $u(x)=t$, if
$x\notin A$ and letting $t\to\infty$ yields the inequality:
$$\int e^{\inf_{y\in A} c^{\oplus n}(x,y)}\,d\mu^n(x)\leq \frac{1}{\mu^n(A)},$$
which immediately implies the concentration inequality.
\endproof

In fact, the Property $(\tau)$ can be viewed as a symmetric transport-entropy inequality.
\begin{prop}
The couple $(\mu,c)$ verifies the Property $(\tau)$ if and only if
$\mu$ verifies the following inequality
$$
    \mathcal{T}_c(\nu_1,\nu_2) \leq H(\nu_1| \mu) + H(\nu_2| \mu),
$$
for all $\nu_1,\nu_2 \in \mathrm{P}(\X)$.
\end{prop}

\proof According to the Kantorovich dual equality, see Theorem
\ref{res-01}, the symmetric transport-entropy inequality holds if
and only if for all couple $(u,v)$ of bounded functions such that
$u\oplus v\leq c$, the inequality
$$\int u \,d\nu_1-H(\nu_1| \mu) + \int v\,d\nu_2- H(\nu_2 | \mu)\leq 0$$
holds for all $\nu_1,\nu_2 \in \mathrm{P}(\X).$ Since
$\sup_{\nu}\{\int u\,d\nu-H(\nu| \mu)\}=\log \int e^u\,d\mu$ (this
is the convex conjugate of the result of Proposition
\ref{resL-13}), the symmetric transport-entropy inequality holds
if and only if
$$\int e^{u}\,d\mu\int e^{v}\,d\mu\leq 1,$$
for all couple $(u,v)$ of bounded functions such that $u\oplus
v\leq c$. One concludes by observing that for a given $f\in
\mathcal{B}_b(\X)$, the best function $u$ such that
$u\oplus(-f)\leq c$ is $u=Q_c f.$
\endproof
As we have seen above, the Property $(\tau)$ is always stronger
than the transport inequality $\mathcal{T}_c\leq H$. Actually,
when the cost function is of the form $c(x,y)=\theta(d(x,y))$ with
a convex $\theta$, the transport-entropy inequality and  the
Property $(\tau)$ are qualitatively equivalent as shown in the
following.
\begin{prop}
Let $\mu$ be a probability measure on $\X$ and
$\theta:[0,\infty)\to[0,\infty)$ be a convex function such that
$\theta(0)=0$. If $\mu$ verifies the transport-entropy inequality
$\mathcal{T}_c\leq H$, then the couple $(\mu, \tilde{c})$, with
$\tilde{c}(x,y)=2\theta\left(\frac{d(x,y)}{2}\right)$ verifies the
Property $(\tau).$
\end{prop}
\proof
According to the dual formulation of the transport inequality $\mathcal{T}_c\leq H$, one has
$$\int e^{Q_c f}\,d\mu\cdot e^{-\int f\,d\mu}\leq 1$$
Applying this inequality with $\pm Q_c f$ instead of $f$, one gets
\[\int e^{Q_c (Q_c f)}\,d\mu\cdot e^{-\int Q_c f\,d\mu}\leq 1\quad\text{and}\quad
\int e^{Q_c(-Q_c f)}\,d\mu \cdot e^{\int Q_c f\,d\mu}\leq 1.\]
Multiplying these two inequalities yields to
\[\int e^{Q_c (Q_c f)}\,d\mu\cdot\int e^{Q_c(-Q_c f)}\,d\mu\leq 1.\]
Now, for all $x,y\in \X,$ one has: $-f(y)+Q_c f(x)\leq \theta(d(x,y))$, and consequently, $-f\leq Q_c  (-Q_c f)$. On the other hand, the convexity of $\theta$ easily yields
\[Q_c (Q_c f)(x)\leq \inf_{y\in\X}\left\{f(y)+2\theta\left(\frac{d(x,y)}{2}\right)\right\}=Q_{\tilde{c}} f.\]
This completes the proof.
\endproof

We refer to the works \cite{S03,S07} by Samson and \cite{LW08} by
Lata\l a and Wojtaszczyk for recent advances in the study of the
Property $(\tau)$.
\\
In \cite{S07}, Samson established different variants of the
Property $(\tau)$ in order to derive sharp deviation results
\textit{\`a la } Talagrand for supremum of empirical processes.
\\
In \cite{LW08}, Lata\l a and Wojtaszczyk have considered a cost
function naturally associated to a probability $\mu$ on $\R^k$. A
symmetric probability measure $\mu$ is said to satisfy the
inequality $\mathrm{IC}(\beta)$ for some constant $\beta>0$ if it
verifies the Property $(\tau)$ with the cost function
$c(x,y)=\Lambda_\mu^*\left(\frac{x-y}{\beta}\right)$, where
$\Lambda_\mu^*$ is the Cram\'er transform of $\mu$ defined by
$$
\Lambda_\mu^*(x)=\sup_{y\in \R^k}\left\{ x\cdot y-\log \int e^{u\cdot y}\,d\mu(u)\right\},\quad x\in \R^k,
$$
For many reasons, this corresponds to an optimal choice for $c$.
They have shown that isotropic log-concave distributions on $\R$
(mean equals zero and variance equals one) satisfy the inequality
$\mathrm{IC}(48)$. They conjectured that isotropic log-concave
distributions in all dimensions verify the inequality
$\mathrm{IC}(\beta)$ with a universal constant $\beta$. This
conjecture is stronger than the Kannan-Lovasz-Simonovits
conjecture on the Poincar\'e constant of isotropic log-concave
distributions \cite{KLS95}.

\subsection{Definitions of the Poincar\'e and logarithmic Sobolev
inequalities}\label{sec-PLS} Let $\mu\in\PX$ be a given
probability measure and $(P_t)_{t\ge0}$ be the semigroup on
$L^2(\mu)$  of a $\mu$-reversible Markov process $(X_t)_{t\ge0}.$
The generator of $(P_t)_{t\ge0}$ is $\mathcal{L}$ and its domain
on $L^2(\mu)$  is $\mathbb{D}_2(\mathcal{L}).$ Define the
Dirichlet form
$$
\mathcal{E}(g,g):=\<-\mathcal{L} g, g\>_{\mu}, \quad g\in
\mathbb{D}_2(\mathcal{L}).
$$
Under the assumptions that
\begin{enumerate}
    \item[(a)] $(X_t)_{t\ge0}$ is $\mu$-reversible,
    \item[(b)] $(X_t)_{t\ge0}$ is $\mu$-ergodic,
\end{enumerate}
 $\mathcal{E}$ is closable in $L^2(\mu)$ and
its closure $(\mathcal{E}, \mathbb{D}(\mathcal{E}))$ admits the
domain $\mathbb{D}(\mathcal{E})=\mathbb{D}_2(\sqrt{-\mathcal{L}})$
in $L^2(\mu).$

\begin{rem}\label{rem-10} About these assumptions.
\begin{enumerate}
    \item[(a)]  means that the semigroup $(P_t)_{t\ge0}$ is
    $\mu$-symmetric.
    \item[(b)]  means that if $f\in
\BX$ satisfies $P_tf=f,$ $\mu$-a.e.\ for all $t\ge0$, then $f$ is
constant $\mu$-a.e.
\end{enumerate}

\end{rem}

\begin{defi}[Fisher information and Donsker-Varadhan information]\
\begin{enumerate}
    \item The \emph{Fisher information} of $f$ with respect to $\mu$ (and
the generator $\mathcal{L}$) is defined by
\begin{equation*}
    I_\mu(f)=\mathcal{E}(\sqrt{f},\sqrt{f})
\end{equation*}
for all $f\ge0$ such that $\sqrt{f}\in\mathbb{D}(\mathcal{E}).$

    \item The \emph{Donsker-Varadhan information} $\IDV(\nu|\mu)$ of the
measure $\nu$ with respect to $\mu$  is defined by
\begin{equation} \label{FDV}
\IDV(\nu|\mu)=
  \begin{cases}
    \mathcal{E}(\sqrt f, \sqrt f) & \text{if } \ \nu=f\mu\in\PX,\ \sqrt f\in \mathbb{D}(\mathcal{E}) \\
    +\infty & \text{otherwise}
  \end{cases}
 \end{equation}
\end{enumerate}
\end{defi}

\begin{expl}[Standard situation]\label{exp-typic}
 As a typical example, considering a
probability measure $\mu= e^{-V(x)}dx$ with $V$ of class
$\mathcal{C}^1$ on a complete connected Riemannian manifold $\X$,
one takes  $ (X_t)_{t\ge0}$ to be the diffusion generated by
\begin{equation*}
    \mathcal{L} =\Delta -\nabla
V\cdot \nabla
\end{equation*}
 where $\Delta,\nabla$ are respectively the Laplacian  and the gradient
on $\X.$ The Markov process $ (X_t)_{t\ge0}$ is $\mu$-reversible
and the corresponding Dirichlet form is given by
$$
\mathcal{E}(g,g)=\IX |\nabla g|^2\, d\mu, \quad  g\in
\mathbb{D}(\mathcal{E})=H^1(\X, \mu)
$$
where $H^1(\X, \mu)$ is the closure with respect to the norm
    $
\sqrt{\IX (|g|^2 +|\nabla g|^2)\, d\mu}
    $ of the space of infinitely
differentiable functions on $\X$ with  bounded derivatives of all
orders. It also matches with the space of these $g\in L^2(\X)$
such that $\nabla g\in L^2(\X\to T\X; \mu)$ in distribution.
\end{expl}

\begin{rem}
The Fisher information in this example
\begin{equation}\label{eqL-06}
     I_\mu(f)=\IX |\nabla \sqrt{f}|^2\, d\mu
\end{equation}
differs from the usual Fisher information
\begin{equation*}
    \IF(f|\mu):=\int_\Rk |\nabla \log f|^2\,fd\mu
\end{equation*}
by a multiplicative factor. Indeed, we have $ I_\mu=\IF/4. $ The
reason for preferring $I_\mu$ to $\IF$ in these notes is that
$\IDV(\cdot|\mu)$ is the large deviation rate function of the
occupation measure of the Markov process $(X_t)_{t\ge0}$ as will
be seen in Section \ref{section transport-information}.
\end{rem}

Let us introduce the $1$-homogenous extension of the relative
entropy $H$:
$$\mathrm{Ent}_\mu(f):=\int f\log f \,d\mu - \int f\,d\mu\ \log \int f\,d\mu,$$
for all nonnegative function $f$. The following relation holds:
$$\mathrm{Ent}_\mu(f)=\int f\,d\mu\ H\left(\left.\frac{f\mu}{\int f\,d\mu}\right|\mu\right).$$
As usual, $ \mathrm{Var}_{\mu}(f):=\IX f^2\,d\mu-(\IX f\,d\mu)^2.
$

\begin{defi}[General Poincar\'e and logarithmic Sobolev inequalities]\
\begin{enumerate}
    \item
A probability $\mu\in\PX$ is said to satisfy the Poincar\'e
inequality with constant $C$ if
\begin{equation*}
    \mathrm{Var}_{\mu}(f)\leq C I_\mu(f^2)
\end{equation*}
for any function $f\in \mathbb{D}(\mathcal{E}).$
    \item A probability $\mu$ on $\X$
is said to satisfy the logarithmic Sobolev inequality with a
constant $C>0$, if
$$H(f\mu|\mu)\leq C \IDV(f\mu|\mu)$$ holds for all  $f:\X\to \R^+.$
Equivalently, $\mu$ verifies this inequality if
\begin{equation*}
\mathrm{Ent}_\mu (f^2)\leq C I_\mu(f^2)
\end{equation*}
for any function $f\in \mathbb{D}(\mathcal{E}).$
\end{enumerate}
\end{defi}

In the special important case where $\mathcal{L}=\Delta-\nabla
V\cdot \nabla,$ the Fisher information is given by \eqref{eqL-06}
and we say that the corresponding Poincar\'e and logarithmic
Sobolev inequalities are \emph{usual}.

\begin{defi}[Usual Poincar\'e and logarithmic Sobolev inequalities, $\PI(C)$ and $\LSI(C)$]\label{def-PLS}\
\begin{enumerate}
    \item
A probability $\mu\in\PX$ is said to satisfy the (usual)
Poincar\'e inequality with constant $C,$ $\PI(C)$ for short, if
\begin{equation*}
    \mathrm{Var}_{\mu}(f)\leq C \int |\nabla f|^2\,d\mu
\end{equation*}
for any function smooth enough function $f.$
    \item A probability $\mu$ on $\X$
is said to satisfy the (usual) logarithmic Sobolev inequality with
a constant $C>0$, $\LSI(C)$ for short, if
\begin{equation}\label{Log Sob}
\mathrm{Ent}_\mu (f^2)\leq C\int |\nabla f|^2\,d\mu
\end{equation}
for any function smooth enough function $f.$
\end{enumerate}
\end{defi}

\begin{rem}[Spectral gap]
The Poincar\'e inequality $\PI(C)$ can be made precise by means of
the Dirichlet form $\mathcal{E}:$
\begin{equation*}
    \Var_\mu(f) \le C\,\mathcal{E}(f, f), \quad  f\in
\dd_2(\LL)
\end{equation*}
for some finite $C\ge0$. The best constant $C$ in the above
Poincar\'e inequality  is the inverse of the spectral gap of
$\LL.$
\end{rem}

\subsection{Links with Poincar\'e inequalities}
We noticed in Section 1 that  $\T_2$ is  stronger than  $\T_1$.
The subsequent proposition  enables us to make precise the gap
between these two inequalities.

\begin{prop}\label{Trans->Poinc}
Let $\mu$ be a probability measure on $\R^k$ and $d$ be the
Euclidean distance ; if $\mu$ verifies the inequality
$\mathcal{T}_{\theta(d)}\leq H$ with a function $\theta(t)\geq
t^2/C$ near $0$ with $C>0$, then $\mu$ verifies the Poincar\'e
inequality $\PI(C/2).$
\end{prop}

So in particular, $\T_2$ implies Poincar\'e inequality with the
constant $C/2,$ while $\T_1$ doesn't imply it. The result above
was established by Otto and Villani in \cite{OV00}. Below is a
proof using the Hamilton-Jacobi semigroup.

\proof According to Corollary \ref{BG2}, for all bounded
continuous function $f:\R^k\to \R$, $\int e^{Rf}\,d\mu \leq
e^{\int f\,d\mu}$, where $Rf(x)=\inf_{y\in
\R^k}\{f(y)+\theta(|x-y|_{2})\}.$ For all $t>0$, define
$R_{t}f(x)=\inf_{y\in \R^k}\{f(y)+\frac{1}{t}\theta(|y-x|_{2})\}.$
Suppose that $\theta(u)\geq u^2/C$, for all $0\leq u\leq r$, for
some $r>0$. If $M>0$, is such that $|f(x)|\leq M$ for all $x\in
\R^k$, then it is not difficult to see that the infimum in the
definition of $R_{t}f(x)$ is attained in the ball of center $x$
and radius $r$ as soon as $t\leq \theta(r)/(2M)$. So, for all
$t\leq \theta(r)/(2M)$,
$$R_{t}f(x)=\inf_{|y-x|\leq r}\{f(y)+\frac{1}{t}\theta(|y-x|_{2})\}\geq \inf_{|y-x|\leq r}\{f(y)+\frac{1}{Ct}|y-x|_{2}^2\}\geq Q_{t}f(x),$$
with $Q_{t}f(x)=\inf_{y\in \R^k}\{f(y)+\frac{1}{tC}|x-y|^2_{2}\},$
$t>0.$ Consequently, the inequality
\begin{equation}\label{Trans->Poinc1}
\int e^{tQ_{t}f}\,d\mu\leq e^{t\int f\,d\mu}
\end{equation}
holds for all $t\geq 0$ small enough.

If $f$ is smooth enough (say of class $\mathcal{C}^2$), then
defining $Q_{0}f=f$, the function $(t,x)\mapsto Q_{t}f(x)$ is
solution of the Hamilton-Jacobi partial differential equation :
\begin{equation}\label{HJ}
\left\{\begin{array}{ll}\frac{\partial u}{\partial
t}(t,x)+\frac{C}{4}|\nabla_{x} u|^2(t,x)=0, & t\geq 0, x\in \R^k
\\u(0,x)=f, & x\in \R^k\end{array}\right.
\end{equation}
 (see for example \cite[Theorem 22.46]{Vill2}).

So if $f$ is smooth enough, it is not difficult to see that
$$
    \int e^{tQ_{t}f}\,d\mu=1+t\int f\,d\mu+ \frac{t^2}{2}\left(\int f^2\,d\mu-\frac{C}{2}\int |\nabla f|^2\,d\mu\right)+o(t^2).
$$
So \eqref{Trans->Poinc1} implies that $\mathrm{Var}_{\mu}(f)\leq
\frac{C}{2}\int |\nabla f|^2\,d\mu$, which completes the proof.
\endproof

\subsection{Around Otto-Villani theorem}
We now present the famous Otto-Villani theorem and its
generalizations.

\begin{thm}[Otto-Villani]\label{OV}
Let $\mu$ be a probability measure on $\R^k$. If $\mu$ verifies
the logarithmic Sobolev inequality $\LSI(C)$ with a constant
$C>0$, then it verifies the inequality then $\T_{2}(C)$.
\end{thm}
Let us mention that Otto-Villani theorem is also true on a
Riemannian manifold. This result was conjectured by Bobkov and
G\"otze in \cite{BG99} and first proved by Otto and Villani in
\cite{OV00}. The proof by Otto and Villani was a rather
sophisticated combination of optimal transport and partial
differential equation results. It was adapted to other situations
(in particular to path spaces) by Wang in \cite{W04,W08a, W08b}.
Soon after \cite{OV00}, Bobkov, Gentil and Ledoux have proposed in
\cite{BGL01} a much more elementary proof relying on simple
computations on the Hamilton-Jacobi semigroup. This approach is at
the origin of many subsequent developments (see for instance
\cite{GM02}, \cite{GGM05}, \cite{CG06}, \cite{LV07} or
\cite{GRS09}). In \cite{G09}, Gozlan gives yet another proof which
is build on the characterization of dimension-free Gaussian
concentration exposed in Corollary \ref{T2=concentration}. It is
very robust and works as well if $\R^k$ is replaced by an (almost)
arbitrary polish space (see \cite[Theorems 4.9 and 4.10]{G09}).

\proof[First proof of Theorem \ref{OV} following \cite{BGL01}.] In
this proof we explain the Hamilton-Jacobi method of Bobkov, Gentil
and Ledoux. Consider the semigroup $Q_t$ defined for all bounded
function $f$ by $$Q_t f(x)=\inf_{y\in
\R^k}\left\{f(y)+\frac{1}{Ct}|x-y|^2\right\},\quad t>0,\qquad Q_0f=f.$$ If $f$ is
smooth enough, $(t,x)\mapsto Q_tf(x)$ solves the Hamilton-Jacobi
equation \eqref{HJ}.

According to \eqref{Log Sob}, and \eqref{HJ}
\begin{equation}\label{OV 1}
\mathrm{Ent}_\mu(e^{tQ_t f})\leq \frac{Ct^2}{4} \int
|\nabla_xQ_tf|^2(t,x)e^{tQ_tf(t,x)}\,d\mu(x)=-t^2\int
\frac{\partial Q_tf}{\partial t}(t,x)e^{tQ_tf(t,x)}\,d\mu(x).
\end{equation}
Let $Z_t=\int e^{tQ_t f(t,x)}\,d\mu(x)$, $t\geq 0$ ; then
$$Z_t'=\int Q_tf(t,x) e^{tQ_t f(t,x)}\,d\mu(x)+t\int
\frac{\partial Q_tf}{\partial t}(t,x)e^{tQ_tf(t,x)}\,d\mu(x).$$
Consequently,
\begin{align*}
\mathrm{Ent}_\mu(e^{tQ_t f})&=t\int Q_tf(t,x)e^{tQ_t f(t,x)}\,d\mu(x)-\int e^{tQ_t f(t,x)}\,d\mu(x)\log \int e^{tQ_t f(t,x)}\,d\mu(x)\\
&= tZ_t'-Z_t\log Z_t-t^2\int \frac{\partial Q_tf}{\partial
t}(t,x)e^{tQ_tf(t,x)}\,d\mu(x).
\end{align*}
This together with \eqref{OV 1}, yields $ tZ_t'-Z_t\log Z_t\leq 0$
for all $t\geq 0$. In other words, the function $t\mapsto
\frac{\log Z_t}{t}$ is decreasing on $(0,+\infty)$. As a result,
$$\log Z_1=\log \int e^{Q_1f}\,d\mu\leq \lim_{t\to0}\frac{\log Z_t}{t}=\int f\,d\mu,$$
which is Bobkov-G\"otze dual version of $\T_2(C)$ stated in
Corollary \ref{BG2}.
\endproof

\proof[Second proof of Theorem \ref{OV} following \cite{G09}] Now
let us explain how to use concentration to prove Otto-Villani
theorem. First let us recall the famous Herbst argument. Take $g$
a $1$-Lipschitz function such that $\int g\,d\mu=0$ and apply
\eqref{Log Sob} to $f=e^{tg/2}$, with $t\geq 0$; then letting
$Z_t=\int e^{tg}\,d\mu$, one gets
$$Z'_t-Z_t\log Z_t\leq \frac{Ct^2}{4}\int |\nabla g|^2e^{tg}\,d\mu\leq \frac{C}{4}t^2Z_t,\quad t>0$$
where the inequality follows from the fact that $g$ is
$1$-Lipschitz. In other word,
$$\frac{d}{dt}\left(\frac{\log Z_t}{t}\right)\leq \frac{C}{4},\quad t>0.$$
Since $\frac{\log Z_t}{t}\to0$ when $t\to0$, integrating the
inequality above yields
$$\int e^{tg}\,d\mu\leq e^{\frac{C}{4}t^2},\quad t>0.$$
Since this holds for all centered and $1$-Lipschitz function $g$,
one concludes from Corollary \ref{BG3} that $\mu$ verifies
the inequality $\T_1(C)$ on $(\R^k, |\cdot|_2).$

The next step is a tensorization argument. Let us recall that the
logarithmic Sobolev inequality enjoys the following well known
tensorization property : if $\mu$ verifies $\LSI(C)$ on $\R^k$,
then for all positive integer $n$, the product probability measure
$\mu^n$ satisfies $\LSI(C)$ on $\left(\R^k\right)^n.$ As a
consequence, the argument above shows that for all positive
integer $n$, $\mu^n$ verifies the inequality $\T_1(C)$ on
$\left(\left(\R^k\right)^n,|\cdot|_2\right).$ According to
Marton's argument (Theorem \ref{Marton}), there is some constant
$K>0$ such that for all positive integer $n$ and all $A\subset
\left(\R^k\right)^n$, with $\mu^n(A)\geq 1/2$, it holds
$$\mu^n(A^r)\geq 1-Ke^{-Cr^2},\quad r\geq 0.$$
The final step is given by Corollary \ref{T2=concentration} : this
dimension-free Gaussian concentration inequality implies $\T_2(C)$
and this completes the proof.
\endproof

Otto-Villani theorem admits the following natural extension
which appears in \cite{BGL01} and \cite[Theorem 2.10]{GGM05}.

For all $p\in [1,2]$, define $\theta_{p}(x)=x^2$ if $|x|\leq 1$
and $\frac{2}{p}x^p+1-\frac{2}{p}$ if $|x|\geq 1$.

\begin{thm}
Suppose that a probability $\mu$ on $\R^k$ verifies the following
modified logarithmic Sobolev inequality
$$\mathrm{Ent}_\mu(f^2)\leq C_{1} \int \sum_{i=1}^k\theta_{p}^*\left(\frac{\partial f}{\partial x_i}\frac{1}{f}\right)f^2\,d\mu,$$
for all $f:\R^k\to\R$ smooth enough, where
$\theta_{p}^*$ is the convex conjugate of $\theta_{p}$. Then there
is a constant $C_{2}$ such that $\mu$ verifies the transport-entropy inequality $\mathcal{T}_{\theta_{p}(|\,\cdot\,|_2)}\leq
C_{2}H$.
\end{thm}
The theorem above is stated in a very lazy way ; the relation
between $C_{1}$ and $C_{2}$ is made clear in \cite[Theorem
2.10]{GGM05}.

\proof[Sketch of proof.] We shall only indicate that
two proofs can be made. The first one uses the following Hamilton
Jacobi semigroup
$$Q_t f(x)=\inf_{y\in \R^k}\left\{f(y)+t\sum_{i=1}^k\theta_p\left(\frac{y_i-x_i}{t}\right)\right\},$$
which solves the following Hamilton-Jacobi equation
$$\left\{\begin{array}{ll}\frac{\partial u}{\partial t}(t,x)+\sum_{i=1}^k\theta_p^*\left(\frac{\partial u}{\partial x_i}\right)(t,x)=0, & t\geq 0, x\in \R^k \\u(0,x)=f, & x\in \R^k\end{array}\right.$$
The second proof uses concentration : according to a result by
Barthe and Roberto \cite[Theorem 27]{BR08}, the modified
logarithmic Sobolev inequality implies dimension-free
concentration for the enlargement
$\mathrm{enl}_{\theta_p}(A,r)=\{x\in \left(\R^k\right)^n;
\inf_{y\in A} \sum_{i=1}^n \theta_p(|x-y|_2)\leq r\}$ and
according to Theorem \ref{Concentration->Transport}, this
concentration property implies the transport-entropy inequality
$\mathcal{T}_{\theta_{p}(|\cdot|_2)}\leq CH$, for some $C>0.$
\endproof

The case $p=1$ is particularly interesting ; namely according to a
result by Bobkov and Ledoux, the modified logarithmic Sobolev
inequality with the function $\theta_{1}$ is equivalent to
Poincar\'e inequality (see \cite[Theorem 3.1]{BL97} for a precise
statement). Consequently, the following results holds
\begin{cor}\label{Poinc=Trans}
Let $\mu$ be a probability measure on $\R^k$; the following
propositions are equivalent:
\begin{enumerate}
\item The probability $\mu$ verifies Poincar\'e inequality for
some constant $C>0$; \item The probability $\mu$ verifies the
transport-entropy inequality $\mathcal{T}_c\leq H$, with a cost
function of the form $c(x,y)= \theta_1(a|x-y|_2)$, for some $a>0$.
\end{enumerate}
Moreover the constants are related as follows: (1) implies (2)
with $a=\frac{1}{\tau\sqrt{C}}$, where $\tau$ is some
universal constant and (2) implies (1) with $C=\frac{1}{2a^2}.$
\end{cor}
Again a precise result can be found in \cite[Corollary
5.1]{BGL01}.

Let us recall that $\mu$ is said to satisfy a super-Poincar\'e
inequality if there is a decreasing function
$\beta:[1,+\infty)\to[0,\infty)$ such that
$$ \int f^2\,d\mu \leq \beta(s)\int |\nabla f|^2\,d\mu + s\left(\int |f|\,d\mu\right)^2,\quad s\ge 1$$
holds true for all sufficiently smooth $f$. This class of
functional inequalities was introduced by Wang in \cite{W00} with
applications in spectral theory. Many functional inequalities
(Beckner-Lata\l a-Oleszkiewicz inequalities for instance
\cite{LO00, W05}) can be represented as a super-Poincar\'e
inequality for a specific choice of the function $\beta$. Recently
efforts have been made to see which transport-entropy inequalities
can be derived from super-Poincar\'e inequalities. We refer to
Wang \cite[Theorem 1.1]{W08b} (in a Riemannian setting) and Gozlan
\cite[Theorem 5.4]{G08} for these very general extensions of
Otto-Villani theorem.

\subsection{$\T_2$ and $\LSI$ under curvature assumptions}
In \cite{OV00}, Otto and Villani proved that the logarithmic
Sobolev inequality was sometimes implied by the inequality
$\T_{2}$. The key argument for this converse is the so called HWI
inequality (see \cite[Theorem 3]{OV00} or Corollary \ref{HWI} of the present paper) which is recalled below. If $\mu$ is an absolutely continuous probability measure with a density of the form $d\mu(x)=e^{-V(x)}\,dx$, with $V$ of class $\mathcal{C}^2$ on
$\R^k$ and such that $\mathrm{Hess}\, V\geq \kappa \mathrm{Id}$, with $\kappa \in \R$,
then for all probability measure $\nu$ on $\R^k$ having a smooth
density with respect to $\mu$,
\begin{equation}\label{eqHWI}
H(\nu|\mu)\leq 2W_{2}(\nu,\mu)\sqrt{I(\nu|\mu)}-\frac{\kappa}{2}W_{2}^2(\nu,\mu),
\end{equation}
where $\IDV(\cdot|\mu)$ is the Donsker-Varadhan information.

\begin{prop}\label{prop HWI}Let $d\mu(x)=e^{-V(x)}\,dx$, with $V$ of class $\mathcal{C}^2$ on $\R^k$ and such that $\mathrm{Hess}\, V\geq \kappa \mathrm{Id}$, with $\kappa \leq 0$; if $\mu$ verifies the inequality $\T_{2}(C)$, with $C< -2/\kappa$, then it verifies the inequality $ \LSI \left(\frac{4C}{\left(1+\kappa C/2\right)^2}\right)$.
In particular, when $V$ is convex then $\mu$ verifies $\LSI (4C)$.
\end{prop}
\proof
Applying \eqref{eqHWI} together with the assumed $\T_{2}(C)$
inequality, yields
$$H(\nu|\mu)\leq 2\sqrt{CH(\nu|\mu)}\sqrt{I(\nu|\mu)}-\frac{\kappa C}{2}H(\nu|\mu),$$ for all $\nu.$ Thus, if $1+\frac{\kappa C}{2}> 0$, one has
$H(\nu|\mu)\leq
\frac{4C}{\left(1+\frac{\kappa C}{2}\right)^2}I(\nu|\mu).$ Taking
$d\nu(x)=f^2(x)\,dx$ with a smooth $f$ yields
$$\mathrm{Ent}_{{\mu}}(f^2)\leq \frac{4C}{\left(1+\frac{\kappa C}{2}\right)^2}\int |\nabla f|^2\,d\mu,$$
which completes the proof.
\endproof

So in the range $C+2/\kappa < 0$, $\T_{2}$ and $\LSI$ are
equivalent. In fact under the condition $\mathrm{Hess}\, V\geq
\kappa$, a strong enough Gaussian concentration property implies
the logarithmic Sobolev inequality, as shown in the following
theorem by Wang \cite{W97}.
\begin{thm}\label{Wang}
Let $d\mu(x)=e^{-V(x)}\,dx$, with $V$ of class $\mathcal{C}^2$ on $\R^k$ and such that $\mathrm{Hess}\, V\geq \kappa \mathrm{Id}$, with $\kappa \leq 0$; if there is some $C< -2/\kappa$, such that $\int e^{\frac{1}{C}d^2(x_o,x)}\,d\mu(x)$ is finite, for some (and thus all) point $x_o$, then $\mu$ verifies the logarithmic Sobolev inequality for some constant $\tilde{C}$.
\end{thm}
Recently Barthe and Kolesnikov have generalized Wang's theorem to
different functional inequalities and other convexity defects
\cite{BK08}. Their proofs rely on Theorem \ref{res-07}. A drawback
of Theorem \ref{Wang} is that the constant $\tilde{C}$ depends too
heavily on the dimension $k$. In a series of papers \cite{Mil09a,
Mil09b,Mil09d}, E. Milman has shown that under curvature
conditions concentration inequalities and isoperimetric
inequalities are in fact equivalent with a \emph{dimension-free}
control of constants. Let us state a simple corollary of Milman's
results.
\begin{cor}
Let $d\mu(x)=e^{-V(x)}\,dx$, with $V$ of class $\mathcal{C}^2$ on $\R^k$ and such that $\mathrm{Hess}\, V\geq \kappa \mathrm{Id}$, with $\kappa \leq 0$; if there is some $C< -2/\kappa$ and $M>1$, such that $\mu$ verifies the following Gaussian concentration inequality
\begin{equation}\label{Wang Milman}
\mu(A^r)\geq 1-Me^{-\frac{1}{C}r^2},\quad r\geq 0
\end{equation}
for all $A$ such that $\mu(A)\geq 1/2$ and with $A^r=\{x \in \R^k; \exists y\in A, |x-y|_{2}\leq r\}$, then $\mu$ verifies the logarithmic Sobolev inequality with a constant $\tilde{C}$ depending only on $C$, $\kappa$ and $M$. In particular, the constant $\tilde{C}$ is independent on the dimension $k$ of the space.
\end{cor}
The conclusion of the preceding results is that when $C+2/\kappa < 0$, it holds
$$\LSI (C)\Rightarrow \T_{2}(C) \Rightarrow \text{ Gaussian concentration }\eqref{Wang Milman}\text{ with constant }1/C\Rightarrow \LSI(\tilde{C}),$$
and so these three inequalities are qualitatively equivalent in this range of parameters.
Nevertheless, the equivalence between $\LSI$ and $\T_2$ is no longer true when $\mathrm{Hess}\,V$ is unbounded from below. In \cite{CG06}, Cattiaux and Guillin were able to give
an example of a probability $\mu$ on $\R$ verifying $\T_{2}$, but
not $\LSI$. Cattiaux and Guillin's counterexample is discussed in
Theorem \ref{Cattiaux Guillin example} below.

\subsection{A refined version of Otto-Villani theorem}
We close this section with a recent result by Gozlan, Roberto and
Samson \cite{GRS09} which completes the picture showing that
$\T_{2}$ (and in fact many other transport-entropy inequalities)
is equivalent to a logarithmic Sobolev inequality restricted to a
subclass of functions.

Let us say that a function $f:\R^k\to\R$ is  $\lambda$-semiconvex, $\lambda\geq 0$, if the function $x\mapsto
f(x)+\frac{\lambda}{2}|x|^2$ is convex. If $f$ is $\mathcal{C}^2$ this is
equivalent to the condition $\mathrm{Hess}\, f(x)\geq -\lambda \mathrm{Id}.$
Moreover, if $f$ is $\lambda$-semiconvex, it is almost everywhere
differentiable, and for all $x$ where $\nabla f(x)$ is well
defined, one has
$$f(y)\geq f(x)+\nabla f(x)\cdot (y-x) -\frac{\lambda}{2}|y-x|^2.$$
for all $y\in \R^k.$

\begin{thm}\label{GRS}
Let $\mu$ be a \prob on $\R^k.$ The following propositions are
equivalent:
\begin{enumerate}
\item There exists $C_1>0$ such that $\mu$ verify the inequality
$\T_2(C_1)$. \item There exists $C_2>0$ such that for all $0\leq
\lambda<\frac{2}{C_2}$ and all $\lambda$-semiconvex
$f:\R^k\to\R$,
$$\mathrm{Ent}_{\mu}(e^f)\leq \frac{C_2}{\left(1-\frac{\lambda C_2}{2}\right)^2} \int |\nabla f|^2 e^f\, d\mu.$$
\end{enumerate}
The constants $C_1$ and $C_2$ are related in the following way:
$$(1)\Rightarrow (2) \text{ with } C_2=C_1.$$
$$(2)\Rightarrow (1) \text{ with } C_1=8C_2.$$
\end{thm}
More general results can be found in \cite[Theorem 1.8]{GRS09}.
Let us emphasize the main difference between this theorem and
Proposition \ref{prop HWI} : in the result above the curvature
assumption is made on the functions $f$ and not on the potential
$V$. A nice corollary of Theorem \ref{GRS}, is the following
perturbation result:
\begin{thm}\label{bounded perturbation}
Let $\mu$ be a probability measure on $\R^k$ and consider
$d\tilde{\mu}(x)=e^{\varphi(x)}\,dx$, where $\varphi:\R^k\to\R$ is
bounded. If $\mu$ verifies $\T_2(C)$, then $\tilde{\mu}$ verifies
$\T_2(8e^{\mathrm{Osc}(\varphi)}C)$, where
$\mathrm{Osc}(\varphi)=\sup \varphi - \inf \varphi.$
\end{thm}
Many functional inequalities of Sobolev type enjoy the same
bounded perturbation property (without the factor $8$). For the
Poincar\'e inequality or the logarithmic Sobolev inequality, the
proof (due to Holley and Stroock) is almost straightforward (see
e.g \cite[Theorems 3.4.1 and 3.4.3]{Log-Sob}). For
transport-entropy inequalities, the question of the perturbation
was raised in \cite{OV00} and remained open for a long time. The
proof of Theorem \ref{bounded perturbation} relies on the
representation of $\T_2$ as a restricted logarithmic Sobolev
inequality provided by Theorem \ref{GRS}. Contrary to Sobolev type
inequalities, no direct proof of Theorem \ref{bounded
perturbation} is known.

\section{Workable sufficient conditions for transport-entropy inequalities}\label{section sufficient conditions}
In this section, we review some of the known sufficient conditions
on $V:\R^k\to\R$ under which $d\mu=e^{-V}\,dx$ verifies  a
transport-entropy inequality of the form
$\mathcal{T}_{\theta(d)}\leq H$. Unlike Section \ref{section
uniformly convex}, the potential $V$ is not supposed to be
(uniformly) convex.
\subsection{Cattiaux and Guillin's restricted logarithmic Sobolev method}
Let $\mu$ be a probability measure on $\R^k$ such that $\int
e^{\varepsilon |x|^2}\,d\mu(x)<+\infty$, for some $\varepsilon>0$.
Following Cattiaux and Guillin in \cite{CG06}, let us say that
$\mu$ verifies the restricted logarithmic Sobolev inequality
$\rLSI(C,\eta)$ if
\begin{equation*}
   \mathrm{Ent}_\mu(f^2)\leq C\int |\nabla f|^2\,d\mu,
\end{equation*}
for all smooth $f:\R^k\to \R$ such that
$$f^2(x)\leq \left(\int f^2\,d\mu\right)e^{\eta |x_{o}-x|^2+\int |x_{o}-y|^2\,d\mu(y)},\quad x\in \R^k.$$
Using Bobkov-Gentil-Ledoux proof of Otto-Villani theorem,
Cattiaux and Guillin obtained the following result (see
\cite[Theorem 1.17]{CG06}).
\begin{thm}
Let $\mu$ be a probability measure on $\R^k$ such that $\int
e^{\varepsilon |x|^2}\,d\mu(x)<+\infty$, for some $\varepsilon>0$.
If the restricted logarithmic Sobolev inequality $\rLSI(C,\eta)$
holds for some $\eta<\varepsilon/2$, then $\mu$ verifies the
inequality $\T_{2}(\tilde{C})$, for some $\tilde{C}>0$.
\end{thm}
The interest of this theorem is that the restricted logarithmic
Sobolev inequality above is strictly weaker than the usual one.
Moreover, workable sufficient conditions for the $\rLSI$ can be
given. Let us start with the case of the real axis.
\begin{thm}
Let $d\mu(x)=e^{-V(x)}\,dx$ be a probability measure on $\R$ with
$\int e^{\varepsilon|x|^2}\,d\mu(x)<+\infty$ for some
$\varepsilon>0$. If $\mu$ is such that
$$A^+=\sup_{x\geq 0} \int_{x}^{+\infty}t^2e^{-V(t)}\,dt \int_{0}^x e^{V(t)}\,dt\quad\text{and}\quad A^-=\sup_{x\leq 0} \int_{-\infty}^xt^2e^{-V(t)}\,dt \int_{x}^0 e^{V(t)}\,dt$$ are finite then $\mu$ verifies $\rLSI(C,\eta)$, for some $C,\eta>0$ and so it verifies also $\T_{2}(\tilde{C})$ for some $\tilde{C}>0$.
\end{thm}
The finiteness of $A^+$ and $A^-$ can be determined using the
following proposition (see \cite[Proposition 5.5]{CG06}).
\begin{prop}\label{CG}
Suppose that $d\mu(x)=e^{-V(x)}\,dx$ be a probability measure on
$\R$ with $V$ of class $\mathcal{C}^2$ such that $\frac{V''}{(V')^2}(x)\to
0$ when $x\to \infty$. If $V$ verifies
\begin{equation}\label{Cattiaux-Guillin}
\limsup_{x\to\pm\infty} \left|\frac{x}{V'(x)}\right|<+\infty,
\end{equation}
then $A^+$ and $A^-$ are finite (and there is $\varepsilon>0$ such
that $\int e^{\varepsilon|x|^2}\,d\mu(x)<+\infty$).
\end{prop}
The condition $\frac{V''}{(V')^2}(x)\to 0$ when $x\to \infty$ is
not very restrictive and appears very often in results of this
type (see \cite[Corollary 6.4.2 and Theorem 6.4.3]{Log-Sob} for
instance).

Now let us recall the following result by Bobkov and G\"otze (see
\cite[Theorem 5.3]{BG99} and \cite[Theorems 6.3.4 and
6.4.3]{Log-Sob}) dealing this time with the logarithmic Sobolev
inequality.
\begin{thm}\label{BG Log Sob}
Let $d\mu(x)=e^{-V(x)}\,dx$ be a probability measure on $\R$, and
$m$ a median of $\mu$. If $V$ is such that
\begin{gather*}
D^-=\sup_{x<m} \int_{-\infty}^x e^{-V(t)}\,dt \log\left(\frac{1}{\int_{-\infty}^x e^{-V(t)}\,dt}\right) \int_{x}^me^{V(t)}\,dt\\
D^+=\sup_{x>m} \int^{+\infty}_{x} e^{-V(t)}\,dt \log\left(\frac{1}{\int^{+\infty}_x e^{-V(t)}\,dt}\right) \int_{m}^xe^{V(t)}\,dt\\
\end{gather*}
are finite, then $\mu$ verifies the logarithmic Sobolev
inequality, and the optimal constant $C_{opt}$ verifies
$$\tau_{1} \max(D^-,D^+)\leq C_{opt}\leq \tau_{2} \max(D^-,D^+),$$
where $\tau_{1}$ and $\tau_{2}$ are known universal constants.

Moreover if $V$ is of class $\mathcal{C}^2$ and verifies
$\lim_{x\to \infty} \frac{V''}{(V')^2}(x)=0$, then $D^-$ and $D^+$
are finite if and only if $V$ verifies the following conditions:
$$\liminf_{x\to \infty} |V'(x)|>0 \qquad \text{and}\qquad \limsup_{x\to \infty}\frac{V(x)}{(V')^2(x)}<+\infty$$
\end{thm}

\begin{thm}[Cattiaux-Guillin's counterexample]\label{Cattiaux Guillin example}
The probability measure $d\mu(x)=\frac{1}{Z}e^{-V(x)}\,dx$ defined
on $\R$ with $V(x)=|x|^3+3x^2\sin^2x+|x|^\beta$, with $Z$ a
normalizing constant and $2<\beta<5/2$ satisfies the inequality
$\T_{2}$ but not the logarithmic Sobolev inequality.
\end{thm}
\proof For all $x>0$,
\begin{align*}
V'(x)&=3x^2+6x\sin^2 x+6x^2\cos x \sin x + \beta x^{\beta-1}\\
&=3x^2(1+\cos 2x)+6x\sin^2 x+\beta x^{\beta-1}.
\end{align*}
and
\begin{align*}
V''(x)&=6x^2\cos 2x + 6x(1+2\sin 2x) + 6\sin^2 x +\beta(\beta-1)x^{\beta-2}.\\
\end{align*}
First, observe that $V'(x)>0$ for all $x>0$ and $V'(x)\to \infty$
when $x\to +\infty$. Moreover, for $x$ large enough,
$\left|\frac{V''}{V'^2}(x)\right|\leq D \frac{x^2}{x^{2\beta-2}}$,
and $0\leq \frac{x}{V'}\leq D\frac{x}{x^{\beta-1}}$, for some
numerical constant $D>0$. Since, $\beta>2$, it follows that
$\frac{V''}{V'^2}(x)\to0,$ and $\frac{x}{V'(x)}\to 0$ when
$x\to+\infty.$ Consequently, it follows from Proposition \ref{CG},
that $\mu$ verifies $\T_2(C)$, for some $C>0.$ On the other hand,
consider the sequence $x_k=\frac{\pi}{4}+k\pi$, then
$V'^2(x_k)=(6x_k+\beta x_k^{\beta-1})^2\sim \beta^2 (\pi
k)^{2\beta-2} $ and $V(x_k)\sim (k\pi)^3$. So
$\frac{V(x_k)}{V'^2(x_k)}\sim \beta^2 (\pi k)^{5-2\beta}$, and
since $\beta<5/2$, one concludes that $\limsup_{x\to +
\infty}\frac{V}{V'^2}(x)=+\infty$. According to Theorem \ref{BG
Log Sob}, it follows that $\mu$ does not verify the logarithmic-
Sobolev inequality.
\endproof

Recently, Cattiaux, Guillin and Wu have obtained in \cite{CGW08}
different sufficient conditions for the restricted logarithmic
Sobolev inequality $\rLSI$ in dimension $k\geq 1.$
\begin{thm}\label{CGW}
Let $\mu$ be a probabilikty measure on $\R^k$ with a density of
the form $d\mu(x)=e^{-V(x)}\,dx$, with $V:\R^k\to\R$ of class
$\mathcal{C}^2$. If one of the following conditions
$$\exists a<1, R,c>0,\text{ such that } \forall |x|>R,\qquad (1-a)|\nabla V(x)|^2-\Delta V(x) \geq c|x|^2$$
or
$$\exists R,c>0, \text{such that } \forall |x|>R,\qquad x\cdot \nabla V(x)\geq c|x|^2$$
is satisfied, then $\rLSI$ holds.
\end{thm}
We refer to \cite[Corollary 2.1]{CGW08} for a proof relying on the
so called Lyapunov functions method.

\subsection{Contraction methods}
In \cite{G07}, Gozlan recovered Cattiaux and Guillin's sufficient
condition \eqref{Cattiaux-Guillin} for $\T_2$ and extended it to
other transport-entropy inequalities on the real axis. The proof
relies on a simple contraction argument, we shall now explain it
in a general setting.
\subsubsection*{Contraction of transport-entropy inequalities}
In the sequel, $\mathcal{X}$ and $\mathcal{Y}$ will be polish
spaces. If $\mu$ is a probability measure on $\mathcal{X}$ and
$T:\mathcal{X}\to\Y$ is a measurable map, the image of $\mu$ under
$T$ will be denoted by $T_\# \mu$; by definition, it is the
probability measure on $\Y$ defined by
$$T_\#\mu(A)=\mu\left(T^{-1}(A)\right),$$
for all measurable subset $A$ of $Y$.

The result below shows that if $\mu$ verifies a transport-entropy inequality on $\X$ then the image $T_\# \mu$ verifies
verifies a transport-entropy inequality on $\Y$ with a new cost
function  expressed in terms on $T$.

\begin{thm}\label{contraction}
Let $\mu_o$ be a probability measure on $\X$ and $T:\X\to\Y$ be a
measurable bijection. If $\mu_o$ satisfies the transport-entropy inequality $\alpha(\mathcal{T}_c)\leq H$ with a cost function $c$
on $\X$, then $T_\#\mu_o$ satisfies the transport-entropy
inequality $\alpha(\mathcal{T}_{c^T})\leq H$ with the cost
function $c^T$ defined on $\Y$ by
\[c^T(y_1,y_2)=c(T^{-1}y_1,T^{-1}y_2),\quad y_1,y_2\in \Y.\]
\end{thm}

\proof Let us define $Q(y_1,y_2)=(T^{-1}y_1,T^{-1}y_2), \
y_1,y_2\in \Y$, and $\mu_1=T_\#\mu_o$. Let $\nu\in \mathrm{P}(\Y)$
and take $\pi\in \Pi(\nu,\mu_1),$ the subset of
$\mathrm{P}(\mathcal{Y}^2)$ consisting of the probability $\pi$
with their marginal measures $\pi_0=\nu$ and $\pi_1=\mu_1.$ Then
$\int c^T(y_1,y_2)\,d\pi=\int c(x,y)\,dQ_\# \pi,$ so
$$
    \mathcal{T}_{c^T}(\nu,\mu_1)=\inf_{\pi\in Q_\# \Pi(\nu,\mu_1)}\int c(x,y)d\pi.
$$
But it is easily seen that $Q_\# \Pi(\nu,\mu_1)=\Pi({T^{-1}}_\#
\nu,\mu_o)$. Consequently
\[\mathcal{T}_{ c^T}(\nu,\mu_1)=\mathcal{T}_{c}({T^{-1}}_\#
\nu,\mu_o).\] Since $\mu_o$ satisfies the transport-entropy
inequality $\alpha(\mathcal{T}_c)\leq H$, it holds
\[\alpha\left(\mathcal{T}_{c}({T^{-1}}_\# \nu,\mu_o)\right)\leq H({T^{-1}}_\# \nu|\mu_o).\]
But it is easy to check , with Proposition \ref{resL-13} and the
fact that $T$ is one-one, that
\[ H({T^{-1}}_\# \nu|\mu_o)=H(\nu| T_\#\mu_o).\] Hence
\[\alpha\left(\mathcal{T}_{c^T}(\nu,\mu_1)\right)\leq H(\nu | \mu_1),\]
for all $\nu\in \mathrm{P}(Y).$
\endproof
\begin{rem}
This contraction property was first observed by Maurey (see
\cite[Lemma 2]{Mau91}) in the context of inf-convolution
inequalities. Theorem \ref{contraction} is a simple but powerful
tool to derive new transport-entropy inequalities from already
known ones.
\end{rem}
\subsubsection*{Sufficient conditions on $\R$}
Let us recall that a probability measure $\mu$ on $\R$ is said to
satisfy Cheeger inequality with the constant $\lambda>0$ if
\begin{equation}\label{Cheeger}
\int \left|f(x)-m(f)\right|\,d\mu(x)\leq \lambda \int
|f'(x)|\,d\mu(x),
\end{equation}
for all $f:\R\to \R$ sufficiently smooth, where $m(f)$ denotes a median of $f$ under $\mu$.\\

Using the contraction argument presented above, Gozlan
obtained the following theorem (\cite[Theorem 2]{G07}).
\begin{thm}\label{unidim}
Let $\theta:[0,\infty)\to [0,\infty)$ be such that $\theta(t)=t^2$
for all $t\in [0,1]$, $t\mapsto \frac{\theta(t)}{t}$ is increasing
and $\sup_{t>0} \frac{\theta(2t)}{\theta(t)}<+\infty$ and let
$\mu$ be a probability measure on $\R$ which verifies Cheeger
inequality for some $\lambda_o>0$. The following propositions are
equivalent :
\begin{enumerate}
\item The probability measure $\mu$ verifies the transport cost
inequality $\mathcal{T}_{\theta}\leq CH$, for some $C>0$. \item
The constants $K^+(\varepsilon)$ and $K^-(\varepsilon)$ defined by
$$K^+(\varepsilon)=\sup_{x\geq m} \frac{\int_x^{+\infty} e^{\varepsilon\theta(u-x)}\,d\mu(u)}{\mu[x,+\infty)}\quad\text{and}\quad K^-(\varepsilon)=\sup_{x\leq m} \frac{\int_{-\infty}^x e^{\varepsilon\theta(x-u)}\,d\mu(u)}{\mu(-\infty,x]}$$
are finite for some $\varepsilon>0$, where $m$ denotes the median
of $\mu.$
\end{enumerate}
\end{thm}

The condition $K^+$ and $K^-$ finite is always necessary to have
the transport-entropy inequality (see \cite[Corollary 15]{G07}).
This condition is sufficient if Cheeger inequality holds. Cheeger
inequality is slightly stronger than Poincar\'e inequality. On the
other hand transport-entropy inequalities of the form
$\mathcal{T}_\theta\leq H$, with a function $\theta$ as above,
imply Poincar\'e inequality (Theorem \ref{Trans->Poinc}). So
Theorem \ref{unidim} offers a characterization of
transport-entropy inequalities except perhaps on the ``small" set
of probability measures verifying Poincar\'e but not Cheeger
inequality.

\proof[Sketch of proof.] We will only prove the sufficiency of the
condition $K^+$ and $K^-$ finite. Moreover, to avoid technical
difficulties, we shall only consider the case $\theta(t)=t^2$. Let
$d\mu_o(x)=\frac{1}{2}e^{-|x|}\,dx$ be the two-sided exponential
measure on $\R$. According to a result by Talagrand \cite{T96} the
probability  $\mu_o$ verifies the transport-entropy inequality
$\mathcal{T}_{c_o}\leq C_oH$, with the cost function
$c_o(x,y)=\min(|x-y|^2,|x-y|),$ for some $C_o>0.$

Consider the cumulative distribution functions of $\mu$ and
$\mu_o$ defined by $F(x)=\mu(-\infty,x]$ and
$F_o(x)=\mu_o(-\infty,x]$, $x\in \R.$ The monotone rearrangement
map $T:\R\to\R$ defined by $T(x)=F^{-1}\circ F_o$, see
\eqref{eq-08}, transports the probability $\mu_o$ onto the
probability $\mu$ : $T_\#\mu_o=\mu.$ Consequently,  by application
of the contraction Theorem \ref{contraction}, the probability
$\mu$ verifies the transport-cost inequality
$\mathcal{T}_{c_o^T}\leq C_oH,$ with
$c_o^T(x-y)=c_o(T^{-1}(x)-T^{-1}(y))$. So, all we have to do is to
show that there is some constant $a>0$ such that
$c_o(T^{-1}(x)-T^{-1}(y))\geq \frac{1}{a^2}|x-y|^2$, for all
$x,y\in\R.$ This condition is equivalent to the following
$$|T(x)-T(y)|\leq a \min (|x-y|, |x-y|^{1/2}),\quad  x,y\in \R.$$
In other words, we have to show that $T$ is $a$-Lipschitz and
$a$-H\"older of order $1/2$.

According to a result by Bobkov and Houdr\'e, $\mu$ verifies
Cheeger inequality \eqref{Cheeger} with the constant $\lambda_o$
if and only if $T$ is $\lambda_o$-Lipschitz (see \cite[Theorem
1.3]{BH97}).

To deal with the H\"older condition, observe that if $T$ is
$a$-H\"older on $[0,\infty)$ and on $\R^-$, then it is
$\sqrt{2}a$-H\"older on $\R$. Let us treat the case of
$[0,\infty)$, the other case being similar. The condition $T$ is
$a$-H\"older on $[0,\infty)$ is equivalent to
$$T^{-1}(x+u)-T^{-1}(x)\geq \frac{u^2}{a^2},\quad x\geq m,u\geq 0.$$
But a simple computation gives : $T^{-1}(x)=-\log(2(1-F(x))),$ for
all $x\geq m.$ So the condition above reads
\begin{equation}\label{unidim1}
\frac{1-F(x+u)}{1-F(x)}\leq
e^{-\frac{u^2}{a^2}},\quad  x\geq m,u\geq 0.
\end{equation}
Since, $K^+(\varepsilon)=\sup_{x\geq m} \frac{\int_x^{+\infty}
e^{\varepsilon(u-x)^2}\,d\mu(u)}{\mu[x,+\infty)}$ is finite, an
application of Markov inequality yields
$$\frac{1-F(x+u)}{1-F(x)}\leq K^+(\varepsilon)e^{-\varepsilon u^2},\quad x\geq m,u\geq 0.$$
On the other hand the Lipschitz continuity of $T$ can be written
$$\frac{1-F(x+u)}{1-F(x)}\leq e^{-\frac{u}{\lambda_o}},\quad x\geq m,u\geq 0.$$
So, if $a>0$ is chosen so that $\frac{u^2}{a^2}\leq
\max\left(\frac{u}{\lambda_o}, \varepsilon u^2-\log
K^+(\varepsilon)\right),$ then \eqref{unidim1} holds and this
completes the proof.
\endproof

The following corollary gives a concrete criterion to decide
whether a probability measure on $\R$ verifies a given
transport-entropy inequality. It can be deduced from Theorem
\ref{unidim} thanks to an estimation of the integrals defining
$K^+$ and $K^-$. We refer to \cite{G07} for this technical proof.
\begin{cor}\label{unidim-cond}
Let $\theta:[0,\infty)\to[0,\infty)$ of class $\mathcal{C}^2$ be
as in Theorem \ref{unidim} and let $\mu$ be a probability measure
on $\R$ with a density of the form $d\mu(x)=e^{-V(x)}\,dx$, with
$V$ of class $\mathcal{C}^2.$ Suppose that
$\frac{\theta''}{\theta'^2}(x)\to0$ and $\frac{V''}{V'^2}(x)\to 0$
when $x\to\infty.$ If there is some $a>0$ such that
$$\limsup_{x\to\pm\infty}\frac{\theta'(a|x|)}{|V'(m+x)|}<+\infty,$$
with $m$ the median of $\mu$, then $\mu$ verifies the
transport-entropy inequality $\mathcal{T}_\theta \leq CH$, for
some $C>0.$
\end{cor}
Note that this corollary generalizes Cattiaux and Guillin's
condition \eqref{Cattiaux-Guillin}.

\subsubsection*{Poincar\'e inequalities for non-Euclidean metrics}

Our aim is now to partially generalize to the multidimensional
case the approach explained in the preceding section. The two main
ingredients of the proof of Theorem \ref{unidim} were the
following :
\begin{itemize}
\item The fact that $d\mu_o(x)=\frac{1}{2}e^{-|x|}\,dx$ verifies
the transport-entropy inequality $\mathcal{T}_c\leq CH$ with the
cost function $c(x,y)=\min(|x-y|^2,|x-y|)$. Let us define the cost
function $c_1(x,y)=\min(|x-y|_2,|x-y|_2^2)$ on $\R^k$ equipped
with its Euclidean distance. We have seen in Corollary
\ref{Poinc=Trans} that a probability measure on $\R^k$ verifies
the transport-entropy inequality $\mathcal{T}_{c_1}\leq C_1 H$
for some $C_1>0$ if and only if it verifies Poincar\'e inequality
with a constant $C_2>0$ related to $C_1$.

\item The fact that the application $T$ sending $\mu_o$ on $\mu$
was both Lipschitz and 1/2-H\"older. Consequently, the application
$\omega=T^{-1}$ which maps $\mu$ on $\mu_o$, behaves like $x$ for
small values of $x$ and like $x^2$ for large values of $x$.
\end{itemize}
So we can combine the two ingredients above by saying that``the
image of $\mu$ by an application $\omega$ which resembles
$\pm\max(|x|,|x|^2)$ verifies Poincar\'e inequality.'' It appears
that this gets well in higher dimension and gives a powerful way
to prove transport-entropy inequalities.

Let us introduce some notation. In the sequel, $\omega:\R\to \R$
will denote an application such that $x\mapsto \omega(x)/x$ is increasing on $(0,+\infty),$ $\omega(x)\geq 0$ for all $x\geq 0$,
and $\omega(-x)=-\omega(x)$ for all $x\in \R$. It will be
convenient to keep the notation $\omega$ to denote the application
$\R^k\to\R^k : (x_1,\ldots,x_k)\mapsto
(\omega(x_1),\ldots,\omega(x_k)).$ We will consider the metric
$d_\omega$ defined on $\R^k$ by
$$d_\omega(x,y)=|\omega(x)-\omega(y)|_2=\sqrt{\sum_{i=1}^k |\omega(x_i)-\omega(y_i)|^2},\quad x,y\in\R^k.$$

\begin{thm}\label{SGw=Tw}
Let $\mu$ be a probability measure on $\R^k$. The following
statements are equivalent.
\begin{enumerate}
\item The probability $\tilde{\mu}=\omega_\#\mu$ verifies
Poincar\'e inequality with the constant $C$:
$$\mathrm{Var}_{\tilde{\mu}}(f)\leq C\int |\nabla f|_2^2\,d\tilde{\mu},$$
for all $f:\R^k\to\R$ smooth enough. \item The probability $\mu$
verifies the following weighted Poincar\'e inequality with the
constant $C>0$:
\begin{equation}\label{SG omega}
\mathrm{Var}_\mu (f)\leq C\int \sum_{i=1}^k
\frac{1}{\omega'(x_i)^2}\left(\frac{\partial f}{\partial
x_i}(x)\right)^2\,d\mu(x),
\end{equation}
for all $f:\R^k\to\R$ smooth enough. \item The probability $\mu$
verifies the transport-entropy inequality $\mathcal{T}_c\leq H$,
with the cost function $c(x,y)=\theta_1(ad_\omega(x,y))$ for some
$a>0$, with $\theta_1(t)=\min(t^2,t)$, $t\geq 0$. More precisely,
\begin{equation}\label{T omega}
\inf_{\pi: \pi_0=\nu,\pi_1=\mu}\int_{\Rk\times\Rk}
\min\left(a^2|\omega(x)-\omega(y)|_2^2,a|\omega(x)-\omega(y)|_2\right)\,d\pi(x,y)\leq
H(\nu|\mu),
\end{equation}
for all $\nu\in \mathrm{P}(\R^k).$
\end{enumerate}
The constants $C$ and $a$ are related in the following way: (1)
implies (3) with $a=\frac{1}{\tau\sqrt{C}}$, where $\tau$ is a
universal constant, and (3) implies (1) with $C=\frac{1}{2a^2}.$
\end{thm}
\proof
The equivalence between (1) and (2) is straightforward.\\
Let us show that (1) implies (3). Indeed, according to Corollary
\ref{Poinc=Trans}, $\tilde{\mu}$ verifies the transport-entropy
inequality $\mathcal{T}_{\tilde{c}}\leq H$ with
$\tilde{c}(x,y)=\theta_1(a|x-y|_2)$, and
$a=\frac{1}{\tau\sqrt{C}}.$ Consequently, according to the
contraction Theorem \ref{contraction}, $\mu$ which is the image of
$\tilde{\mu}$ under the map $\omega^{-1}$ verifies the
transport-entropy inequality $\mathcal{T}_{c}\leq H$ where
$c(x,y)=\tilde{c}(\omega(x),\omega(y))=\theta_1(ad_\omega(x,y)).$
The proof of the converse is similar.
\endproof
\begin{defi}
When $\mu$ verifies \eqref{SG omega}, one  says that the
inequality $\SG(\omega,C)$ holds.
\end{defi}
\begin{rem}
If $f:\R^k\to\R$ let us denote by $|\nabla f|_\omega(x)$ the
``length of the gradient'' of $f$ at point $x$ with respect to the
metric $d_\omega$ defined above. By definition,
$$|\nabla f|_\omega(x)=\limsup_{y\to x}\frac{|f(y)-f(x)|}{d_\omega(x,y)},\quad x\in\R^k.$$
It is not difficult to see that $\mu$ verifies the inequality
$\SG(\omega,C)$ if and only if it verifies the following
Poincar\'e inequality
$$\mathrm{Var}_\mu(f)\leq C \int |\nabla f|^2_\omega\,d\mu,$$
for all $f$ smooth enough. So, the inequality
$\SG(\omega,\,\cdot\,)$ is a true Poincar\'e inequality for the
non-Euclidean metric $d_\omega$.
\end{rem}
So according to Theorem \ref{SGw=Tw}, the Poincar\'e inequality
\eqref{SG omega} is qualitatively equivalent to the transport cost
inequality $\eqref{T omega}.$ Those transport-entropy inequalities
are rather unusual, but can be compared to more classical
transport-entropy inequalities using the following proposition.
\begin{prop}
The following inequality holds
\begin{equation}\label{comparaison couts}
\theta_1(ad_{\omega}(x,y))\geq
\theta_1\left(\frac{a}{\sqrt{k}}\right)\sum_{i=1}^k
\theta_1\circ\omega\left(\frac{|x_i-y_i|}{2}\right),\quad x,y\in \R^k.
\end{equation}
\end{prop}
We skip the technical proof of this inequality and refer to
\cite[Lemma 2.6 and Proof of Proposition 4.2]{G08}. Let us
emphasize an important particular case. In the sequel,
$\omega_2:\R\to\R$ will be the function defined by
$\omega_2(x)=\max(x,x^2)$, for all $x\geq 0$ and such that
$\omega_2(-x)=-\omega_2(x)$, for all $x\in \R$.

\begin{cor}
If a probability measure $\mu$ on $\R^k$ verifies the inequality
$\SG(\omega_2,C)$ for some $C>0$ then it verifies the inequality
$\T_2(4\omega_2(\tau\sqrt{kC}))$, where $\tau$ is some universal
constant.
\end{cor}
In other words, a sufficient condition for $\mu$ to verify $\T_2$
is that the image of $\mu$ under the map $\omega_2$ verifies
Poincar\'e inequality. We do not know if this condition is also
necessary. \proof According to Theorem \ref{SGw=Tw}, if $\mu$
verifies $\SG(\omega_2,C)$ then it verifies the transport-entropy
inequality $\mathcal{T}_c\leq H$ with the cost function
$c(x,y)=\theta_1(ad_{\omega_2}(x,y)),$ with $a=\frac{1}{\tau
\sqrt{C}}$. According to \eqref{comparaison couts}, one has
$$\theta_1(ad_{\omega_2}(x,y))\geq
\theta_1\left(\frac{a}{\sqrt{k}}\right)\sum_{i=1}^k
\theta_1\circ\omega_2\left(\frac{|x_i-y_i|}{2}\right)=\frac{\theta_1\left(\frac{1}{\tau\sqrt{kC}}\right)}{4}|x-y|_2^2,$$
since $\theta_1\circ\omega_2(t)=t^2$, for all $t\in \R.$ Observing
that $\frac{1}{\theta_1(1/t)}=\omega_2(t)$, $t>0$, one concludes
that $\mu$ verifies the inequality
$\T_2(4\omega_2(\tau\sqrt{kC})),$ which completes the proof.
\endproof

Poincar\'e inequality has been deeply studied by many authors and
several necessary or sufficient   conditions are now available for
this functional inequality. Using the equivalence
\begin{equation}\label{SGw<->P}
\mu \text{ verifies } \SG(\omega,C)\Leftrightarrow\omega_\#\mu
\text{ verifies }\PI(C),
\end{equation}
it is an easy job to convert the known criteria for Poincar\'e
inequality into criteria for the $\SG(\omega,\,\cdot\,)$
inequality.

In dimension one, one has a necessary and sufficient condition.
\begin{prop}\label{unidimSG}
An absolutely continuous probability measure $\mu$ on $\R$  with
density $h>0$ satisfies the inequality $\SG(\omega,C)$ for some
$C>0$ if and only if
\begin{equation}
D_\omega^-=\sup_{x\leq
m}\mu(-\infty,x]\int_x^m\frac{\omega'(u)^2}{h(u)}\,du<+\infty
\quad\text{and}\quad D_\omega^+=\sup_{x\geq
m}\mu[x,+\infty)\int_m^x\frac{\omega'(u)^2}{h(u)}\,du<+\infty,
\end{equation}
where $m$ denotes the median of $\mu$. Moreover the optimal
constant $C$ denoted by $C_{\mathrm{opt}}$ verifies $$\max
(D_\omega^-,D_\omega^+)\leq C_{\mathrm{opt}}\leq 4\max
(D_\omega^-,D_\omega^+).$$
\end{prop}
\proof This proposition follows at once from the celebrated
Muckenhoupt condition for Poincar\'e inequality (see
\cite{Muck72}). According to Muckenhoupt condition, a probability
measure $d\nu=h\,dx$ having a positive density with respect to
Lebesgue measure, satisfies Poincar\'e inequality if and only if
\[D^-=\sup_{x\leq m} \nu(-\infty,x]\int_x^m \frac{1}{h(u)}\,du<+\infty\quad\text{and}\quad D^+=\sup_{x\geq m} \nu[x,+\infty)\int_m^x \frac{1}{h(u)}\,du<+\infty,\]
and the optimal constant $C_{opt}$ verifies $\max(D^-,D^+)\leq
C_{opt}\leq 4\max(D^-,D^+).$ Now, according to \eqref{SGw<->P},
$\mu$ satisfies $\SG(\omega,C)$ if and only if
$\tilde{\mu}=\omega_\#\mu$ satisfies Poincar\'e inequality with
the constant $C$. The density of $\tilde{\mu}$ is
$\tilde{h}=\frac{h\circ \omega^{-1}}{\omega'\circ \omega^{-1}}$.
Plugging $\tilde{h}$ into Muckenhoupt conditions immediately gives
us the announced result.
\endproof

Estimating the integrals defining $D^{-}$ and $D^+$ by routine
arguments, one can obtain the following workable sufficient
conditions (see \cite[Proposition 3.3]{G08} for a proof).
\begin{prop}
Let $\mu$ be an absolutely continuous probability measure on $\R$
with density $d\mu(x)=e^{-V(x)}\,dx$. Assume that the potential
$V$ is of class $\mathcal{C}^1$ and that $\omega$ verifies the following
regularity condition:
\[\frac{\omega''(x)}{\omega'^2(x)}\xrightarrow[x\to +\infty]{} 0.\]
If $V$ is such that
$$\limsup_{x\to\pm\infty}\frac{|\omega' (x)|}{|V'(x)|}<+\infty,$$
then the probability measure $\mu$ verifies the inequality
$\SG(\omega, C)$ for some $C>0$.
\end{prop}
Observe that this proposition together with the inequality
\eqref{comparaison couts} furnishes another proof of Corollary
\ref{unidim-cond} and enables us to recover (as a  particular
instance, taking $\omega=\omega_2$) Cattiaux and Guillin's
condition for $\T_2.$

In dimension $k$, it is well known that a probability
$d\nu(x)=e^{-W(x)}\,dx$ on $\R^k$ satisfies Poincar\'e inequality
if $W$ verifies the following condition:
$$\liminf_{|x|\to +\infty}\ \frac{1}{2}|\nabla W|_2^2(x)-\Delta W(x)>0.$$
This condition is rather classical in the functional inequality
literature. The interested reader can find a nice elementary proof
in \cite{BBCG08}. Using \eqref{SGw<->P} again, it is not difficult
to derive a similar multidimensional condition for the inequality
$\SG(\omega,\,\cdot\,)$ (see \cite[Proposition 3.5]{G08} for a
proof).

\section{Transport-information inequalities}\label{section transport-information}

Instead of the transport-entropy inequality
$\alpha(\mathcal{T}_c)\le H,$ Guillin, L\'eonard, Wu and Yao have
investigated in \cite{GLWY09} the following transport-information
inequality
\begin{equation}
    \alpha(\mathcal{T}_c(\nu,\mu))\le \IDV(\nu|\mu),
    \tag{$\mathcal{T}_cI$}
\end{equation}
for all $\nu\in \PX$, where the relative entropy $H(\nu|\mu)$ is
replaced by the Donsker-Varadhan information $\IDV(\nu|\mu)$ of
$\nu$ with respect to $\mu$ which  was defined at \eqref{FDV}.

This section reports some results of \cite{GLWY09}.

\subsection*{Background material from large deviation theory}
We have seen in Section \ref{section large deviations} that any
transport-entropy inequality satisfied by a probability measure
$\mu$ is connected to the large deviations of the empirical
measure $L_n=\frac 1n\sum_{i=1}^n\delta_{X_i}$ of the sequence
$(X_i)_{i\ge1}$ of independent copies of $\mu$-distributed random
variables. The link between these two notions is Sanov's theorem
which asserts that $L_n$ obey the large deviation principle with
the relative entropy $\nu\mapsto H(\nu|\mu)$ as its rate function.
In this section, we are going to play the same game replacing
$(X_i)_{i\ge1}$ with an $\X$-valued time-continuous Markov process
$(X_t)_{t\ge0}$  with a unique invariant probability measure
$\mu.$ Instead of the large deviations of $L_n,$ it is natural to
consider the large deviations of the occupation measure
$$L_t:=\frac 1t\int_0^t \delta_{X_s} ds $$
as the length of observation $t$ tends to infinity. The random
probability measure $L_t$ describes the ratio of time the random
path $(X_s)_{0\le s\le t}$ has spent in each subset of $\X.$ If
$(X_t)_{t\ge0}$ is $\mu$-ergodic, then the ergodic theorem states
that, almost surely, $L_t$ tends to $\mu$ as $t$ tends to
infinity. If in addition $(X_t)_{t\ge0}$ is $\mu$-reversible, then
$L_t$ obeys the large deviation principle with some rate function
$\IDV(\,\cdot\,|\mu).$ Roughly speaking:
\begin{equation}\label{eqL-10}
    \P(L_{t}\in
A)\underset{t\rightarrow\infty}{\asymp} e^{-t\inf_{\nu\in
A}\IDV(\nu| \mu)}.
\end{equation}
The functional $\nu\in\PX\mapsto \IDV(\nu|\mu)\in[0,\infty]$
measures some kind of difference between $\nu$ and $\mu,$ i.e.\
some quantity of information that $\nu$ brings out with respect to
the prior knowledge of $\mu.$ With the same strategy as in Section
\ref{section large deviations}, based on similar heuristics, we
are lead  to a  new class of transport inequalities which are
called transport-information inequalities.

We give now a rigorous statement of \eqref{eqL-10} which  plays
the same role as Sanov's theorem played in Section \ref{section
large deviations}.

Let the Markov process $(X_t)_{t\ge0}$ satisfy the assumptions
which have been described at Section \ref{sec-PLS}. Recall that
the Donsker-Varadhan information  $\IDV(\cdot|\mu)$ is defined at
\eqref{FDV}.

\begin{thm}[Large deviations of the occupation measure]\label{DV}
Denoting $\P_\beta(\,\cdot\,):=\int_\X \P_x(\,\cdot\,)\,d\beta(x)$
for any initial probability measure $\beta,$ suppose as in Remark
\ref{rem-10} that $((X_t)_{t\ge0}, \P_\mu)$ is a stationary
ergodic process.
\\
In addition to these assumptions on the Markov process, suppose
that the initial law $\beta\in\PX$ is absolutely continuous with
respect to $\mu$ and $d\beta/d\mu$ is in $L^2(\mu).$ Then, $L_t$
obeys the large deviation principle in $\PX$ with the rate
function $\IDV(\,\cdot\,|\mu),$ as $t$ tends to infinity. This
means that, for all Borel measurable $A\subset \PX$,
$$
-\inf_{\nu\in \mathrm{int}(A)}\IDV(\nu| \mu)\leq \liminf_{t
\to+\infty}\frac{1}{t}\log \P_\beta\left(L_{t}\in A\right)\leq
\limsup_{t \to+\infty}\frac{1}{t}\log \P_\beta\left(L_{t}\in
A\right)\leq -\inf_{\nu\in \mathrm{cl}(A)}\IDV(\nu| \mu)
$$
where $\mathrm{int}(A)$ denotes the interior of $A$ and
$\mathrm{cl}(A)$ its closure (for the weak topology).
\end{thm}
This was proved by Donsker and Varadhan \cite{DV75a}
under some conditions of absolute continuity and regularity of
$P_t(x,dy)$ but without any restriction on the initial law. The
present statement has been proved by Wu \cite[Corollary
B.11]{Wu00b}.

\subsection*{The inequalities $\WI$ and $\WWI$}

The derivation of the large deviation results for $L_t$ as $t$
tends to infinity is intimately related to the Feynman-Kac
semigroup
\begin{equation*} P_t^u g(x):=\E^x\left[ g(X_t)
\exp\left(\int_0^tu(X_s)\,ds\right)\right].
\end{equation*}
When $u$ is bounded, $(P_t^u)$ is a strongly continuous semigroup
of bounded operators on $L^2(\mu)$ whose generator is given by
$\mathcal{L}^u g=\mathcal{L} g +u g$, for all $g\in
\mathbb{D}_2(\mathcal{L}^u)=\mathbb{D}_2(\mathcal{L})$.

\begin{thm}[Deviation of the empirical mean, \cite{GLWY09}]\label{resL-12}
Let $d$ be a lower semicontinuous metric on the polish space $\X$,
$(X_t)$ be a $\mu$-reversible and ergodic Markov process on $\X$
and $\alpha$ a function in the class $\mathcal{A},$ see Definition
\ref{class A}.
\begin{enumerate}
    \item The following statements are equivalent:
\begin{enumerate}
    \item[-] $\forall\nu\in\PX,$
$\IDV(\nu|\mu)<\infty \Rightarrow \IX
d(x_o,\,\cdot\,)\,d\nu<\infty;$
    \item[-] $\E_\mu\exp\left(\lambda_o\int_0^1
d(x_o,X_t)\,dt\right)<\infty$ for some $\lambda_o>0.$
\end{enumerate}
    \item Under this condition, the subsequent statements are equivalent.
\begin{enumerate}
    \item The following inequality  holds true:
\begin{equation*}\tag{$W_1I$}
     \alpha(W_1(\nu, \mu)) \le \IDV(\nu|\mu),
\end{equation*}
for all $\nu\in\PX$.

  \item For all Lipschitz function $u$ on $\X$ with
$\|u\|_{\mathrm{Lip}}\le 1$ and all $\lambda, t\ge0,$
$$\|P_t^{\lambda u}\|_{L^2(\mu)} \le \exp\left(t[
\lambda \IX u\,d\mu+ \alpha^\mc(\lambda)]\right);$$

    \item For all Lipschitz function $u$ on $\X$ with
$\|u\|_{\mathrm{Lip}}\le 1, \ \IX u\,d\mu=0$ and all $\lambda\ge
0$,
$$
\limsup_{t\rightarrow\infty}\frac 1t\log \E_\mu \exp\left(\lambda
\int_0^t u(X_s)\,ds\right)\le \alpha^\mc(\lambda);
$$

    \item For all Lipschitz function $u$ on $\X,$  $r,t> 0$ and $\beta\in
    \PX$ such that $d\beta/d\mu\in L^2(\mu),$
$$
\P_\beta\left(\frac 1t\int_0^t u(X_s)\,ds\ge\IX
u\,d\mu+r\right)\le \left\|\frac{d\beta}{d\mu}\right\|_2
\exp\Big(-t\alpha\left(r/\|u\|_{\mathrm{Lip}}\right)\Big).
$$
 \end{enumerate}
\end{enumerate}
\end{thm}

\begin{rem}
The Laplace-Varadhan principle allows us to identify the left-hand
side of the inequality stated at (c), so that (c) is equivalent
to: For all Lipschitz function $u$ on $\X$ with
$\|u\|_{\mathrm{Lip}}\le 1, \ \IX u\,d\mu=0,$ all $\lambda\ge 0$
and all $\nu\in\PX$,
\begin{equation*}
    \lambda \IX u\,d\nu-\IDV(\nu|\mu)
    \le \alpha^\mc(\lambda).
\end{equation*}
\end{rem}

The proof of statement (1) follows the proof of \eqref{eqL-f} once
one knows that $\nu\mapsto \IDV(\nu|\mu)$ and $u\mapsto
\Upsilon(u):= \log \|P_1^{ u}\|_{L^2(\mu)}=\frac 1t\log \|P_t^{
u}\|_{L^2(\mu)}$ (for all $t>0$) are convex conjugate to each
other. The idea of the proof of the second statement is pretty
much the same as in Section \ref{section large deviations}. As was
already mentioned, one has to replace Sanov's theorem with Theorem
\ref{DV}. The equivalence of (a) and (c) can be obtained without
appealing to large deviations, but only invoking the duality of
inequalities stated at Theorem \ref{res-09} and the fact that
$\IDV$ and $\Upsilon$ are convex conjugate to each other, as was
mentioned a few lines above.

Let us turn our attention to the analogue of $\T_2.$

\begin{defi}
The probability measure $\mu\in \mathrm{P}_2(\X)$ satisfies the
inequality $\WWI(C)$ with constant $C$ if
\begin{equation*}
     W_2^2(\nu, \mu) \le C^2 \IDV(\nu|\mu),
\end{equation*}
for all $\nu\in\PX.$
\end{defi}

\begin{thm}[$\WWI,$ \cite{GLWY09}]
The statements below are equivalent.
\begin{enumerate}[(a)]
\item The probability measure $\mu\in\PX$ verifies $\WWI(C)$.

\item For any $v\in\BX$, $\|P_t^{\frac{Qv} {C^2} }\|_{L^2(\mu)}
\le e^{\frac{t} {C^2} \mu(v)},$ for all $t\ge0$ where
$\displaystyle{ Qv(x)=\inf_{y\in \X} \{v(y) + d^2(x,y)\}}.$

 \item For any $u\in\BX$,
$\|P_t^{\frac{u} {C^2}}\|_{L^2(\mu)} \le e^{\frac{t} {C^2}
\mu(Su)},$ for all $t\ge0$ where  $\displaystyle{ Su(y)=\sup_{x\in
\X} \{u(y) - d^2(x,y)\} }.$
 \end{enumerate}
\end{thm}

\begin{prop}[$\WWI$ in relation with $\LSI$ and $\PI,$ \cite{GLWY09}]\label{prop25} In the framework of the Riemannian manifold
as above, the following results  hold.
\begin{enumerate}[(a)]
\item $\LSI(C)$ implies $\WWI(C)$.

 \item $\WWI(C)$ implies  $\PI(C/2)$.

 \item Assume that $\rm{Ric} + {\rm Hess} V\ge\kappa \mathrm{Id}$ with $\kappa\in\R$. If
$C\kappa\le 2,$ Then,\\ $\WWI(C)$ implies $\LSI(2C -
C^2\kappa/2).$
\end{enumerate}
\end{prop}

Note that $\WWI(C)$ with $C\kappa\le 2$ is possible. This follows
from Part (a) and the Bakry-Emery criterion in the case
$\kappa>0,$ see Corollary \ref{resL-d}.

\begin{proof}\boulette{(a)}
    By Theorem \ref{OV}, we know that $\LSI(C)$ implies $\T_2(C).$ Hence,
    $W_2(\nu,\mu)\le\sqrt{ 2C H(\nu|\mu)}\le 2C \sqrt{\IDV(\nu|\mu)}.$

   \Boulette{(b)} The proof follows from the usual
linearization procedure. Set $\mu_\epsilon=(1+\epsilon g)\mu$ for
some smooth and compactly supported $g$ with $\int g\,d\mu=0$, we
easily get\\ 
    $\lim_{\epsilon\rightarrow 0}
\IDV(\mu_\epsilon|\mu)/\epsilon^2= \frac 14 \mathcal{E}(g,g)$ and
by Otto-Villani \cite[p.394]{OV00}, there exists $r$ such that
    $
\int g^2\,d\mu\le
\sqrt{\mathcal{E}(g,g)}\frac{W_2(\mu_\epsilon,\mu)}{\epsilon}
+\frac{r}{\epsilon}W_2^2(\mu_\epsilon,\mu).
    $
Using now $\WWI(C)$ we get
$$
\int g^2\,d\mu\le
C\sqrt{\mathcal{E}(g,g)}\sqrt{\frac{\IDV(\mu_\epsilon|\mu)}{\epsilon^2}}
+\frac{rC^2}{\epsilon}\IDV(\mu_\epsilon|\mu).
$$
Letting $\epsilon\to0$ gives the result.
    \Boulette{(c)} It is a direct application of the HWI inequality, see Corollary \ref{HWI}.
\end{proof}

\subsection*{Tensorization}
In Section \ref{sec-overview} we have already seen how
transport-entropy inequalities tensorize. We revisit
tensorization, but this time we replace the relative entropy
$H(\cdot|\mu)$ with the Donsker-Varadhan $\IDV(\cdot|\mu).$ This
will be quite similar in spirit to what as already been done in
Section \ref{sec-overview}, but we are going to use alternate
technical lemmas which will prepare the road to Section
\ref{sec-Gibbs} where a Gibbs measure will replace our product
measure. This approach which is partly based on Gozlan \&
L\'eonard \cite{GL07} is developed in Guillin, L\'eonard, Wu \&
Yao's article \cite{GLWY09}.

On the polish product space $\X^{(n)}:=\prod_{i=1}^n \X_i$ equipped
with the product measure $\mu:=\otimes_{i=1}^n \mu_i$, consider
the cost function
\begin{equation*}
\oplus_ic_i (x,y) :=\sum_{i=1}^n c_i(x_i, y_i), \quad  x,y\in
\X^{(n)}
\end{equation*}
where for  each index $i,$ $c_i$ is \lsc\ on $\X_i^2$ and assume
that for each $1\le i\le n,$ $\mu_i\in \mathrm{P}(\X_i)$ satisfies
the transport-information inequality
\begin{equation}\label{a21}
\alpha_i(\mathcal{T}_{c_i}(\nu, \mu_i))\le
I_{\mathcal{E}_i}(\nu|\mu_i),\quad  \nu\in \mathrm{P}(\X_i)
\end{equation}
where $I_{\mathcal{E}_i}(\nu|\mu_i)$ is the Donsker-Varadhan
information related to some Dirichlet form $(\mathcal{E}_i,
\dd(\mathcal{E}_i))$, and $\alpha_i$ stands in the class
$\mathcal{A}$, see Definition \ref{class A}. Define the global
Dirichlet form $\oplus_i^\mu\mathcal{E}_i$ by
\begin{equation*}
 \dd(\oplus_i^\mu\mathcal{E}_i)
    :=\left\{g\in L^2(\mu): g_i^{\widetilde{x}_i}\in
    \dd(\mathcal{E}_i), \textrm{for }
\mu\textrm{-a.e. } \widetilde{x}_i \textrm{ and } \int_{\X^{(n)}}
\sum_{i=1}^n\mathcal{E}_i(g_i^{\widetilde{x}_i},g_i^{\widetilde{x}_i})\,d\mu(x)<+\infty\right\}
\end{equation*}
where $ g_i^{\widetilde{x}_i}: x_i\mapsto
g_i^{\widetilde{x}_i}(x_i):=g(x)$ with
$\widetilde{x}_i:=(x_1,\cdots, x_{i-1}, x_{i+1}, \cdots, x_n)$
considered as fixed and
\begin{equation}\label{a24}
    \oplus_i^\mu\mathcal{E}_i(g, g)
    := \int_{\X^{(n)}} \sum_{i=1}^n
\mathcal{E}_i(g_i^{\widetilde{x}_i},g_i^{\widetilde{x}_i})\,d\mu(x),
\quad  g\in \dd(\oplus_i^\mu\mathcal{E}_i).
\end{equation}
Let $I_{\oplus_i\mathcal{E}_i}(\nu|\mu)$ be the Donsker-Varadhan
information associated with $(\oplus_i^\mu\mathcal{E}_i,
\dd(\oplus_i^\mu\mathcal{E}_i))$, see \eqref{FDV}. We denote
$\alpha_1\Box\cdots\Box \alpha_n$ the inf-convolution of
$\alpha_1,\dots,\alpha_n$ which is defined by
$$\alpha_1\Box\cdots\Box
\alpha_n(r)=\inf\{\alpha_1(r_1)+\cdots+\alpha_n(r_n);
r_1,\dots,r_n\ge0, r_1+\cdots+r_n=r\},\quad r\ge0.$$

\begin{thm}[\cite{GLWY09}]\label{thm22}
Assume that for each $i=1,\cdots, n,$ $\mu_i$ satisfies the
transport-entropy inequality \eqref{a21}.  Then, the product
measure $\mu$ satisfies the following transport-entropy inequality
\begin{equation}
    \alpha_1\Box\cdots\Box \alpha_n (\mathcal{T}_{\oplus c_i}(\nu,\mu))
    \le I_{\oplus_i\mathcal{E}_i}(\nu|\mu),
    \quad \nu\in\mathrm{P}(\X^{(n)}).
\end{equation}
\end{thm}

This result is similar to Proposition \ref{res-tensorization}. But
its proof will be different. It is based on the following
sub-additivity result for the transport cost of a product measure.

Let $(X_i)_{1\le i\le n}$ be the canonical process on $\X^{(n)}.$
For each $i,$ $\widetilde{X}_i=(X_j)_{1\le j\le n; j\not=i}$ is
the configuration without its value at index $i.$

Given a probability measure $\nu$ on $\X^{(n)}$,
$$\nu_i^{\widetilde{x}_i}=\nu(X_i\in\cdot|\widetilde{X}_i=\widetilde{x}_i)$$ denotes the regular conditional
distribution of $X_i$ knowing that
$\widetilde{X}_i=\widetilde{x}_i$ under $\nu$ and
$$
\nu_i=\nu(X_i\in\cdot)
$$
denotes the $i$-th marginal of $\nu.$

\begin{prop}[\cite{GLWY09}]\label{lem23} Let $\mu=\bigotimes_{i=1}^n\mu_i$ be a product probability measure on $\X^{(n)}.$
For all $\nu\in\mathrm{P}(\X^{(n)}),$
    $$
    \mathcal{T}_{\oplus c_i}(\nu,\mu)
    \le \int_{\X^{(n)}}\left(\sum_{i=1}^n\mathcal{T}_{c_i}(\nu_i^{\widetilde{x}_i}, \mu_i)\right)\,d\nu(x).
    $$
\end{prop}

The proof of this proposition is given in the Appendix at
Proposition \ref{res-tensliming}.

The following additivity property of the Fisher information will
be needed. It holds true even in the dependent case.

\begin{lem}[\cite{GLWY09}]\label{lem26} Let $\nu, \mu$ be probability measures on
$\X^{(n)}$ such that $I_{\oplus_i\mathcal{E}_i}(\nu|\mu)<+\infty$.
 Then,
\begin{equation*}
    I_{\oplus_i\mathcal{E}_i}(\nu|\mu)
    = \int_{\X^{(n)}}\sum_{i=1}^n I_{\mathcal{E}_i}(\nu_i^{\widetilde{x}_i}|\mu_i^{\widetilde{x}_i})\,d\nu(x).
\end{equation*}
\end{lem}

\begin{proof}[Sketch of proof]
Let $f$ be a regular enough function. This why this is only a
sketch of proof, because an approximation argument which we do not
present here, is needed to obtain the result for any $f$ in the
domain $\dd(\oplus_i\mathcal{E}_i).$
\\
Then,
    $\D{\frac{d\nu_i^{\widetilde{x}_i}}{d\mu_i^{\widetilde{x}_i}}(x_i)
    = \frac{f_i^{\widetilde{x}_i}(x_i)}{\mu_i^{\widetilde{x}_i}(f_i^{\widetilde{x}_i})}},$ $\nu$-a.s. where $f_i^{\widetilde{x}_i}$ is the
function $f$ of $x_i$ with $\widetilde{x}_i$ fixed. For $\nu$-a.e.
$\widetilde{x}_i$,
$$
    I_{\mathcal{E}_i}(\nu_i^{\widetilde{x}_i}|\mu_i^{\widetilde{x}_i})
    = \mathcal{E}_i\left(\sqrt{ \frac{f_i^{\widetilde{x}_i}}{\mu_i^{\widetilde{x}_i}(f_i^{\widetilde{x}_i})}},
        \sqrt{ \frac{f_i^{\widetilde{x}_i}}{\mu_i^{\widetilde{x}_i}(f_i^{\widetilde{x}_i})}}\right)
    = \frac 1{\mu_i^{\widetilde{x}_i}(f_i^{\widetilde{x}_i})} \mathcal{E}_i(\sqrt{f_i^{\widetilde{x}_i}},
\sqrt{f_i^{\widetilde{x}_i}}).
$$
Thus,
$$
\aligned
    \int_{\X^{(n)}}\sum_{i=1}^n
    I_{\mathcal{E}_i}(\nu_i^{\widetilde{x}_i}|\mu_i^{\widetilde{x}_i})\,d\nu(x)
    &= \int_{\X^{(n)}}f(x)\sum_{i=1}^n \frac 1{\mu_i^{\widetilde{x}_i}(f_i^{\widetilde{x}_i})}
    \mathcal{E}_i(\sqrt{f_i^{\widetilde{x}_i}},\sqrt{f_i^{\widetilde{x}_i}})\,
    d\mu(x)\\
    &= \int_{\X^{(n)}} \sum_{i=1}^n  \mathcal{E}_i(\sqrt{f_i^{\widetilde{x}_i}}, \sqrt{f_i^{\widetilde{x}_i}})\, d\mu(x)\\
    &=\oplus_i^\mu\mathcal{E}_i (\sqrt{f},\sqrt{f})\\
    &=I_{\oplus_i\mathcal{E}_i}(\nu|\mu),
\endaligned
$$
which completes the sketch of the proof.
\end{proof}

This additivity  is different from the super-additivity of the
Fisher information for a product measure obtained by Carlen
\cite{Car91}.

We are now ready to write the proof of Theorem \ref{thm22}.
\begin{proof}[Proof of Theorem \ref{thm22}] Without loss of
generality we may assume that $I(\nu|\mu)<+\infty$. By Proposition
\ref{lem23}, Jensen inequality and the definition of
$\alpha_1\Box\cdots\Box \alpha_n$,
$$
\aligned
    \alpha_1\Box\cdots\Box \alpha_n(\mathcal{T}_{\oplus
    c_i}(\nu,\mu))
    &\le \alpha_1\Box\cdots\Box \alpha_n\left(\int_{\X^{(n)}}\sum_{i=1}^n \mathcal{T}_{c_i}(\nu_i^{\widetilde{x}_i},
    \mu_i)\right)\,d\nu(x)\\
    &\le \int_{\X^{(n)}}\alpha_1\Box\cdots\Box \alpha_n\left(\sum_{i=1}^n \mathcal{T}_{c_i}(\nu_i^{\widetilde{x}_i}, \mu_i)\right)\,d\nu(x)\\
    &\le \int_{\X^{(n)}}\sum_{i=1}^n\alpha_i(
    \mathcal{T}_{c_i}(\nu_i^{\widetilde{x}_i},\mu_i))\,d\nu(x)\\
    &\le \int_{\X^{(n)}}\sum_{i=1}^n I_{\mathcal{E}_i}(\nu_i^{\widetilde{x}_i}|\mu_i)\,d\nu(x).
\endaligned
$$
The last quantity is equal to $I_{\oplus\mathcal{E}_i}(\nu|\mu)$,
by Lemma \ref{lem26}.
\end{proof}

As an example of application, let $(X^i_t)_{t\ge 0}, i=1,\cdots,
n$ be $n$ Markov processes with the same transition semigroup
$(P_t)$ and the same symmetrized Dirichlet form $\mathcal{E}$ on
$L^2(\rho)$, and conditionally independent once the initial
configuration $(X^i_0)_{i=1,\cdots, n}$ is fixed. Then
$X_t:=(X^1_t,\cdots, X^n_t)$ is a Markov process with the
symmetrized Dirichlet form given by
$$
\oplus_n^{\rho^n}\mathcal{E}(g, g)= \int \sum_{i=1}^n
\mathcal{E}(\ii gx, \ii gx)\, \rho(dx_1)\cdots\rho(dx_n).
$$

\begin{cor}[\cite{GLWY09}]\label{cor24}\
\begin{enumerate}
    \item Assume that  $\rho$ satisfies the transport-information inequality $\alpha(\mathcal{T}_c)\le I_\mathcal{E}$ on $\X$ with
$\alpha$ in the class $\mathcal{A}$.  Then $\rho^{n}$ satisfies
\begin{equation*}
n \alpha\left(\frac{\mathcal{T}_{\oplus_n c}(\nu, \rho^{n})}{n}
\right) \le I_{\oplus_n\mathcal{E}}(\nu|\rho^{n}),\quad 
\nu\in \mathrm{P}(\X^n).
\end{equation*}
    \item Suppose in particular that $\rho$ verifies $\alpha(\mathcal{T}_d)\le I_\mathcal{E}$ for the metric \lsc\ cost $d.$
    Then, for any Borel measurable $d$-Lipschitz(1) function $u,$ any initial
measure $\beta$ on $\X^n$ with $d\beta/d\rho^{n}\in L^2(\rho^{n})$
and  any $t,r>0$,
\begin{equation*}
\P_\beta\left(\frac 1n\sum_{i=1}^n \frac 1t \int_0^t u(X^i_s)\,ds
\ge \rho(u)+r\right)\le \left\|\frac{d\beta}{d\rho^{n}}\right\|_2
e^{-nt\alpha (r)}.
\end{equation*}
    \item If $\rho$ satisfies $\WWI(C):$ $\mathcal{T}_{d^2}\le C^2
    I_\mathcal{E},$ then $\rho^n$ satisfies $\WWI(C):$ $\mathcal{T}_{\oplus_{n}d^2}\le C^2
    I_{\oplus_n\mathcal{E}}.$
\end{enumerate}
\end{cor}

\begin{proof} As $\alpha^{\Box n}(r)= n \alpha(r/n),$ the first
part (1) follows from Theorem \ref{thm22}. The second part (2)
follows from Theorem \ref{resL-12} and the third part (3) is a
direct application of (1).
\end{proof}

\subsection*{Transport-information inequalities in the literature}

Several integral criteria are worked out in \cite{GLWY09}, mostly
in terms of Lyapunov functions. Note also that the assumptions in
\cite{GLWY09} are a little less restrictive than those of the
present section, in particular the Markov process might not be
reversible, but it is required that its Dirichlet form is
closable.

For further relations between $\alpha(\mathcal{T})\le I$ and other
functional inequalities, one can read the paper \cite{GLWW08} by
Guillin, L\'eonard, Wang and  Wu.

In \cite{GGW09}, Gao, Guillin and Wu have refined the
above concentration results in such a way that Bernstein
inequalities are accessible. The strategy remains the same since
it is based on the transport-information inequalities of Theorem
\ref{resL-12}, but the challenge is to express the constants in
terms of asymptotic variances. Lyapunov function conditions allow
to derive explicit rates.

An interesting feature with Theorem \ref{resL-12} is that it
allows to treat time-continuous Markov processes with jumps. This
is widely done in \cite{GGW09}. But processes with jumps might not
verify a Poincar\'e inequality even in presence of good
concentration properties, for instance when considering processes
with strong pulling-back drifts. In such cases, even the
$\alpha(\mathcal{T})\le I$ strategy fails. An alternative attack
of the problem of finding concentration estimates for the
empirical means (of Lipschitz observables) has been  performed by
Wu in \cite{Wu09} where usual  transport inequalities
$\alpha(\mathcal{T})\le H$ at the level of the Markov transition
kernel are successfully exploited.

Gibbs measures are also investigated by Gao and Wu \cite{GW07} by
means of transport-information inequalities. This is developed in
the next Section \ref{sec-Gibbs}.

\section{Transport inequalities for Gibbs measures}\label{sec-Gibbs}

We have seen transport inequalities with respect to a reference
measure $\mu$ and  how to derive  transport inequalities for the
product measure $\mu=\rho^n$ from transport inequalities for
$\rho.$ A step away from this product measure structure, one is
naturally lead to consider Markov structures. This is the case
with Gibbs measures, a description of equilibrium states in
statistical physics. Three natural problem encountered with Gibbs
measures are:
\begin{enumerate}
    \item Find criteria for the uniqueness/non-uniqueness of the solutions to the Dobrushin-Lanford-Ruelle (DLR) problem associated with the local
    specifications (see next subsection below). This uniqueness corresponds to the absence of
    phase coexistence of the physical system and the
    unique solution is our Gibbs measure $\mu.$
    \item Obtain concentration estimates for the Gibbs measures.
    \item In case of uniqueness, estimate the speed of convergence
    of the Glauber dynamics (see below) towards the equilibrium $\mu.$
\end{enumerate}
A powerful tool for investigating this program is the logarithmic
Sobolev inequality. This is known since the remarkable
contribution in 1992 of Zegarlinski \cite{Zeg92}, see also the
papers  \cite{SZ92a,SZ92b} by Stroock \& Zegarlinski. Lecture
notes on the subject have been written by Martinelli \cite{Mar97},
Royer \cite{Roy99} and Guionnet \& Zegarlinski \cite{GZ03}. An
alternate approach is to exchange logarithmic Sobolev inequalities
with Poincar\'e inequality. Indeed, in some situations both these
inequalities are equivalent  \cite{SZ92a,SZ92b}.

Recently, another approach of this problem has been proposed which
consists of replacing logarithmic Sobolev inequalities by
transport inequalities. This is what this section is about. The
main recent contributions in this area are due to Marton
\cite{M04}, Wu \cite{Wu06}, Gao \& Wu \cite{GW07} and Ma, Shen,
Wang \& Wu \cite{MSWW09}.

\subsection*{Gibbs measures}
The configuration space is $\X^\II$ where $\X$ is the spin space
and $\II$ is a countable  set of sites, for instance a finite set
with a graph structure or the lattice $\II=\mathbb{Z}^d$. A
configuration is $x=(x_i)_{i\in \II}$ where $x_i\in \X$ is the
spin value at site $i\in \II.$ The spin space might be finite, for
instance $\X=\{-1,1\} $ as in the Ising model, or infinite, for
instance $\X=S^k$ the $k$-dimensional sphere or $\X=\Rk.$ It is
assumed that $\X$ is a polish space furnished with its Borel
$\sigma$-field. Consequently, any conditional probability measure
admits a regular version.

Let us introduce some notation.  For any $i\in \II,$
$\widetilde{x}_i$ is the restriction of  the configuration $x$ to
$\{i\}^c:=\II\setminus\{i\}.$ Given $\nu\in\PXI,$ one can consider
the family of conditional probability laws of $X_i$ knowing
$\widetilde{X}_i$ where $X=(X_i)_{i\in \II}$ is the canonical
configuration. We denote these conditional laws:
$$
    \nu_i^{\widetilde{x}_i}:=\mu(X_i\in\cdot|\widetilde{X}_i=\widetilde{x}_i),\quad i\in \II, x\in\X^\II.
$$
As different projections of the same $\nu,$ these conditional laws
satisfy a collection of compatibility conditions.

The DLR problem is the following inverse problem. Consider a
family of prescribed local specifications
$\mu_i^{\widetilde{x}_i},$ $i\in \II,$ $x\in\X^\II$ which satisfy
the appropriate  collection of compatibility conditions.  Does
there exist some  $\mu\in\PXI$ whose conditional distributions are
precisely these prescribed local specifications? Is there a unique
such $\mu?$
\\
The solutions of the DLR problem are called Gibbs measures.

\subsection*{Glauber dynamics}
It is well-known that $d\mu(x)=Z^{-1}e^{-V(x)}\,dx$ where $Z$ is a
normalizing constant, is the invariant probability measure of the
Markov generator $\Delta -\nabla V\cdot\nabla.$ This fact is
extensively exploited in the semigroup approach of the Poincar\'e
and logarithmic Sobolev inequalities. Indeed, these inequalities
exhibit on their right-hand side the Dirichlet form $\mathcal{E}$
associated with this Markov generator.

This differs from the  $\mathbf{WH}$ inequalities such as $\T_1$
or $\T_2$ which do not give any role to any Dirichlet form: it is
the main reason why we didn't encounter the semigroup approach in
these notes up to now. But replacing the entropy $H$  by the
information $I(\cdot|\mu),$ one obtains transport-information
inequalities $\mathbf{WI}$ and the semigroups might have something
to tell us.

Why should one introduce some dynamics related to a Gibbs measure?
Partly because in practice the normalizing constant $Z$  is
inaccessible to computation in very high dimension, so that
simulating a Markov process $(X_t)_{t\ge0}$ admitting our Gibbs
measure as its (unique) invariant measure during a long period of
time allows us to compute estimates for average quantities.
Another reason is precisely the semigroup approach which helps us
deriving functional inequalities dealing with Dirichlet forms.
This relevant dynamics, which is often called the Glauber
dynamics, is precisely the Markov dynamics associated with the
closure of the Dirichlet form which admits our Gibbs measure as
its invariant measure.

Now, let us describe the Glauber dynamics  precisely.

Let $\mu$ be a Gibbs measure (solution of the DLR problem) with
the local specifications $\{\ii\mu{x}\in \mathrm{P}(\X);\ i\in\II,
x\in\X^\II\}.$ For each $i\in\II, x\in\X^\II,$ consider a
Dirichlet form $(\ii{\mathcal{E}}x, \dd(\ii{\mathcal{E}}x))$ and
define the global Dirichlet form $\mathcal{E}^\mu$ by
\begin{equation*}
\begin{split}
    \dd(\mathcal{E}^\mu)
    :=\Big\{f\in L^2(\mu): \textrm{for all }i\in\II, &\ f_i^{\widetilde{x}_i}\in
    \dd(\ii{\mathcal{E}}x), \textrm{for }
\mu\textrm{-a.e. } \widetilde{x}_i \\
        & \textrm{ and } \int_{\X^\II}
\sum_{i\in\II}\ii{\mathcal{E}}x(f_i^{\widetilde{x}_i},f_i^{\widetilde{x}_i})\,d\mu(x)<+\infty\Big\}
\end{split}
\end{equation*}
where $ f_i^{\widetilde{x}_i}: x_i\mapsto
f_i^{\widetilde{x}_i}(x_i):=f(x)$ with $\widetilde{x}_i$
considered as fixed and
\begin{equation}\label{eqL-07}
    \mathcal{E}^\mu(f, f)
    := \int_{\X^\II} \sum_{i\in\II}
\ii{\mathcal{E}}{x}(f_i^{\widetilde{x}_i},f_i^{\widetilde{x}_i})\,d\mu(x),
\quad  f\in \dd(\mathcal{E}^\mu).
\end{equation}

Assume that $\mathcal{E}^\mu$ is closable. Then, the Glauber
dynamics is the Markov process associated with the closure of
$\mathcal{E}^\mu.$

\begin{expl}\label{exp-typic-b}
An interesting example is given by the following extension of the
standard Example \ref{exp-typic}. Let $\X$ be a complete connected
Riemannian manifold. Consider a Gibbs measure $\mu$ solution to
the DLR problem as above. For each $i\in\II$ and $x\in\X^\II,$ the
one-site Dirichlet form $\ii{\mathcal{E}}x$ is defined for any
smooth enough function $f$ on $\X$ by
$$
\ii{\mathcal{E}}x(f,f)=\IX |\nabla f|^2\, d\ii{\mu}{x}
$$
and the global Dirichlet form $ \mathcal{E}^\mu$ which is defined
by \eqref{eqL-07} is given for any smooth enough cylindrical
function $f$ on $\X^\II$ by
$$
\mathcal{E}^\mu(f,f)=\int_{\X^\II}|\nabla_\II f|^2\,d\mu
$$
where $\nabla_\II$ is the gradient on the product manifold
$\X^\II.$ The corresponding Markov process is a family indexed by
$\II$ of interacting diffusion processes, all of them sharing the
same fixed temperature (diffusion coefficient=2). This process on
$\X^\II$ admits the Gibbs measure $\mu$ as an invariant measure.
\end{expl}

\subsection*{Dimension-free
tensorization property}

It is well known that the Poincar\'e inequality $\PI$ implies an
exponentially fast $L^2$-convergence as $t$ tends to infinity of
the law of $X_t$  to the invariant measure $\mu.$ Similarly, a
logarithmic Sobolev inequality $\LSI$ implies a stronger
convergence in entropy. Moreover, both $\PI$ and $\LSI$ enjoy a
dimension-free tensorization property which is of fundamental
importance when working in an infinite dimensional setting. This
dimension-free tensorization property is also shared by
$\T_2=\mathbf{W}_2\mathbf{H},$ see Corollary \ref{Marton T2}, and
by $\WWI,$ see Corollary \ref{cor24}-(3).

Now, suppose that each one-site specification
$\mu_i^{\widetilde{x}_i},$ for any $i\in\II$ and any $x\in\X^\II,$
satisfies a functional inequality with the dimension-free
tensorization property. One can reasonably expect that, provided
that the constants $C_i^{\widetilde{x}_i}$ in these inequalities
enjoy some uniformity property in $i$ and $x,$ any Gibbs measure
built with the local specifications $\mu_i^{\widetilde{x}_i}$ also
shares some non-trivial functional inequality (in the same family
of inequalities). This is what Zegarlinski \cite{Zeg92} discovered
with bounded spin systems and $\LSI.$ On the other hand, this
inequality (say $\PI$ or $\LSI$) satisfied by the Gibbs measures
$\mu$ entails an exponentially fast convergence as $t$ tends to
infinity of the global Glauber dynamics to $\mu.$ By standard
arguments, one can prove that this implies the uniqueness of the
invariant measure and therefore, the uniqueness of the solution of
the DLR problem.

In conclusion, some uniformity property in $i$ and $x$ of the
inequality constants $C_i^{\widetilde{x}_i}$ is a sufficient
condition for the uniqueness of the DLR problem and an
exponentially fast convergence of the Glauber dynamics.

Recently, Marton \cite{M04} and Wu \cite{Wu06} considered the
``dimension-free'' transport-entropy inequality $\T_2$ and Gao \&
Wu \cite{GW07} the ``dimension-free'' transport-information
inequality $\WWI$ in the setting of Gibbs measures.

\subsection*{Dobrushin coefficients}

Let $d$ be a lower semicontinuous metric on $\X$ and let $\PpX$ be
the set of all Borel probability measures $\rho$ on $\X$ such that
$\int_\X d^p(\xi_o,\xi)\,d\rho(\xi)<\infty$ with $p\ge 1.$ Assume
that for each site $i\in \II$ and each boundary condition
$\widetilde{x}_i,$ the specification $\mu_i^{\widetilde{x}_i}$ is
in $\PpX.$ For any $i,j\in \II,$ the Dobrushin interaction
$W_p$-coefficient is defined by
\begin{equation*}
    c_p(i,j):=\sup_{x,y;\ x=y\textrm{ off }j}
    \frac{W_p\big(\mu_i^{\widetilde{x}_i},\mu_i^{\widetilde{y}_i}\big)}{d(x_j,y_j)}
\end{equation*}
where $W_p$ is the Wasserstein metric of order $p$ on $\PpX$ which
is built on the metric $d.$ Let
$\mathbf{c}_p=(c_p(i,j))_{i,j\in\II}$ denote the corresponding
matrix which is seen as an endomorphism of $\ell ^p(\II).$ Its
operator norm is denoted by $\|\mathbf{c}_p\|_p.$

Dobrushin \cite{Dob68,Dob70} obtained a criterion for the
uniqueness of the Gibbs measure (cf.\ Question (1) above) in terms
of the coefficients $ c_1(i,j)$ with $p=1.$ It is
\begin{equation*}
    \sup_{j\in \II}\sum_{i\in \II}c_1(i,j)<1.
\end{equation*}
This quantity is $\|\mathbf{c}_1\|_1$, so that Dobrushin's
condition expresses that $\mathbf{c}_1$ is contractive on
$\ell^1(\II)$ and the uniqueness follows from a fixed point
theorem, see F\"ollmer's lecture notes \cite{Foe85} for this
well-advised proof.

\subsection*{Wasserstein metrics on $\mathrm{P}(\X^\II)$}

Let $p\ge1$ be fixed. The metric on $\X^\II$ is
\begin{equation}\label{eq-43}
    d_{p,\II}(x,y):=\Big(\sum_{i\in\II}d^p(x_i,y_i)\Big)^{1/p},\quad
x,y\in\X^\II
\end{equation}
and the Wasserstein metric $W_{p,\II}$ on $\mathrm{P}(\X^\II)$ is
built upon $d_{p,\II}.$ One sees that it corresponds to the tensor
cost $d_{p,\II}^p=\oplus_{i\in\II}d_i^p$ with an obvious notation.

Gao and Wu \cite{GW07} have proved the following tensorization
result for the Wasserstein distance between Gibbs measures.

\begin{prop}\label{res-10}
Assume that $\ii\mu{x}\in \mathrm{P}_p(\X)$ for all $i\in\II$ and
$x\in\X^\II$ and also suppose that $\|\mathbf{c}_p\|_p<1.$ Then,
$\mu\in\mathrm{P}_p(\X^\II)$ and  for all $\nu\in
\mathrm{P}_p(\X^\II),$
\begin{equation*}
    W_{p,\II}^p(\nu,\mu)\le (1-\|\mathbf{c}_p\|_p)^{-1}
    \int_{\X^\II}\sum_{i\in\II}W_p^p(\ii\nu{x},\ii\mu{x})\,d\nu(x).
\end{equation*}
\end{prop}

\begin{proof}[Sketch of proof]
As a first step, let us follow exactly the beginning  of the proof
of Proposition \ref{res-tensliming} in the Appendix. Keeping the
notation of Proposition \ref{res-tensliming}, we have
\begin{equation}\label{eqL-08}
    \E\sum_{i\in\II}d^p(U_i,V_i)=W^p_{p,\II}(\nu,\mu).
\end{equation}
and we arrive at \eqref{eqL-05} which, with $c_i=d^p,$ is
\begin{equation*}
    \E d^p(U_i,V_i)
    \le \E d^p(\widehat{U}_i, \widehat{V}_i)
    = \E
    W_p^p\Big(\nu_i^{\widetilde{U}_i},\mu_i^{\widetilde{V}_i}\Big).
\end{equation*}

 As in Marton's paper \cite{M04}, we can use the triangular
inequality for $W_p$ and the definition of $\mathbf{c}_p$ to
obtain for all $i\in\II,$
$$
    W_p\Big(\ii\nu{U},\ii\mu{V}\Big)
    \le W_p\Big(\ii\nu{U},\ii\mu{U}\Big)+ \sum_{j\in\II,\ j\ne i} c_p(i,j)
    d(U_j,V_j).
$$
Putting both last inequalities together, we see that
\begin{equation}\label{eqL-09}
    \E d^p(U_i,V_i)
    \le \E \left(W_p\Big(\ii\nu{U},\ii\mu{U}\Big)+ \sum_{j\in\II,\ j\ne i} c_p(i,j)
    d(U_j,V_j)\right)^p,\quad\textrm{ for all }i\in\II,
\end{equation}
and summing them over all the sites $i$ gives us
\begin{equation*}
    \E \sum_{i\in\II} d^p(U_i,V_i)
    \le \E \sum_{i\in\II}\left(W_p\Big(\ii\nu{U},\ii\mu{U}\Big)+ \sum_{j\in\II,\ j\ne i} c_p(i,j)
    d(U_j,V_j)\right)^p.
\end{equation*}
Consider  the norm $\|A\|:=(\E\sum_{i\in\II}|A_i|^p)^{1/p}$ of the
random vector $A=(A_i)_{i\in\II}.$ With $A_i=d(U_i,V_i)$ and
$B_i=W_p\Big(\ii\nu{U},\ii\mu{U}\Big),$ this inequality is simply
\begin{equation*}
    \|A\|\le \|\mathbf{c}_pA+B\|,
\end{equation*}
since $c_p(i,i)=0$ for all $i\in\II.$ This implies that
$$
(1-\|\mathbf{c}_p\|_p)\|A\|\le \|B\|
$$
which, with \eqref{eqL-08}, is the announced result.

Similarly to  the first step of the proof of Proposition
\ref{res-tensliming}, this proof contains a measurability bug and
one has to correct it exactly as in the complete proof of
Proposition \ref{res-tensliming}.
\end{proof}

Recall that the global Dirichlet form $\mathcal{E}^\mu$ is defined
at \eqref{eqL-07}. The corresponding Donsker-Varadhan information
is defined by
$$
\IDV_{\mathcal{E}^\mu}(\nu|\mu)=\left\{\begin{array}{ll}
  \mathcal{E}_\mu (\sqrt{f},\sqrt{f})& \textrm{if }\nu=f\mu\in\PXI, f\in\dd(\mathcal{E}_\mu) \\
  +\infty & \textrm{otherwise}. \\
\end{array} \right.
$$
Similarly, we define for each $i\in\II$ and $x\in\X^\II,$
$$
\IDV_{\ii{\mathcal{E}}{x}}(\rho|\ii\mu{x})=\left\{\begin{array}{ll}
  \ii{\mathcal{E}}{x} (\sqrt{g},\sqrt{g})& \textrm{if }\rho=g \ii\mu{x}\in\PX, g\in\dd(\ii{\mathcal{E}}{x}) \\
  +\infty & \textrm{otherwise}. \\
\end{array} \right.
$$

We are now ready to present a result of tensorization of one-site
$\WWI$ inequalities in the setting of Gibbs measures.

\begin{thm}[\cite{GW07}]\label{res-11} Assume that for each site $i\in\II$ and any
configuration $x\in\X^\II$ the local specifications are in
$\mathrm{P}_2(\X)$ and  satisfy the following one-site $\WWI$
inequality
\begin{equation*}
    W_2^2(\rho,\ii\mu{x})\le C^2 \IDV_{\ii{\mathcal{E}}{x}}(\rho|\ii\mu{x}),
    \ \rho\in \mathrm{P}_2(\X),
\end{equation*}
the constant $C$ being uniform in $i$ and $x.$
\\
It is also assumed that the Dobrushin $W_2$-coefficients satisfy
$\|\mathbf{c}_2\|_2<1.$
\\
Then, any Gibbs measure $\mu$ is in $\mathrm{P}_2(\X^\II)$ and
satisfies the following $\WWI$ inequality:
\begin{equation*}
    W^2_{2,\II}(\nu,\mu) \le   \frac{C^2}{1-\|\mathbf{c}_2\|_2} \IDV_{\mathcal{E}^\mu}(\nu|\mu),\quad 
    \nu\in\mathrm{P}_2(\X^\II).
\end{equation*}
\end{thm}

\begin{proof}
By Proposition \ref{res-10}, we have for all
$\nu\in\mathrm{P}_2(\X^\II)$

\begin{equation*}
    W_{2,\II}^2(\nu,\mu)\le (1-\|\mathbf{c}_2\|_2)^{-1}
    \int_{\X^\II}\sum_{i\in\II}W_2^2(\ii\nu{x},\ii\mu{x})\,d\nu(x).
\end{equation*}
Since the local specifications satisfy a uniform inequality
$\WWI,$ we obtain
\begin{align*}
    W_{2,\II}^2(\nu,\mu)
    &\le \frac{C^2}{1-\|\mathbf{c}_2\|_2}
    \int_{\X^\II}\sum_{i\in\II}\IDV_{\ii{\mathcal{E}}{x}}(\ii\nu{x}|\ii\mu{x})\,d\nu(x)\\
    &= \frac{C^2}{1-\|\mathbf{c}_2\|_2} \IDV_{\mathcal{E}^\mu}(\nu|\mu)
\end{align*}
where the last equality is Lemma \ref{lem26}.
\end{proof}

We decided to restrict our attention to the case $p=2$ because of
its free-dimension property, but a similar result still holds with
$p>1$ under the additional requirement that $\II$ is a finite set.

As a direct consequence, under the assumptions of Theorem
\ref{res-11}, $\mu$ satisfies a fortiori the $\WI$ equality
\begin{equation*}
    W_{1,\II}^2(\nu,\mu)
    \le \frac{C^2}{1-\|\mathbf{c}_2\|_2}
    \IDV_{\mathcal{E}^\mu}(\nu|\mu),
    \quad  \nu\in\mathrm{P}_1(\X^\II).
\end{equation*}
Therefore, we can derive from Theorem \ref{resL-12} the following
deviation estimate for the Glauber dynamics.

\begin{cor}[Deviation of the Glauber dynamics]
Suppose that $\mu$ is the unique Gibbs measure (for instance if
$\|\mathbf{c}_1\|_1<1$) and that the assumptions of Theorem
\ref{res-11} are satisfied. Then, the Glauber dynamics
$(X_t)_{t\ge0}$ verifies the following deviation inequality.
\\
For all  $d_{1,\II}$-Lipschitz function $u$ on $\X^\II$ (see
\eqref{eq-43} for the definition of $d_{1,\II}$), for all $r,t> 0$
and all $\beta\in
    \PXI$ such that $d\beta/d\mu\in L^2(\mu),$
$$
\P_\beta\left(\frac 1t\int_0^t u(X_s)\,ds\ge\int_{\X^\II}
u\,d\mu+r\right)\le \left\|\frac{d\beta}{d\mu}\right\|_2
\exp\Big(-
\frac{1-\|\mathbf{c}_2\|_2}{C^2\|u\|_{\mathrm{Lip}}^2}\,tr^2
\Big).
$$
\end{cor}

\section{Free transport inequalities}
The semicircular law is the probability distribution $\sigma$ on $\R$ defined by
$$d\sigma(x)=\frac{1}{2\pi} \sqrt{4-x^2}\1_{[-2,2]}(x)\,dx.$$
This distribution plays a fundamental role in the asymptotic
theory of Wigner random matrices.
\begin{defi}[Wigner matrices]
Let $N$ be a positive integer; a (complex) $N\times N$ Wigner
matrix $M$ is an Hermitian random matrix such that the entries
$M(i,j)$ with $i<j$ are i.i.d $\mathbb{C}$-valued random variables
with $\E[M(i,j)]=0$ and $\E[|M(i,j)|^2]=1$ and such that the
diagonal entries $M(i,i)$ are i.i.d centered real random variables
independent of the off-diagonal entries and having finite
variance. When the entries of $M$ are Gaussian random variables
and $\E[M(1,1)^2]=1$, $M$ is referred to as the Gaussian Unitary
Ensemble (GUE).
\end{defi}
Let us recall the famous Wigner theorem (see e.g
\cite{Anderson-Guionnet-Zeitouni10} or \cite{Guionnet09} for a
proof).

\begin{thm}[Wigner theorem]
Let $(M_{N})_{N\geq 0}$ be a sequence of complex Wigner matrices
such that $\max_{N\geq 0} (\E[M_{N}(1,1)^2])<+\infty$ and let
$L_{N}$ be the empirical distribution of
$X_{N}:=\frac{1}{\sqrt{N}}M_{N}$, that is to say
$$L_{N}=\frac{1}{N}\sum_{i=1}^N \delta_{\lambda^N_{i}},$$
where $\lambda^N_{1}\leq \lambda^N_{2}\leq \ldots\leq \lambda^N_{N}$ are the (real) eigenvalues of $X_{N}$. Then the sequence of random probability measures $L_{N}$ converges almost surely to the semicircular law (for the weak topology).
\end{thm}

In \cite{BV01}, Biane and Voiculescu have obtained the following
transport inequality for the semicircular distribution $\sigma$
\begin{equation}\label{Biane-Voiculescu}
\mathcal{T}_{2}(\nu,\sigma)\leq 2 \widetilde{\Sigma}(\nu|\sigma),
\end{equation}
which holds for all $\nu\in \mathrm{P}(\R)$ with compact support
(see \cite[Theorem 2.8]{BV01}). The functional appearing in the
left-hand side of \eqref{Biane-Voiculescu} is the \emph{relative
free entropy} defined as follows:
$$
    \widetilde{\Sigma}(\nu|\sigma)=E(\nu)-E(\sigma),
$$
where $$E(\nu)=\int \frac{x^2}{2}\,d\nu(x)- \iint \log
(|x-y|)\,d\nu(x)d\nu(y).$$ The relative free entropy
$\widetilde{\Sigma} (\,\cdot\,|\sigma)$ is a natural candidate to
replace the relative entropy $H(\,\cdot\,|\sigma)$, because it
governs the large deviations of $L_{N}$ when $M_{N}$ is drawn from
the GUE, as was shown by Ben Arous and Guionnet in \cite{BAG97}.
More precisely, we have the following: for every open (resp.
closed) subset $O$ (resp. $F$) of $\mathrm{P}(\R)$,
\begin{gather*}
\liminf_{N\to \infty} \frac{1}{N^2} \log\P\left(L_{N}\in O\right)\geq -\inf \{\widetilde{\Sigma}(\nu|\sigma); \nu\in O\},\\
\limsup_{N\to \infty} \frac{1}{N^2} \log\P\left(L_{N}\in F\right)\leq -\inf \{\widetilde{\Sigma}(\nu|\sigma); \nu\in F\}.
\end{gather*}

Different approaches were considered to prove
\eqref{Biane-Voiculescu} and to generalize it to other compactly
supported probability measures. The original proof by Biane and
Voiculescu was inspired by \cite{OV00}. Then Hiai, Petz and Ueda
\cite{HPU04} proposed a simpler proof relying on Ben Arous and
Guionnet large deviation principle. Later Ledoux gave alternative
arguments based on a free analogue of the Brunn-Minkowski
inequality. Recently, Ledoux and Popescu \cite{P07,LP09} proposed
yet another approach using optimal transport tools. Here, we will
sketch the proof of Hiai, Petz and Ueda.

We need to introduce some supplementary material. Define
$\mathcal{H}_{N}$ as the set of Hermitian $N\times N$ matrices. We
will identify $\mathcal{H}_N$ with the space $\R^{N^2}$ using the
map
\begin{equation}\label{hermitian}
H\in \mathcal{H}_N \mapsto \left((H(i,i))_i, (\mathrm{Re}(H(i,j)))_{i<j}, (\mathrm{Im}(H(i,j)))_{i<j}\right).
\end{equation}
The Lebesgue measure $dH$ on $\mathcal{H}_N$ is
$$
    dH:=\prod_{i=1}^N dH_{i,i} \prod_{i<j} d\left(\mathrm{Re}(H_{i,j})\right)\, \prod_{i<j}d\left(\mathrm{Im}(H_{i,j})\right).
$$
For all continuous function $Q:\R\to\R$, let us define the probability measure $P_{N,Q}$ on $\mathcal{H}_N$ by
\begin{equation}\label{PNQ}
\int f\,dP_{{N,Q}}:=\frac{1}{Z_{N}(Q)}\int
f(H)e^{-N\mathrm{Tr}(Q(H))}\,dH
\end{equation}
for all bounded and measurable $f:\mathcal{H}_{N}\to \R$, where
$Q(H)$ is defined using the basic functional calculus, and
$\mathrm{Tr}$ is the trace operator. In particular, when $M_{N}$
is drawn from the GUE, then it is easy to check that the law of
$X_{N}=N^{-1/2}M_{N}$ is $P_{N, x^2/2}$.

The following theorem is due to Ben Arous and Guionnet.
\begin{thm}\label{BAG}
Assume that $Q:\R\to\R$ is a continuous function such that
\begin{equation}\label{condition sur Q}
\liminf_{|x|\to\infty} \frac{Q(x)}{\log |x|} >2,
\end{equation}
and for all $N\geq 1$ consider a random matrix $X_{N,Q}$
distributed according to $P_{N,Q}$. Let
$\lambda_{1}^N\leq\ldots\leq\lambda_{N}^N$ be the ordered
eigenvalues of $X_{N,Q}$ and define $L_{N}=\frac{1}{N}\sum_{i=1}^N
\delta_{\lambda_{i}^N}\in \mathrm{P}(\R)$. The sequence of random
measures $(L_N)_{N\geq 1}$ obeys a large deviation principle, in
$\mathrm{P}(\R)$ equipped with the weak topology, with speed $N^2$
and the good rate function $I_{Q}$ defined by
$$
    I_{Q}(\nu)=E_{Q}(\nu)-\inf_{\nu}E_{Q}(\nu),\quad \nu\in \mathrm{P}(\R)
$$
where
$$
    E_{Q}(\nu)=\int Q(x)\,d\nu(x)-\iint \log |x-y|\,d\nu(x)d\nu(y),\quad \nu\in \mathrm{P}(\R).
$$
In other words, for all open (resp. closed) $O$ (resp. $F$) of $\mathrm{P}(\R)$, it holds
\begin{gather*}
\liminf_{N\to \infty} \frac{1}{N^2} \log\P\left(L_{N}\in O\right)\geq -\inf \{I_{Q}(\nu); \nu\in O\},\\
\limsup_{N\to \infty} \frac{1}{N^2} \log\P\left(L_{N}\in F\right)\leq -\inf \{I_{Q}(\nu); \nu\in F\}.
\end{gather*}
Moreover, the functional $I_{Q}$ admits a unique minimizer denoted
by $\mu_{Q}$. The probability measure $\mu_{Q}$ is compactly
supported and is characterized by the following two conditions:
there is a constant $C_{Q}\in \R$ such that
$$
    Q(x)\geq 2\int \log|x-y|\,d\mu_{Q}(y)+C_{Q},\quad\text{for all } x\in \R
$$
and
$$
    Q(x)=2\int \log|x-y|\,d\mu_{Q}(y)+C_{Q},\quad\text{for all } x\in \mathrm{Supp}(\mu_{Q}).
$$
Finally, the asymptotic behavior of the normalizing constant
$Z_{N}(Q)$ in \eqref{PNQ} is given by:
$$
    \lim_{N\to\infty}\frac{1}{N^2}\log Z_{N}(Q)=-E_{Q}(\mu_{Q})=-\inf_{\nu}E_{Q}(\nu).
$$
\end{thm}

\begin{rem}
Let us make a few comments on this theorem.
\begin{enumerate}
    \item When $Q(x)=x^2/2$, then $\mu_{Q}$ is the semicircular law
$\sigma$.

    \item For a general $Q$, one has the identity
$I_{Q}(\nu)=E_{Q}(\nu)-E_{Q}(\mu_{Q})$. So, to be coherent with
the notation given at the beginning of this section, we will
denote $I_{Q}(\nu)=\widetilde{\Sigma}(\nu|\mu_{Q})$ in the sequel.

    \item As a by-product of the large deviation principle, we can
conclude that the sequence of random measures $L_{N}$ converges
almost surely to $\mu_{Q}$ (for the weak topology). When
$Q(x)=x^2/2$, this provides a proof of Wigner theorem in the
particular case of the GUE.
\end{enumerate}
\end{rem}

Now we can prove the transport inequality
\eqref{Biane-Voiculescu}.

\proof[Proof of\eqref{Biane-Voiculescu}]
 We will prove the
inequality \eqref{Biane-Voiculescu} only in the case where $\nu$
is a probability measure with support included in $[-A;A]$, $A>0$
and such that the function $$S_{\nu}(x):=2\int
\log|x-y|\,d\nu(y)$$ is finite and continuous over $\R$. The
general case is then obtained by approximation (see \cite{HPU04}
for explanations).

\noindent\textbf{First step.} To prove that
$\mathcal{T}_{2}(\nu,\sigma)\leq 2\widetilde{\Sigma}(\nu|\sigma)$,
the first idea is to use Theorem \ref{BAG} to provide a matrix
approximation of $\nu$ and $\sigma$.

Let $Q_\nu:\R\to\R$ be a continuous function such that
$Q_\nu=S_\nu$ on $[-A,A]$, $Q_\nu\geq S_\nu$ and
$Q_\nu(x)=\frac{x^2}{2}$ when $|x|$ is large. Let $X_{N,\nu}$,
$N\geq 1$ be a sequence of random matrices distributed according
to the probability $P_{N,\nu}$ associated to $Q_\nu$ in
\eqref{PNQ} (we shall write in the sequel $P_{N,\nu}$ instead of
$P_{N,Q_\nu}$).  The characterization of the equilibrium measure
$\mu_{Q_{\nu}}$ easily implies that $\mu_{Q_{\nu}}=\nu$. So, the
random empirical measures $L_{N,\nu}$ of $X_{N,\nu}$ follows the
large deviation principle with the good rate function
$\widetilde{\Sigma}(\,\cdot\,|\nu)$. In particular, $L_{N,\nu}$
converges almost surely to $\nu$ (for the weak topology). Let us
consider the probability measure $\nu_{N}$ defined  for all
bounded measurable function $f$ by
$$\int f\,d\nu_{N}:=\E\left[\int f\,dL_{N,\nu}\right].$$
The almost sure convergence of $L_{N,\nu}$ to $\nu$ easily implies
that $\nu_{N}$ converges to $\nu$ for the weak topology. We do the
same construction with $Q_{\sigma}(x)=x^2/2$ yielding a sequence
$\sigma_{N}$ converging to $\sigma$ (note that in this case, the
sequences $X_{N,\sigma}$ and $P_{N,\sigma}$ correspond to the GUE
rescaled by a factor $\sqrt{N}$).

\noindent\textbf{Second step.} Now we compare the Wasserstein
distance between $\nu_{N}$ and $\sigma_{N}$ to the one between
$P_{N,\nu}$ and $P_{N,\sigma}$. To define the latter, we equip
$\mathcal{H}_{N}$ with the Frobenius norm defined as follows:
$$
    \|A-B\|^2_{F}=\sum_{i=1}^N\sum_{j=1}^N |A(i,j)-B(i,j)|^2.
$$
By definition, if $P_{1},P_{2}$ are probability measures on $\mathcal{H}_{N}$, then
$$
    \mathcal{T}_{2}(P_{1},P_2):=\inf \E\left[\|X-Y\|_{F}^2\right],
$$
where the infimum is over all the couples of $N\times N$ random
matrices $(X,Y)$ such that $X$ is distributed according to $P_{1}$
and $Y$ according to $P_{2}$. According to the classical
Hoffman-Wielandt inequality (see e.g \cite{Horn-Johnson94}), if
$A,B\in \mathcal{H}_{N}$ then,
$$
    \sum_{i=1}^N |\lambda_{i}(A)-\lambda_{i}(B)|^2\leq\|A-B\|_{F}^2,
$$
where $\lambda_{1}(A)\leq \lambda_{2}(A)\leq \ldots\leq
\lambda_{N}(A)$ (resp. $\lambda_{1}(B)\leq \lambda_{2}(B)\leq
\ldots\leq \lambda_{N}(B)$) are the eigenvalues of $A$ (resp. $B$)
in increasing order. So, if $(X_{N,\nu},X_{N,\sigma})$ is an
optimal coupling between $P_{N,\nu}$ and $P_{N,\sigma}$, we have
\begin{align*}
    \mathcal{T}_{2}(P_{N,\nu},P_{N,\sigma})=\E[\|X_{N,\nu}-X_{N,\sigma}\|_{F}^2]
    &\geq \E \left[\sum_{i=1}^N|\lambda_{i}(X_{N,\nu})-\lambda_{i}(X_{N,\sigma})|^2\right]\\
    &=N\E\left[\iint |x-y|^2\,dR_N\right],
\end{align*}
where $R_N$ is the random probability measure on $\R\times \R$
defined by
$$
    R_N:=\frac{1}{N}\sum_{i=1}^N \delta_{(\lambda_i(X_{N,\nu}),\lambda_i(X_N,\sigma))}.
$$
It is clear that $\pi_N:=\E[R_N]$ has marginals $\nu_N$ and
$\sigma_N$. Hence, applying Fubini theorem in the above inequality
yields
$$
    \mathcal{T}_2(P_{N,\nu},P_{N,\sigma})\geq N\mathcal{T}_2(\nu_N,\sigma_N).
$$

\noindent\textbf{Third step.} If we identify the space
$\mathcal{H}_N$ to the space $\R^{N^2}$ using the map defined in
\eqref{hermitian}, then $P_{N,\sigma}$ is a product of Gaussian
measures:
$$
    P_{N,\sigma}=\mathcal{N}(0,1/N)^N\otimes
\mathcal{N}(0,1/(2N))^{N(N-1)/2} \otimes
\mathcal{N}(0,1/(2N))^{N(N-1)/2}.
$$
Each factor verifies Talagrand inequality $\T_2$ (with the
constant $2/N$ or $1/N$). Therefore, using the dimension-free
tensorization property of $\T_2$, it is easy to check that
$P_{N,\sigma}$ verifies the transport inequality
$\mathcal{T}_c\leq 2N^{-1}H$ on $\mathcal{H}_N$, where the cost
function $c$ is defined by
$$
    c(A,B):=\sum_{i=1}^N |A(i,i)-B(i,i)|^2+2\sum_{i<j}|A(i,j)-B(i,j)|^2=\|A-B\|_{F}^2.
$$
As a conclusion, for all $N\geq 1$, the inequality
$\mathcal{T}_2(P_{N,\nu},P_{N,\sigma})\leq
2N^{-1}H\left(P_{N,\nu}| P_{N,\sigma}\right)$ holds. Using Step 2,
we get
$$
    \mathcal{T}_2(\nu_N,\sigma_N)\leq \frac{2}{N^2}H\left(P_{N,\nu}| P_{N,\sigma}\right),\quad N\geq 1.
$$

\noindent\textbf{Fourth step.} The last step is devoted to the
computation of the limit of $N^{-2}H\left(P_{N,\nu}|
P_{N,\sigma}\right)$ when $N$ goes to $\infty$. We have
\begin{align*}
    \frac{H\left(P_{N,\nu}| P_{N,\sigma}\right)}{N^2}
    &=\frac{1}{N^2}\log Z_N(Q_\sigma)-\frac{1}{N^2}\log Z_N(Q_\nu)+\frac{1}{N}\int \mathrm{Tr}\left(Q_\nu(A)-\frac{1}{2}A^2\right)\,dP_{N,\nu}\\
    &=\frac{1}{N^2}\log Z_N(Q_\sigma)-\frac{1}{N^2}\log Z_N(Q_\nu)+\int Q_\nu(x)-\frac{x^2}{2}\,d\nu_N(x)
\end{align*}
Using Theorem \ref{BAG} and the convergence of $\nu_N$ to $\nu$,
it is not difficult to see that the right-hand side tends to
$\widetilde{\Sigma}(\nu|\sigma)$ when $N$ goes to $\infty$
(observe that the function $x\mapsto Q_\nu(x)-x^2/2$ is continuous
and has a compact support). Since $\mathcal{T}_2$ is lower
semicontinuous (this is a direct consequence of the Kantorovich
dual equality, see Theorem \ref{res-01}),
$\mathcal{T}_2(\nu,\sigma)\leq \liminf_{N\to \infty}
\mathcal{T}_2(\nu_N,\sigma_N)$, which completes the proof.
\endproof
\begin{rem}
\begin{enumerate}
    \item It is possible to adapt the preceding proof to show that
probability measures $\mu_Q$ with $Q''\geq \rho$, with $\rho>0$
verify the transport inequality $\mathcal{T}_2(\nu,\mu_Q)\leq
\frac{2}{\rho}\widetilde{\Sigma}(\nu|\mu_Q)$, for all $\nu\in
\mathrm{P}(\R),$  see \cite{HPU04}.

    \item The random matrix
approximation method can be applied to obtain a free analogue of
the logarithmic Sobolev inequality, see \cite{B03}. It has been
shown by Ledoux in \cite{L05} that a free analogue of Otto-Villani
theorem holds. Ledoux and Popescu have also obtained in
\cite{LP09} a free HWI inequality.
\end{enumerate}
\end{rem}

\section{Optimal transport is a tool for proving other functional inequalities}\label{section optimal transport is a tool}

We already saw in Section \ref{section uniformly convex} that the
logarithmic Sobolev inequality can be derived by means of the
quadratic optimal transport. It has  been discovered by Barthe,
Cordero-Erausquin, McCann, Nazaret and Villani, among others, that
this is also true for other well-known  functional inequalities
such as Pr\'ekopa-Leindler, Brascamp-Lieb and Sobolev
inequalities, see \cite{Bar97, Bar98, CMS01, CMS06, CNV04, McC94}.

In this section, we do an excursion a step away from transport
inequalities and visit Brunn-Minkowski and Pr\'ekopa-Leindler
inequalities.  We are going to sketch their proofs. Our main tool
will be the Brenier map which was described at Theorem
\ref{res-06}. For a concise and enlightening discussion on this
topic, it is worth reading Villani's exposition in \cite[Ch.\!
6]{Vill}.

\subsection*{The Pr\'ekopa-Leindler inequality} It is a functional
version of Brunn-Minkowski inequality which has been proved
several times and named after the papers by Pr\'ekopa
\cite{Pre73} and Leindler \cite{Lei72}.

\begin{thm}[Pr\'ekopa-Leindler inequality]\label{PL}
Let $f,g,h$ be three nonnegative integrable functions on $\Rk$ and
$0\le\lambda\le1$ be such that for all $x,y\in\Rk,$
\begin{equation*}
    h((1-\lambda)x+\lambda y)\ge f(x)^{1-\lambda}g(y)^\lambda.
\end{equation*}
Then,
\begin{equation*}
    \int_{\Rk}h(x)\,dx\ge \left(\int_{\Rk}f(x)\,dx\right)^{1-\lambda}\left(\int_{\Rk}g(x)\,dx\right)^\lambda.
\end{equation*}
\end{thm}
 The next proof comes from Barthe's PhD thesis \cite{Bar97}.

 \begin{proof}
Without loss of generality, assume that $f, g$ and $h$ are
probability densities. Pick another probability density $p$ on
$\Rk,$ for instance the indicator function of the unit cube
$[0,1]^d.$ By Theorem \ref{res-06}, there exist two Brenier maps
$\nabla \Phi_1$  and $\nabla \Phi_2$ which transport $p$ onto $f$
and $p$ onto $g,$ respectively. Since $\Phi_1$ is a convex
function, it admits an Alexandrov Hessian (defined almost
everywhere) $\nabla^2_A\phi_1$ which is nonnegative definite.
Similarly, for $\Phi_2$ and $\nabla^2_A\phi_2.$ The change of
variable formula leads us to the Monge-Amp\`ere equations
\begin{equation*}
    f(\nabla\phi_1(x))\, \det(\nabla^2_A\phi_1(x))=1,\qquad
    g(\nabla\phi_2(x))\, \det(\nabla^2_A\phi_2(x))=1
\end{equation*}
for almost all $x\in[0,1]^d.$ Defining
$\phi=(1-\lambda)\phi_1+\lambda\phi_2,$ one obtains
\begin{align*}
  &\int_{\Rk}h(y)\,dy\\
  &\ge \int_{[0,1]^d}h(\nabla\phi(x))\,\det(\nabla^2_A\phi(x))\,dx \\
  &\stackrel{\mathrm{(i)}}{\ge} \int_{[0,1]^d}h\Big((1-\lambda)\nabla\phi_1(x)+\lambda\nabla\phi_2(x)\Big)\,
  \Big[\det(\nabla^2_A\phi_1(x))\Big]^{1-\lambda} \Big[\det(\nabla^2_A\phi_2(x))\Big]^{\lambda}\,dx\\
  &\stackrel{\mathrm{(ii)}}{\ge} \int_{[0,1]^d}f(\nabla\phi_1(x))^{1-\lambda}g(\nabla\phi_2(x))^\lambda\,
  \Big[\det(\nabla^2_A\phi_1(x))\Big]^{1-\lambda} \Big[\det(\nabla^2_A\phi_2(x))\Big]^{\lambda}\,dx \\
  &\stackrel{\mathrm{(iii)}}{=} \int_{[0,1]^d} 1\,dx=1
\end{align*}
where inequality (i) follows from the claim below, inequality (ii)
uses the assumption on $f,g$ and $h$ and the equality (iii) is a
direct consequence of the above Monge-Amp\`ere equations.

\noindent \textit{Claim.}\ The function $S\in
\mathcal{S}_+\mapsto\log\det (S)\in[-\infty,\infty)$ is a concave
function on the convex cone $\mathcal{S}_+$ of nonnegative
definite symmetric matrices.
 \end{proof}

The decisive trick of this proof is to take advantage of  the
concavity of $\log\det,$ once it is noticed that the Hessian of
the \emph{convex} function $\phi,$ which gives rise to the Brenier
map $\nabla\phi,$ belongs to $\mathcal{S}_+.$

As a corollary, one obtains the celebrated Brunn-Minkowski
inequality.

\begin{cor}[Brunn-Minkowski inequality]
For all  $A,B$ compact subsets of $\Rk,$
\begin{equation*}
    \mathrm{vol}^{1/d}(A+B)\ge\mathrm{vol}^{1/d}(A)+\mathrm{vol}^{1/d}(B)
\end{equation*}
where $\mathrm{vol}^{1/d}(A):=\left(\int_A dx\right)^{1/d}$ and
$A+B:=\{a+b;a\in A,b\in B\}.$
\end{cor}

\begin{proof}
For any $0\le\lambda\le1,$ the functions $f=\mathbf{1}_A,$ $g=\mathbf{1}_B$ and
$h=\mathbf{1}_{[(1-\lambda)A+\lambda B]}$ satisfy Theorem \ref{PL}'s
assumptions. Therefore, we have $\int h\ge(\int
f)^{1-\lambda}(\int g)^\lambda$ which is
$\mathrm{vol}((1-\lambda)A+\lambda B)\ge
\mathrm{vol}(A)^{1-\lambda}\mathrm{vol}(B)^\lambda.$ It follows
that
$\mathrm{vol}(A+B)=\mathrm{vol}((1-\lambda)\frac{A}{1-\lambda}+\lambda
\frac{B}{\lambda})\ge
\mathrm{vol}(\frac{A}{1-\lambda})^{1-\lambda}\mathrm{vol}(\frac{B}{\lambda})^\lambda$
which is equivalent to $\mathrm{vol}^{1/d}(A+B)\ge
\left(\frac{\mathrm{vol}^{1/d}(A)}{1-\lambda}\right)^{1-\lambda}\left(\frac{\mathrm{vol}^{1/d}(B)}{\lambda}\right)^\lambda.$
It remains to optimize in $\lambda.$
\end{proof}

\appendix

\section{Tensorization of transport costs}
 During the proof of the tensorization property of
 transport-entropy inequalities at Proposition
 \ref{res-tensorization}, we made use of the following
 tensorization property of transport costs. A detailed  proof of
 this property in the literature being unknown to the authors, we find it useful to present
 it here.

  \begin{prop}\label{res-tenstrans}
  We assume that the cost functions $c_1$ and $c_2$ are lower
  semicontinous on the products of polish spaces $\X_1\times\Y_1$
  and $\X_2\times\Y_2,$ respectively. Then,
for all $\nu\in\mathrm{P}(\Y_1\times\Y_2),$
$\mu_1\in\mathrm{P}(\X_1)$ and $\mu_2\in\mathrm{P}(\X_2),$ we have
\begin{equation}\label{eq-41b}
     \mathcal{T}_{c_1\oplus c_2}(\nu,\mu_1\otimes\mu_2)
 \leq
 \mathcal{T}_{c_1}(\nu_1,\mu_1)+\int_{\Y_1}\mathcal{T}_{c_2}(\nu_2^{y_1},\mu_2)\,d\nu_1(y_1)
\end{equation}
where $\nu$ disintegrates as follows: $
    d\nu(y_1,y_2)=d\nu_1(y_1)d\nu_2^{y_1}(y_2).
$
 \end{prop}

\begin{proof}
One first faces a nightmare of notation. It might be helpful to
introduce random variables and see
 $\pi\in\mathrm{P}(\X\times\Y)=\mathrm{P}(\X_1\times\X_2\times\Y_1\times\Y_2)$
as the law of $(X_1,X_2,Y_1,Y_2).$ One denotes
 $\pi_1=\mathcal{L}(X_1,Y_1),$
 $\pi_2^{x_1,y_1}\mathcal{L}(X_2,Y_2 |  X_1=x_1, Y_1=y_1),$
 $\pi_{X_2}^{x_1,y_1}=\mathcal{L}(X_2 |  X_1=x_1,Y_1=y_1),$
 $\pi_{Y_2}^{x_1,y_1}=\mathcal{L}(Y_2 |  X_1=x_1,Y_1=y_1),$
 $\pi_X=\mathcal{L}(X_1,X_2),$  $\pi_Y=\mathcal{L}(Y_1,Y_2)$ and so on.

 Let us denote $\Pi(\nu,\mu)$ the set of all $\pi\in\mathrm{P}(\X\times\Y)$ such
 that
 $\pi_X=\nu$ and $\pi_Y=\mu,$  $\Pi_1(\nu_1,\mu_1)$ the set of all $\eta\in\mathrm{P}(\X_1\times\Y_1)$ such that
 $\eta_{X_1}=\nu_1$ and $\eta_{Y_1}=\mu_1$ and
 $\Pi_2(\nu_2,\mu_2)$ the set of all $\eta\in\mathrm{P}(\X_2\times\Y_2)$ such that
 $\eta_{X_2}=\nu_2$ and $\eta_{Y_2}=\mu_2.$

We only consider couplings $\pi$ such that under the law $\pi$
\begin{itemize}
    \item $\mathcal{L}(X_1,X_2)=\nu,$
    \item $\mathcal{L}(Y_1,Y_2)=\mu,$
    \item $Y_1$ and $X_2$ are independent conditionally on $X_1$ and
    \item $X_1$ and $Y_2$ are independent conditionally on $Y_1.$
\end{itemize}
By the definition of the optimal cost, optimizing over this
collection of couplings leads us to
\[
\mathcal{T}_c(\nu,\mu)\leq \inf_{\pi_1,\pi_2^{\diamond}} \int
c_1\oplus
c_2(x_1,y_1,x_2,y_2)\,d\pi_1(x_1,y_1)d\pi_2^{x_1,y_1}(x_2,y_2)
\]
where the infimum is taken over all $\pi_1\in \Pi_1(\nu_1,\mu_1)$
and all Markov kernels $\pi_2^\diamond=(\pi_2^{x_1,y_1};
x_1\in\X_1, y_1\in\Y_1)$ such that $\pi_2^{x_1,y_1}\in
\Pi_2(\nu_{X_2}^{x_1},\mu_{Y_2}^{y_1})$ for $\pi_1$-almost every
$(x_1,y_1).$ As $\mu$ is a tensor product:
$\mu=\mu_1\otimes\mu_2,$ we have $\mu_{Y_2}^{y_1}=\mu_2,$
$\pi_1$-a.e. so that $\pi_2^{x_1,y_1}\in
\Pi_2(\nu_{X_2}^{x_1},\mu_2)$ for $\pi_1$-almost every
$(x_1,y_1).$
\\
We obtain
\begin{align*}
    \mathcal{T}_c(\nu,\mu)
   &\leq \inf_{\pi_1,\pi_2^{\diamond}} \int c_1\oplus c_2(x_1,y_1,x_2,y_2)\,d\pi_1(x_1,y_1)d\pi_2^{x_1,y_1}(x_2,y_2) \\
   &= \inf_{\pi_1} \left[\int_{\X_1\times\Y_1} c_1\,d\pi_1
        + \inf_{\pi_2^{\diamond}} \int_{\X_1\times\Y_1}\int_{\X_2\times\Y_2}c_2(x_2,y_2)d\pi_2^{x_1,y_1}(x_2,y_2)\,d\pi_1(x_1,y_1)\right] \\
     &\stackrel{(a)}{=} \inf_{\pi_1} \left[\int_{\X_1\times\Y_1} c_1\,d\pi_1
        + \int_{\X_1\times\Y_1}\left(\inf_{\pi_2^{\diamond}}\int_{\X_2\times\Y_2}c_2(x_2,y_2)d\pi_2^{x_1,y_1}(x_2,y_2)\right)\,d\pi_1(x_1,y_1)\right] \\
  &\stackrel{(b)}{=} \inf_{\pi_1} \left[\int_{\X_1\times\Y_1} c_1\,d\pi_1
        +
        \int_{\X_1\times\Y_1}\mathcal{T}_{c_2}\big(\nu_{X_2}^{x_1},\mu_2\big)\,d\pi_1(x_1,y_1)\right]\\
  &= \inf_{\pi_1}\left\{\int_{\X_1\times\Y_1} c_1\,d\pi_1\right\}+ \int_{\X_1}\mathcal{T}_{c_2}(\nu_{X_2}^{x_1},\mu_2)\,d\nu_1(x_1) \\
   &= \mathcal{T}_{c_1}(\nu_1,\mu_1)+ \int_{\X_1}\mathcal{T}_{c_2}(\nu_{X_2}^{x_1},\mu_2)\,d\nu_1(x_1)
\end{align*}
which is the desired result.
\\
Equality (a) is not that obvious. First of all, one is allowed to
commute $\inf_{\pi_2^{\diamond}}$ and $\int_{\X_1\times\Y_1}$
since $\pi_2^{\diamond}$ lives in a rich enough family for being
able to optimize separately for each $(x_1,y_1).$ But also, one
must check that after commuting, the integrand
$\inf_{\pi_2^{\diamond}}\int_{\X_2\times\Y_2}c_2(x_2,y_2)d\pi_2^{x_1,y_1}(x_2,y_2)$
is measurable as a function of $(x_1,y_1).$ But for each fixed
$(x_1,y_1),$ this integrand is the optimal transport cost
$\mathcal{T}_{c_2}(\nu_{X_2}^{x_1},\mu_2)$ (this is the content of
equality (b)). Now, with the Kantorovich dual equality
\eqref{eq-01a}, one sees that $\mathcal{T}_c$ is a lower
semicontinuous function as the supremum of a family of continuous
functions. A fortiori, $\mathcal{T}_{c_2}$ is measurable on
$\mathrm{P}(\X_2)\times\mathrm{P}(\Y_2)$ and
$(x_1,y_1)\mapsto\mathcal{T}_{c_2}(\nu_{X_2}^{x_1},\mu_2)$ is also
measurable as a composition of measurable functions (use the
polish assumption for the existence of measurable Markov kernels).
This completes the proof of the proposition.
\end{proof}

Let us have a look at another tensorization result which appears
in \cite{GLWY09}. On the polish product space
$\X^{(n)}:=\prod_{i=1}^n \X_i,$ consider the cost function
\begin{equation*}
\oplus_ic_i (x,y) :=\sum_{i=1}^n c(x_i, y_i), \quad  x,y\in
\X^{(n)}
\end{equation*}
where for  each index $i,$ $c_i$ is \lsc\ on $\X_i^2.$ Let
$(X_i)_{1\le i\le n}$ be the canonical process on
$\X^{(n)}=\prod_{i=1}^n \X_i.$ For each $i,$
$\widetilde{X}_i=(X_j)_{1\le j\le n; j\not=i}$ is the
configuration without its value at index $i.$ Given a probability
measure $\nu$ on $\X^{(n)}$,
$$\nu_i^{\widetilde{x}_i}=\nu(X_i\in\cdot|\widetilde{X}_i=\widetilde{x}_i)$$ denotes the regular conditional
distribution of $X_i$ knowing that
$\widetilde{X}_i=\widetilde{x}_i$ under $\nu$ and
$$
\nu_i=\nu(X_i\in\cdot)
$$
denotes the $i$-th marginal of $\nu.$

\begin{prop}\label{res-tensliming} Let $\mu=\bigotimes_{i=1}^n\mu_i$ be a product probability measure on $\X^{(n)}.$
For all $\nu\in\mathrm{P}(\X^{(n)}),$
    $$
\mathcal{T}_{\oplus c_i}(\nu,\mu) \le
\int_{\X^{(n)}}\left(\sum_{i=1}^n
\mathcal{T}_{c_i}(\nu_i^{\widetilde{x}_i}, \mu_i)\right)\,d\nu(x).
    $$
\end{prop}

\begin{proof} \
$\bullet$ \emph{A first sketch}.\
 Let $(W_i)_{1\le i\le n}=(U_i, V_i)_{1\le i\le n}$
be a sequence of random variables taking their  values in
$\prod_{i=1}^n \X_i^2$ which is defined on some probability space
$(\Omega, \P)$ so that it realizes $\mathcal{T}_{\oplus
c_i}(\nu,\mu).$ This means that the law of $U=(U_i)_{1\le i\le n}$
is $\nu$, the law of $V=(V_i)_{1\le i\le n}$ is $\mu$ and
    $
 \E \sum_{i} c_i(U_i, V_i)= \mathcal{T}_{\oplus c_i}(\nu,\mu).
    $

Let $i$ be a fixed index. There exists a couple of random
variables $\widehat{W}_i:=(\widehat{U}_i, \widehat{V}_i)$ such
that its conditional law given $(\widetilde{U}_i,
\widetilde{V}_i)=\widetilde{W}_i:=(W_j)_{j\ne i}$ is a coupling of
$\nu_i^{\widetilde{U}_i}$ and $\mu_i^{\widetilde{V}_i},$ and
$\P$-a.s.,
    $
    \E [c_i(\widehat{U}_i, \widehat{V}_i)|\widetilde{W}_i]
    =\mathcal{T}_{c_i}\Big(\nu_i^{\widetilde{U}_i},\mu_i^{\widetilde{V}_i}\Big).
    $ This implies
\begin{equation}\label{eqL-04}
     \E c_i(\widehat{U}_i, \widehat{V}_i)
    =\E\mathcal{T}_{c_i}\Big(\nu_i^{\widetilde{U}_i},\mu_i^{\widetilde{V}_i}\Big).
\end{equation}
Clearly, $\left[(\widetilde{U}_i, \widehat{U}_i);(\widetilde{V}_i,
\widehat{V}_i)\right]$ is a coupling of $(\nu,\mu).$ The
optimality of $W$ gives us\\
    $
    \E \sum_{j} c_j(U_j, V_j)
    \le \E \left(\sum_{j\ne i} c_j(U_j, V_j)+ c_i(\widehat{U}_i,\widehat{V}_i)\right)
    $
which boils down to
\begin{equation}\label{eqL-05}
    \E c_i(U_i,V_i)
    \le \E c_i(\widehat{U}_i, \widehat{V}_i)
    = \E\mathcal{T}_{c_i}\Big(\nu_i^{\widetilde{U}_i},\mu_i^{\widetilde{V}_i}\Big)
\end{equation}
where the equality is \eqref{eqL-04}. Summing over all the indices
$i,$ we see that
$$
    \mathcal{T}_{\oplus c_i}(\nu,\mu)
    = \E \sum_{i=1}^n\ c_i(U_i, V_i)
    \le \E \sum_{i=1}^n
    \mathcal{T}_{c_i}\Big(\nu_i^{\widetilde{U}_i},\mu_i^{\widetilde{V}_i}\Big).
$$
As $\mu$ is a product measure, we have
$\mu_i^{\widetilde{V}_i}=\mu_i,$ $\P$-almost surely and we obtain
$$
\mathcal{T}_{\oplus c_i}(\nu,\mu) \le\int_{\X^{(n)}} \sum_{i=1}^n
 \mathcal{T}_{c_i}(\nu_i^{\widetilde{x}_i},\mu_i)\,d\nu(x)
$$
which is the announced result.
\par\medskip\noindent$\bullet$ \emph{Completion of the proof}.\
This first part is an incomplete proof, since one faces a
measurability problem when constructing the conditional optimal
coupling $\widehat{W}_i:=(\widehat{U}_i, \widehat{V}_i).$ This
measurability is needed to take the expectation in \eqref{eqL-04}.
More precisely, it is true that for each value $\widetilde{w}_i$
of $\widetilde{W}_i,$ there exists a coupling
$\widehat{W}_i(\tilde{w}_i)$ of $\nu_i^{\widetilde{u}_i}$ and
$\mu_i^{\widetilde{v}_i}.$ But the dependence in $\widetilde{w}_i$
must be Borel measurable for
$\widehat{W}_i=\widehat{W}_i(\widetilde{W}_i)$ to be a random
variable.
\\
One way to circumvent this problem is to proceed  as in
Proposition \ref{res-tenstrans}. The important features of this
proof are:
\begin{enumerate}
    \item  The
 Markov kernels $\nu_i^{\widetilde{u}_i}$ and
$\mu_i^{\widetilde{v}_i}$ are built with conditional independence
properties as in Proposition \ref{res-tenstrans}'s proof. More
precisely
\begin{itemize}
    \item $\widetilde{V}_i$ and $U_i$ are independent
    conditionally on $\widetilde{U}_i$ and
    \item $\widetilde{U}_i$ and $V_i$ are independent
    conditionally on $\widetilde{V}_i.$
\end{itemize}
These kernels admit measurable versions since the state space is
polish.
    \item The measurability is not required
at the level of the optimal coupling but only through the optimal
cost.
\end{enumerate}
 This leads us to \eqref{eqL-05}. We omit the details of the
proof which is a variation on Proposition \ref{res-tenstrans}'s
proof with another ``nightmare of notation".
\end{proof}

Proposition \ref{res-tensliming} differs from Marton's original
result \cite{M96b} which requires an ordering of the indices.

\section{Variational representations of the relative entropy}

At Section \ref{section transport inequalities}, we  took  great
advantage of the variational representations of $\mathcal{T}_c$
and the relative entropy. Here, we give a proof of the variational
representation formulae \eqref{eqL-c} and \eqref{eqL-d} of the
relative entropy.

\begin{prop}\label{resL-13}
For all $\nu\in\PX,$
\begin{equation}\label{eqAp-c}
\begin{split}
  H(\nu|\mu)&= \sup\left\{\int u\,d\nu-\log\int e^u\,d\mu; u\in\CX\right\}.\\
        &= \sup\left\{\int u\,d\nu-\log\int e^u\,d\mu; u\in \BX\right\}
\end{split}
\end{equation}
and for all $\nu\in\PX$ \textrm{ such that } $\nu\ll\mu,$
\begin{equation}\label{eqAp-d}
    H(\nu|\mu) =\sup\left\{\int u\,d\nu-\log\int e^u\,d\mu; u: \textrm{measurable,} \int
    e^{u}\,d\mu<\infty, \int u_-\,d\nu<\infty
    \right\}
\end{equation}
where $u_-=(-u)\vee 0$ and $\int u\,d\nu\in(-\infty,\infty]$ is
well-defined for all $u$ such that $\int u_-\,d\nu<\infty.$
\end{prop}
\begin{proof} Once we
have \eqref{eqAp-d}, \eqref{eqAp-c} follows by standard
approximation arguments.
\\
The proof of \eqref{eqAp-d} relies on Fenchel inequality for the
convex function $h(t)=t\log t-t+1$: $$st\le (t\log
t-t+1)+(e^s-1)$$ for all $s\in[-\infty,\infty),$ $t\in[0,\infty),$
with the conventions $0\log 0=0,$ $e^{-\infty}=0$ and
$-\infty\times 0=0$ which are legitimated by limiting procedures.
The equality is attained when $t=e^s.$

Taking $s=u(x),$  $t=\frac{d\nu}{d\mu}(x)$ and integrating with
respect to $\mu$ leads us to $$\int u\,d\nu\le
H(\nu|\mu)+\int(e^u-1)\,d\mu,$$ whose terms are meaningful with
values in $(-\infty,\infty],$ provided that $\int
u_-\,d\nu<\infty.$ Formally, the case of equality corresponds to
$\frac{d\nu}{d\mu}=e^u.$ With the monotone convergence theorem,
one sees that it is approached by the sequence $u_n=
\log(\frac{d\nu}{d\mu}\vee e^{-n}),$ as $n$ tends to infinity.
This gives us $H(\nu|\mu)=\sup\left\{\int u\,d\nu-\int
(e^u-1)\,d\mu; u: \int e^{u}\,d\mu<\infty,\inf u>-\infty
\right\},$ which in turn implies that
\begin{equation*}
    H(\nu|\mu)=\sup\left\{\int u\,d\nu-\int (e^u-1)\,d\mu; u: \int
e^{u}\,d\mu<\infty,\int u_-\,d\nu<\infty \right\},
\end{equation*}
since the integral $\int \log(d\nu/d\mu)\,d\nu=\int
h(d\nu/d\mu)\,d\mu\in[0,\infty]$  is well-defined.

Now, we take advantage of the unit mass of $\nu\in\PX:$ $$\int
(u+b)\,d\nu-\int (e^{(u+b)}-1)\,d\mu=\int u\,d\nu-e^b\int
e^u\,d\mu +b+1,\quad  b\in \R,$$ and we use the easy
identity $\log a=\inf_{b\in\R}\{ae^b-b-1\}$ to obtain
    $$\sup_{b\in\R}\left\{\int (u+b)\,d\nu-\int
(e^{(u+b)}-1)\,d\mu\right\}=\int u\,d\nu-\log\int e^u\,d\mu.$$
Whence,
\begin{align*}
  &\sup\left\{\int u\,d\nu-\int (e^u-1)\,d\mu; u: \int e^{u}\,d\mu<\infty,\int u_-\,d\nu<\infty
  \right\}\\
    &= \sup\left\{\int (u+b)\,d\nu-\int (e^{(u+b)}-1)\,d\mu;b\in\R,  u: \int e^{u}\,d\mu<\infty,\int u_-\,d\nu<\infty \right\}\\
  &= \sup\left\{\int u\,d\nu-\log\int e^u\,d\mu; u: \int e^{u}\,d\mu<\infty,\int u_-\,d\nu<\infty
  \right\}.
\end{align*}
This completes the proof of \eqref{eqAp-d}.
\end{proof}

\bibliographystyle{plain}

\end{document}